\documentclass{amsart}

\usepackage{amsmath}
\usepackage{amssymb}
\usepackage{amsthm}
\usepackage{enumerate}
\usepackage{pxfonts}
\usepackage{xcolor}
\usepackage{graphicx}
\usepackage{mathrsfs}
\usepackage{latexsym}
\usepackage{hyperref}
\usepackage{subcaption}
\usepackage[numbers,sort&compress]{natbib}

\newtheorem{theorem}{Theorem}[section]
\newtheorem{proposition}[theorem]{Proposition}

\newtheorem{lemma}[theorem]{Lemma}
\newtheorem{definition}[theorem]{Definition}
\newtheorem{remark}[theorem]{Remark}
\newtheorem{assumption}[theorem]{Assumption}

\newcommand{\Indi}{\mbox{\rm{1}}\hspace{-0.25em}\mbox{\rm{l}}}
\newcommand{\Cov}{\mathrm{Cov}}
\newcommand{\Var}{\mathrm{Var}}

\usepackage{ulem}

\newcommand{\norm}[1]{ {\left\lVert  #1\right\rVert}}

\newcommand{\opnorm}[1]{{\vert\kern-0.25ex\vert\kern-0.25ex\vert #1 
  \vert\kern-0.25ex\vert\kern-0.25ex\vert}}

\author[
G. Szymanski and T. Takabatake]
{Gr{\'e}goire Szymanski \and Tetsuya Takabatake}

\address{Gr\'egoire Szymanski, Ecole Polytechnique, CMAP, route de Saclay, 91128 Palaiseau, France} 
\email{gregoire.szymanski@polytechnique.edu}

\address{Tetsuya Takabatake, Hiroshima University, School of Economics, 2-1 Kagamiyama 1-Chome, Higashi-Hiroshima, Hiroshima, Japan}
\email{tkbtk@hiroshima-u.ac.jp}

\date{\today}

\begin{document}

\title[Asymptotic Efficiency for Fractional Brownian Motion with general noise]{Asymptotic Efficiency for Fractional Brownian Motion with general noise}

\begin{abstract}
We investigate the Local Asymptotic Property for fractional Brownian models based on discrete observations contaminated by a Gaussian moving average process. We consider both situations of low and high-frequency observations in a unified setup and we show that the convergence rate $n^{1/2} (\nu_n \Delta_n^{-H})^{-1/(2H+2K+1)}$ is optimal for estimating the Hurst index $H$, where $\nu_n$ is the noise intensity, $\Delta_n$ is the sampling frequency and $K$ is the moving average order. We also derive asymptotically efficient variances and we build an estimator achieving this convergence rate and variance. This theoretical analysis is backed up by a comprehensive numerical analysis of the estimation procedure that illustrates in particular its effectiveness for finite samples.
\end{abstract}

\maketitle

\noindent \textbf{Mathematics Subject Classification (2020) }: 62F12, 62M09, 62M10, 62M15.

\noindent \textbf{Keywords}: 
{High-frequency data; Fractional Brownian motion; LAN property;
Moving Average processes; Pre-averaging; Iterated estimation procedure}


\section{Introduction}
\label{sec:introduction}

We aim to recover the parameters of a fractional Brownian motion from discrete-time observations contaminated by an additive Gaussian noise. We consider the observations
\begin{equation}
\label{eq:fbm:model}
   X_{i}^{n}=\sigma W^{H}_{i\Delta_n} + \tau \nu_{n} Y_{i},\;\; i=0,1,\cdots,n,
\end{equation}
where $W^H$ is a fractional Brownian motion with Hurst index $H\in(0,1]$, that is the unique self-similar Gaussian process with stationary increments, characterized by its covariance structure
\begin{equation*}
\Cov[W^{H}_{s}, W^{H}_{u}] = \frac{1}{2}(s^{2H} + u^{2H} - |s-u|^{2H}).
\end{equation*}
Additionnally, in \eqref{eq:fbm:model}, the noise $\{Y_{i}\}_{i\in\mathbb{Z}}$ is a Gaussian stationary sequence independent of $W^{H}$ and $\sigma$, $\tau$, $\nu_n$ and $\Delta_n$ are postive parameters. We consider $\sigma$, $H$, and $\tau$ as unknown, while assuming knowledge of the sampling frequency $\Delta_n$ and the noise decay rate $\nu_n$.\\

The main contribution of this paper is to understand how the structure of the noise $Y$ affects the estimation of parameters $\sigma$ and $H$. In particular, we show that non-i.i.d. structure may significantly improve the convergence rate of estimation procedures. We first establish the Local Asymptotic Property (LAN) property for \eqref{eq:fbm:model}. Unlike \cite{gloter2007estimation, Szymanski-2022} where the noise essentially satisfies martingale conditions, we focus here on noises with a more general structure, which is supposed known. Under the assumptions $\nu_n = 1$ and $\Delta_n = n^{-1}$ the convergence rate $n^{-1/(4H+2)}$ is found in \cite{gloter2007estimation, Szymanski-2022}. Instead, we show that when the noise $Y$ is a moving average process of order $K$, this rate becomes 
\begin{equation*}
    (n/\nu_n^2)^{-(2K+1)/(4H+4K+2)}.
\end{equation*}
This naturaly extends the case where the noise is i.i.d. which corresponds to $K=0$. We then build an  asymptotically efficient estimator for all parameters $(H,\sigma,\tau)$. The construction uses a pre-averaging technique similar to \cite{jacod2015microstructure, mykland2016between}. This method is  based on non-uniform weights chosen to optimally denoise the signal. The estimators are then built from the energy levels and we use a non-standard debiaising technique in the spirit of \cite{Chong-Hoffmann-Liu-Rosenbaum-Szymanski-2022-MinMax}. This method produces rate optimal estimators which are not necessarily efficient and we retrieve efficiency through a One-step procedure (see Section 5.7 in \cite{vaart2000asymptotic} for details). This theoretical analysis is backed up by a comprehensive numerical analysis of the estimation procedure that illustrates in particular its effectiveness for finite samples.\\

The estimation of the parameters of a fractional Brownian motion is a long-standing problem, first studied through the spectrum of the self-similar stationary Gaussian time series. The first estimation procedure relied strongly on the knowledge of the precise distribution of the Gaussian series, including maximum likelihood estimators \cite{Yajima-1985, Dahlhaus-1989,Dahlhaus-2006}, Whittle estimators \cite{fox1986large,fox1987central}, local Whittle estimators \cite{Robinson-1995-LWE}, log-periodogram regressions \cite{Robinson-1995-Log-Periodogram}, \textit{etc}. \cite{istas1997quadratic} put forward a new strategy by using instead a more flexible estimator based on quadratic variations across times scales, studied more precisely in the specific case of the fractional Brownian motion by \cite{coeurjolly2001estimating}.\\

These theoretical results recently gained significant attention due to their applications in mathematical finance. In the seminal works \cite{Comte-Renault-1998} and later \cite{gatheral2018volatility}, the authors suggest a fractional volatility model for the price of financial assets, namely
\begin{equation}
\label{eq:price:roughvol}
\begin{cases}
     \mathrm{d}S_{u} = S_{u} ( \mu_{u}\,\mathrm{d}u + \sigma_{u}\,\mathrm{d}B_{u}),
     \\
     \sigma_{u} = \exp \Big(m + \int_{-\infty}^{u} e^{-\alpha(u-s)}\,\mathrm{d}W^{H}_{s}\Big),
\end{cases}
\end{equation}
where $\mu$ and $\sigma$ represent the drift and volatility processes, where $B$ is a Brownian Motion, and where $m$ and $\alpha$ are some constant. 
However, it is important to acknowledge that in \eqref{eq:price:roughvol}, the volatility $\sigma_u$ is not directly observed and needs to be estimated. Therefore, a carefull statistical treatment is needed in to estimate $H$. This question is particularily challenging when $H<1/2$ because in this context, volatility is highly irregular and cannot be estimated from local averages. This question was first raised in \cite{gloter2007estimation} and applied to \eqref{eq:price:roughvol} in \cite{rosenbaum2008estimation}. These works focused on the case $H > 1/2$ to remove the volatility estimation problem.
This limitation was recently removed in \cite{Szymanski-2022,Chong-Hoffmann-Liu-Rosenbaum-Szymanski-2022-MinMax,Chong-Hoffmann-Liu-Rosenbaum-Szymanski-2022-CLT}. Due to the martingale-price/rough-volatility structure of the model described in \eqref{eq:price:roughvol} naturally relates to \eqref{eq:fbm:model} with i.i.d. noise $Y$. The estimators proposed in this setup exhibit a convergence rate $n^{-1/(4H+2)}$ which is optimal whenever $0<H<1$, see \cite{gloter2007estimation, Chong-Hoffmann-Liu-Rosenbaum-Szymanski-2022-MinMax}. This rate, slower than the parametric rate $n^{-1/2}$ is still fast for small $H$ but quickly deteriorates for large $H$. Moreover, their approach does not incorporate microstructure noise: they assume that the price is directly observed, which is the case when considering daily increments. When considering higher-frequency data, this microstructure noise cannot be ignored. Its statistical properties were studied in depth \cite{jacod2017statistical} where it showed that it is far from i.i.d. Other approaches also exist as one could estimate spot volatility from option data \cite{chong2023asymptotic} or even use directly the VIX index \cite{floc2022roughness}. In those cases, the noise structure could differ, resulting in different estimators. It is therefore crucial to understand the impact of the noise on the estimation procedure of $H$ to identify the difference between each of these estimates.\\

In this work, we focus on the additive noise model \eqref{eq:fbm:model}. Although this parametric problem is by nature a simpler version of \eqref{eq:price:roughvol} where the noise is multiplicative, it provides a first comprehensive analysis of the effects of the regularity of the microstructure noise. It also highlights that different denoising techniques must be used to dampen the effect of the noise depending on its regularity. The huge advantage of using a parametric model is that it allows an in-depth analysis of the asymptotic efficient rates through the LAN property. LAN property was fist introduced by \cite{LeCam-1960} and was further developed by \cite{Sweeting-1980, ibragimov1981statistical, LeCam-1986-book, Yoshida-2011-PLDI}. It is used to provide asymptotic lower bounds for the convergence rate estimators, see {\it e.g.} \cite{Hajek-1972, LeCam-1972, ibragimov1981statistical}, and also helps to identify the optimal limiting variances through the limiting Fisher information matrix. However, this theory faces challenges when applied in the context of fractional Gaussian sequences under high-frequency observations. Indeed, the estimation errors of $H$ and $\sigma$ interplay significantly in the high-frequency asymptotics, {\it i.e.} $\Delta_{n}\to 0$ as $n\to\infty$, resulting in a degenerate Fisher information matrix when the usual "diagonal" rate matrices are used to rescale the estimation errors, see \cite{kawai2013fisher}. Due to this singularity, we cannot apply the existing minimax theorems and new techniques need to be developed. Without noise, this limitation has recently been broken in \cite{brouste2018lan} where the authors identified a class of non-diagonal rate matrices resulting in non-degenerate Fisher information matrices enabling a precise study of the efficiency of estimating $H$ and $\sigma$. In this paper, we identify new rate matrices and extend their methodology in the presence of noise.
\\

The remainder of this paper is structured as follows. Section~\ref{sec:notation} introduces the statistical model we work in and the main assumptions needed for our results. Then Section~\ref{Sec:LAN} presents the LAN property and the minimax rates derived from it, while Section~\ref{Sec:estimation} presents the estimation procedure. We provide an efficient implementation of theses estimators in Section~\ref{sec:numerical}. Eventually, all mathematical proofs and derivations can be found in the appendices.

\section{Notation and Statistical Model}
\label{Sec:stat_model}

\subsection{Statistical Model}
Let $\Theta = \Theta_{H}\times\Theta_{\sigma}\times\Theta_{\tau} = [H_{-}, H_{+}] \times [\sigma_{-}, \sigma_{+}] \times [\tau_{-}, \tau_{+}]  \subset (0,1)\times(0,\infty)^{2}$ with a non-empty interior. We write $\theta = (H,\sigma,\tau)$ for the elements of $\Theta$ and we may write $\xi=(H,\sigma)$ for conciseness when needed.
For each $n$, we consider a probability space $( \Omega_n, \mathcal{A}_n)$ on which are defined a family of random variable 
$\mathbf{X}_n = (X_{0}^{n},X_{1}^{n},\cdots,X_{n}^{n})$ and a family of probability measures $\mathbb{P}^{n}_{\theta}$ such that $\mathbf{X}_n$ can be realised under $\mathbb{P}^{n}_{\theta}$ by
\begin{equation*}
       X_{i}^{n}=\sigma W^{H}_{i\Delta_n} + \tau \nu_{n} Y_{i},\;\; i=0,1,\cdots,n,
\end{equation*}
where $\Delta_{n}>0$ is a sampling frequency, $\nu_{n}>0$ a known noise decay rate, $W^H$ a fractional Brownian motion with Hurst index $H$ and $Y=\{Y_{i}\}_{i\in\mathbb{Z}}$ a centered Gaussian sequence independant of $W^H$. We write $s$ its spectral density.
The behaviour near $0$ of the power spectral density characterises the asymptotic decay of the autocorrelation of $Y$. Suppose that there exists a non-negative integer $K$ such that
\begin{equation}
\label{eq:g:K}
    s(\lambda) \sim c_s |\lambda|^{2K}\ \ \mbox{as $|\lambda|\to 0$}
\end{equation}
for some $c_s > 0$. By periodicity of $s$, we then know that there exists a continuous function $b$ such that
\begin{equation*}
    s(\lambda) = b(\lambda) (2(1-\cos(\lambda)))^{K}.
\end{equation*}
In that case, $K$ can also be viewed as the moving average order of $Y$. Indeed, since $Y$ is Gaussian, we have
\begin{equation}
\label{eq:constructionY}
    Y_{j} = \sum_{l=0}^{K} {K \choose l}(-1)^{l}\varepsilon_{j-l}
\end{equation}
in distribution, where $\{\varepsilon_{j}\}_{j\in\mathbb{Z}}$ is an stationary Gaussian sequence independent of $W^{H}$ with spectral density $b$. In this paper, we focus on the following case.

\begin{assumption}
\label{Assumption_Y}
$K$ is a non-negative integer and we have $s(\lambda):=(2(1-\cos(\lambda)))^{K}$.
\end{assumption} 

A careful examination of our proofs ensures that our results can be proved under more general assumptions on $b$. Indeed, the LAN property proved in Section~\ref{subsec:LAN} extends readily and the presence of $b$ only changes the asymptotic variance. For the construction of an estimator, the presence of $b$ creates an additional problems in the computation of the variance of a local average of the noise $\gamma^{(w)}_p$ appearing in Lemma~\ref{lem:noise:structure}. This is easy to compute when $b(\lambda)=1$, see Section~\ref{Sec:lem:noise:structure:proof}. Although it is reasonable to think that the properties of $\gamma^{(w)}_p$ extends to the case of a general function $b$, we do not proceed to further investigation in this paper. \\

The assumption that $K$ is a non-negative integer is more crucial. Indeed, if \eqref{eq:g:K} holds for some non integer $K$, then we could rewrite $K = K^{\prime} + 1/2 - H^{\prime}$ for some $K^{\prime}\geq 0$ and $0 <  H^{\prime} < 1$. In that case, we still have \eqref{eq:constructionY} using this time a fractional stationary process $\varepsilon$ and the LAN property still holds, although new methods are needed to disentangle the two fractional processes in the estimation procedure. The case $K\in (-3/2,- 1/2)$ is linked to the model $X_i^n = \sigma W^H_{i\Delta_n} + \sigma^{\prime} W^{H^{\prime}}_{i\Delta_n}$ with $H^\prime = - 1/2 - K$ and the particular case $H^{\prime}=\frac{1}{2}$ (that is $K=-1$) corresponds to the mixed fractional Brownian model. This case exhibits new asymptotics and is left for future investigation.\\

Moreover, we also write 
$Z_{i}^{n}=\sigma F_{i}^{n}+\tau G_{i}^{n}$ where 
\begin{equation}
    \label{eq:def_FG}
    F_{i}^{n}:=W_{i\Delta_{n}}^{H}-W_{(i-1)\Delta_{n}}^{H} \;\;\text{ and }\;\;
    G_{i}^{n}:=\nu_{n}(Y_{i}-Y_{i-1})
\end{equation}
These sequences are stationary and are therefore fully characterized by their spectral density. The  spectral density of $\sigma F^{n}$ is denoted $f_{\xi}^{n}$ and is given by
\begin{equation*}
	f_{\xi}^{n}(\lambda):=
	\sigma^{2}\Delta_{n}^{2H}f_{H}(\lambda)
\;\;\;\;\text{ where }\;\;\;\;
	f_{H}(\lambda):=
	c_{H}\{2(1-\cos(\lambda))\}\sum_{\tau\in\mathbb{Z}}|\lambda+2\pi\tau|^{-1-2H}
\end{equation*}
with $c_{H}:=(2\pi)^{-1}\Gamma(2H+1)\sin(\pi H)$, {\it e.g.} see \cite{Samorodnitsky-Taqqu-1994}.
The spectral density of $\tau G^{n}$ is denoted $g_{\tau}^{n}$ and is given by Assumption~\ref{Assumption_Y} by
 \begin{equation*}
	g_{\tau}^{n}(\lambda):=\tau^{2}\nu_{n}^{2}g(\lambda) \;\;\text{ where }\;\; g(\lambda):=(2(1-\cos(\lambda)))^{K+1}
\end{equation*}
Since $F^{n}$ and $G^{n}$ are independent, it follows that the spectral density of $\{Z_{i}^{n}\}_{i\in\mathbb{Z}}$ is given by
\begin{equation*}
	h_{\theta}^{n}(\lambda):=f_{\xi}^{n}(\lambda)+g_{\tau}^{n}(\lambda).
\end{equation*}
We write $f_{\xi}(\lambda):=\sigma^{2}f_{H}(\lambda)$ and $g_{\tau}(\lambda):=\tau^{2}g(\lambda)$ for conciseness.

\subsection{Notation}
\label{sec:notation}
Throughout this paper, we use the following notation:
\begin{itemize}
	\item 
    We write $\partial_{1} := \partial_{H}:=\partial/\partial{H}$, $\partial_{2} := \partial_{\sigma}:=\partial/\partial{\sigma}$, $ \partial_{3} := \partial_{\tau}=\partial/\partial{\tau}$ and $\partial_{i_{1},\cdots,i_{m+1}}^{m+1}:=\partial_{i_{1},\cdots,i_{m}}^{m}\circ\partial_{i_{m+1}}$ for all $i_{1},\cdots,i_{m+1}\in\{1,2,3\}$ and $m\in\mathbb{N}$. 
    
    \item
    For an arbitrary matrix $M$, we write $\|M\|_{\mathrm{F}}:=(\mathrm{Tr}[M^{\ast}M])^{\frac{1}{2}}$ the Frobenius norm of $M$ and $\|M\|_{\mathrm{op}}:=\sup_{\|x\|=1}\|Mx\|_{\mathrm{F}}$ for its operator norm.
    
    \item 
    For $m\in\mathbb{N}$, $I_{m}$ denotes the $(m\times m)$-identity matrix. 
    \item For non-negative sequences $\{a_{n}\}_{n\in\mathbb{N}}$ and $\{b_{n}\}_{n\in\mathbb{N}}$, we write $a_{n}\lesssim b_{n}$ if there exists an absolute constant $C>0$ such that $a_{n}\leq Cb_{n}$ for sufficiently large $n$. 

	\item 
    For $d\in\mathbb{N}$, 
    $B(\theta,\epsilon)$ (resp. $\overline{B}(\theta,\epsilon)$) denotes an open (resp. a closed) ball of center $\theta$ and radius $\epsilon>0$ in $\mathbb{R}^{d}$.
    
    \item We write $ \mathbb{N} = \{ 1, 2, \dots \}$ and $ \mathbb{N}_0 = \{0, 1, 2, \dots \}$.
      
\end{itemize}

\section{Local Asymptotic Normality}
\label{Sec:LAN}

\subsection{Choice of Rate Matrices}
\label{Sec:RateMatrix}

In this section, we prove that the LAN property holds for the increments $\{X^n_i - X^n_{i-1}\}_{i}$. This sequence, which we denote $\{Z_{i}^{n}\}_{i}$, is a stationary sequence and the likelihood of $\mathbf{Z}_{n}:=(Z_{1}^{n},\cdots,Z_{n}^{n})$ is linked to the one of $\mathbf{X}_{n}$, {\it e.g.} see Proposition~4.1 of \cite{Cohen-Gamboa-Lacaux-Loubes-2013}. Therefore, with a slight abuse of notation, we still write $\mathbb{P}^{n}_{\theta}$ the distribution of $\mathbf{Z}_{n}$ under $\mathbb{P}^{n}_{\theta}$, which can be seen in this context as a probability measure on $(\mathbb{R}^n, \mathcal{B}(\mathbb{R}^n))$.
We first recall the classical definition of the LAN property, {\it e.g.} see \cite{ibragimov1981statistical}.
\begin{definition}
\label{Def:decaying:LAN}
Let $\Theta$ be a convex set of $\mathbb{R}^{d}$ and $\{(\mathcal{X}_{n},\mathcal{A}_{n},\{\mathbb{P}^{n}_{\theta}\}_{\theta\in\Theta})\}_{n\in\mathbb{N}}$ be a sequence of statistical experiments. 
The (sequence of) family of distributions $\{\{\mathbb{P}^{n}_{\theta}\}_{\theta\in\Theta}\}_{n\in\mathbb{N}}$ satisfies the LAN (Local Asymptotic Normality) property at an interior point $\theta$ of $\Theta$ if there exist some non-degenerate $(d\times d)$-matrices $\varphi_{n}(\theta)$ and $\mathcal{I}(\theta)$ such that for any $u \in \mathbb{R}^d$ such that $\theta + \varphi_{n}(\theta) u \in \Theta$, the likelihood ratio has the representation
\begin{equation*}
\frac{\mathrm{d}\mathbb{P}^{n}_{\theta+\Phi_{n}(\theta)u}}{\mathrm{d}\mathbb{P}^{n}_{\theta}}
=
\exp \left( u^{\top}\zeta_{n}(\theta) - \frac{1}{2}u^{\top}\mathcal{I}(\theta) u + v_{n}(u,\theta)
\right)
\end{equation*}
where $\zeta_{n}(\theta)$ is a $d$-dimensional random vector that converges in $\mathbb{P}^{n}_{\theta}$-distribution toward a normal distribution $\mathcal{N}(0,\mathcal{I}(\theta))$ and $v_{n}(u,\theta)$ converges to $0$ in $\mathbb{P}^{n}_{\theta}$-probability as $n\to\infty$.
\end{definition}

In our setting, the main driver in the convergence rate lies in the relationship between $H$ and $K$. Intuitively, the rougher the trajectory is, the easier it is to see it behind the noise. Similarly, when the noise is a moving average of higer order, it is more regular and it is therefore easier to denoise. We introduce the key values
\begin{equation*}
	\nu(H):=\lim_{n\to\infty}\nu_{n}\Delta_{n}^{-H}\in[0,\infty]\;\;\;\mbox{and}\;\;\; \Diamond(H):=2(K+1)-(1-2H) = 2K + 2H + 1 >1.
\end{equation*}
We draw attention to the following relationships involving $\nu(H)$ and $\Diamond(H)$:\begin{itemize}
	\item Note that 
 \begin{equation*}
     \nu_{n}\Delta_{n}^{-H} = \sqrt{\Var[G_{i}^{n}]/\Var[F_{i}^{n}]}
 \end{equation*} where $\{G_{i}^{n}\}_{i\in\mathbb{Z}}$ and $\{F_{i}^{n}\}_{i\in\mathbb{Z}}$ are defined in Equation \eqref{eq:def_FG}. 
	Therefore, $\nu(H)=\infty$ (or, conversely, $\nu(H)=0$) means that the magnitude of the additive noise becomes asymptotically dominant (or negligible) in comparison to that of the fractional Brownian motion. Throughout this paper, we always assume $\nu(H)=\infty$ because the LAN property when $\nu(H)\in[0,\infty)$ can be proved by extending the proof of \cite{Szymanski-2022}, where the case $\nu(H) = 0$ and $K=0$ is studied.
	\item On the other hand, $\Diamond(H)$ is the exponent corresponding to the asymptotic behaviour of a ratio of the normalized spectral density function of $\{G_{i}^{n}\}_{i\in\mathbb{Z}}$ to that of $\{F_{i}^{n}\}_{i\in\mathbb{Z}}$ around the origin
	\begin{equation*}
		g(\lambda)/f_{H}(\lambda)\sim c_{H}^{-1}|\lambda|^{\Diamond(H)}\ \ \mbox{as $|\lambda|\to 0$.}
	\end{equation*}
	This exponent is closely related to the condition on the integrability of integrands and appear in the integrals in the Fisher information matrix $\mathcal{I}(\theta)$, see Section~\ref{subsec:LAN}.
\end{itemize}

Naturally, the convergence rates depend on these quantities. We define 
\begin{equation*}
		r_{n}^{1}(H):=n(\nu_{n}\Delta_{n}^{-H})^{-\frac{2}{\Diamond(H)}},\ \ r_{n}^{2}(H):=n.
\end{equation*}
As we will see later, $r_{n}^{1}(H)$ is linked to the convergence rate of $H$ and $\sigma$ while $r_{n}^{2}(H)$ is related to the convergence rate of $\tau$. For technical purposes, we assume the following.
\begin{assumption}
\mbox{}
\label{Assumption:asymptotics_lan}
\begin{enumerate}[$(1)$]
	\item $\inf_{H\in\Theta_{H}}\nu_{n}\Delta_{n}^{-H} \to \infty$ as $n\to\infty$,
	\item there exists some $\epsilon\equiv\epsilon_{\Theta_{H}}\in(0,1)$ such that $\sup_{H\in\Theta_{H}}n^{\epsilon}r_{n}^{1}(H)^{-1}=o(1)$ as $n\to\infty$.
    \item There exists some constants $0 \leq \varrho_{d} \leq 1$ and $\varrho_{\nu},c_{\nu}\geq 0$ and $c_{d}>0$ such that
\begin{equation*}
\Delta_{n}=c_{d}n^{-\varrho_{d}} \;\;\;\;\text{ and }\;\;\;\;
\nu_{n}=c_{\nu}n^{-\varrho_{\nu}}.
\end{equation*}
\end{enumerate}
\end{assumption}

As mentionned previously, the first condition means that we are in a large noise scheme. The breaking point between the large noise and the small noise asymptotics corresponds to $\nu_n = \Delta_n^{H}$ which is natural, since $\Delta_n^H$ is the typical size of the increments of $W_H$ on a timespan $\Delta_n$. The second ensures that the convergence rate of $H$ is non-degenerated.\\

The crux to prove an effective LAN property for fractional Brownian motion under high-frequency observations is to define a well-chosen class of non-diagonal rate matrices. Indeed, if we choose the classic diagonal rate matrix, we would only be able to prove a weak LAN property as in \cite{kawai2013fisher} with a degenerate Fisher information matrix. Here, we introduce the non-diagonal rate matrices needed to derive the non-degenerate Fisher information matrix given in Section~\ref{subsec:LAN}.  
\begin{assumption}\label{Assump:RateMat}
	Working under Assumption~\ref{Assumption:asymptotics_lan},  we define
	\begin{equation}\label{def_an(H)}
		a_{n}(H):=\log\Delta_{n}-\log(r_{n}^{1}(H)/n).
	\end{equation}
    We consider a sequence of $(3\times 3)$-matrix valued continuous functions $\{\varphi_{n}(\theta)\}_{n\in\mathbb{N}}$ defined on $\Theta$ by
    \begin{align*}
    	&\Phi_{n}(\theta):=
    	\mathrm{diag}\left(r_{n}^{1}(H),r_{n}^{1}(H),r_{n}^{2}(H)\right)^{-\frac{1}{2}}
    	\begin{pmatrix}
    	\varphi_{n}(\theta)&0_{2\times 1} \\
    	0_{1\times 2}&1
    	\end{pmatrix},\;\; \text{ where }\;\;
        \varphi_{n}(\theta):=
    	\begin{pmatrix}
    	\varphi^{n}_{11}(\theta)&\varphi^{n}_{12}(\theta) \\
    	\varphi^{n}_{21}(\theta)&\varphi^{n}_{22}(\theta)
    	\end{pmatrix}
    \end{align*}
    for each $n\in\mathbb{N}$, satisfying the following properties:
    \begin{itemize}
        \item For each $j\in\{1,2\}$, there exist continuous functions $\overline{\varphi}_{1j}(\theta)$ and $\overline{s}_{2j}(\theta)$ on $\Theta$ such that 
            \begin{equation*}
                \varphi^{n}_{1j}(\theta)\to\overline{\varphi}_{1j}(\theta),\ \ 
                2\left(\varphi^{n}_{1j}(\theta)a_{n}(H)+\sigma^{-1}\varphi^{n}_{2j}(\theta)\right)\to \overline{s}_{2j}(\theta)\ \ \mbox{as $n\to\infty$}
            \end{equation*}
            uniformly on $\Theta$.
            \item
            $
                \inf_{\theta\in\Theta}|\overline{\varphi}_{11}(\theta)\overline{s}_{22}(\theta)-\overline{\varphi}_{12}(\theta)\overline{s}_{21}(\theta)|>0$.
	  \end{itemize}
\end{assumption}

This assumption is strongly related to the rate matrix structure introduced in Theorem 3.1 in \cite{brouste2018lan}. The novelty here is the presence of the additional term $-\log(r_{n}^{1}(H)/n)$ in the definition of $a_n(H)$. This term relates to the strength of the noise and governs the convergence rate of the LAN property. We refer to Section~\ref{Sec:Lower-Bounds} for several examples of sequences of rate matrices satisfying the conditions of Assumption~\ref{Assump:RateMat}. Using the LAN property and the H\'ajek–Le Cam minimax theorem, they provide asymptotic optimal convergence rates and variances of estimating the parameters $\theta=(H,\sigma,\tau)$. We also refer to Appendix~\ref{Appendix:Rate-Matrix} for properties satisfied by rate matrices under Assumption~\ref{Assump:RateMat}.

\subsection{Local Asymptotic Normality Property}
\label{subsec:LAN}

We write $\ell_{n}(\theta)$ for the Gaussian log-likelihood function of the difference vector $\mathbf{Z}_{n}$, which is given by
\begin{equation*}
	\ell_{n}(\theta)=-\frac{n}{2}\log(2\pi)-\frac{1}{2}\log\mathrm{det}[\Sigma_{n}(h_{\theta}^{n})]-\frac{1}{2}\mathbf{Z}_{n}^{\top}\Sigma_{n}(h_{\theta}^{n})^{-1}\mathbf{Z}_{n}.
\end{equation*}
From explicit computations, we see that the score function $\partial_{\theta}\ell_{n}(\theta)=(\partial_{H}\ell_{n}(\theta),\partial_{\sigma}\ell_{n}(\theta),\partial_{\tau}\ell_{n}(\theta))$ is given by
\begin{align*}
\begin{cases}
    \partial_{H}\ell_{n}(\theta)
    =-\frac{1}{2}\mathrm{Tr}[\Sigma_{n}(h_{\theta}^{n})^{-1}\Sigma_{n}(\partial_{H}f_{\xi}^{n})]
	+\frac{1}{2}\mathbf{Z}_{n}^{\top}\Sigma_{n}(h_{\theta}^{n})^{-1}\Sigma_{n}(\partial_{H}f_{\xi}^{n})\Sigma_{n}(h_{\theta}^{n})^{-1}\mathbf{Z}_{n}, \\
    \partial_{\sigma}\ell_{n}(\theta)
    =-\frac{1}{2}\mathrm{Tr}[\Sigma_{n}(h_{\theta}^{n})^{-1}\Sigma_{n}(\partial_{\sigma}f_{\xi}^{n})]
	+\frac{1}{2}\mathbf{Z}_{n}^{\top}\Sigma_{n}(h_{\theta}^{n})^{-1}\Sigma_{n}(\partial_{\sigma}f_{\xi}^{n})\Sigma_{n}(h_{\theta}^{n})^{-1}\mathbf{Z}_{n}, \\
    \partial_{\tau}\ell_{n}(\theta)
    =-\frac{1}{2}\mathrm{Tr}[\Sigma_{n}(h_{\theta}^{n})^{-1}\Sigma_{n}(\partial_{\tau}g_{\tau}^{n})]
	+\frac{1}{2}\mathbf{Z}_{n}^{\top}\Sigma_{n}(h_{\theta}^{n})^{-1}\Sigma_{n}(\partial_{\tau}g_{\tau}^{n})\Sigma_{n}(h_{\theta}^{n})^{-1}\mathbf{Z}_{n}
\end{cases}
\end{align*}
where we have
\begin{align*}\begin{cases}
\partial_{\sigma}f_{\xi}^{n}
= 2\sigma\Delta_{n}^{2H}f_{H},
\\
\partial_{\tau}g_{\tau}^{n}=2\tau\nu_{n}^{2}g,\\
\partial_{H}f_{\xi}^{n}=
    (2\log\Delta_{n})f_{\xi}^{n} + \sigma^{2}\Delta_{n}^{2H}\partial_{H}f_{H}
    =(2\log\Delta_{n}+\partial_{H}\log{f_{H}})f_{\xi}^{n}.
\end{cases}
\end{align*}
Moreover, we define the normalized score function $\zeta_{n}(\theta)$ and the observed Fisher information matrix $\mathcal{I}_{n}(\theta)$ by
\begin{equation*}
	\zeta_{n}(\theta):=
	\Phi_{n}(\theta)^{\top}\partial_{\theta}\ell_{n}(\theta)\;\;\;\text{ and }\;\;\;
	\mathcal{I}_{n}(\theta):=
	-\Phi_{n}(\theta)^{\top}\partial_{\theta}^{2}\ell_{n}(\theta)\Phi_{n}(\theta).
\end{equation*}
Let $c(\theta):=(\sigma^{2}\tau^{-2}c_{H})^{\frac{1}{\Diamond(H)}}$ and $d(\theta):=\partial_{H}\log{c_{H}}-2\log{c(\theta)}$. We also define the asymptotic Fisher information matrix $\mathcal{I}(\theta)$ by
\begin{equation*}
	\mathcal{I}(\theta):=
	\begin{pmatrix}
		\overline{\varphi}(\theta)& 0_{2\times 1} \\
		0_{1\times 2}& 2\tau^{-1}
	\end{pmatrix}^{\top}
	\mathcal{F}(\theta)
	\begin{pmatrix}
		\overline{\varphi}(\theta)& 0_{2\times 1} \\
		0_{1\times 2}& 2\tau^{-1}
	\end{pmatrix},
\end{equation*}
where  $\overline{\varphi}(\theta):=(\overline{\varphi}_{ij}(\theta))_{i,j=1,2}$ with $\overline{\varphi}_{2j}(\theta):=d(\theta)\overline{\varphi}_{1j}(\theta)+\overline{s}_{2j}(\theta)$ for $j\in\{1,2\}$, see Assumption~\ref{Assump:RateMat} for the definitions of $\overline{\varphi}_{1j}(\theta)$ and $\overline{s}_{2j}(\theta)$, 
and $\mathcal{F}(\theta)=(\mathcal{F}_{ij}(\theta))_{i,j=1,2,3}$ is defined by	
	\begin{equation*}
		\mathcal{F}_{ij}(\theta) :=
		\frac{c(\theta)}{2\pi}\int_{0}^{\infty}\frac{\left(-2\log\mu\right)^{4-(i+j)}}{\left(1+\mu^{\Diamond(H)}\right)^{2}}\,\mathrm{d}\mu
	\end{equation*}
if $i,j \in \{1, 2\}$, $\mathcal{F}_{33}(\theta):=\frac{1}{2}$ and $\mathcal{F}_{k3}(\theta):=\mathcal{F}_{3k}(\theta):=0$ for $k \in \{1, 2\}$. 
We now state the main result of this section.
\begin{theorem}\label{Thm:LAN}
	Under Assumption~\ref{Assump:RateMat}, the family of distributions $\{\mathbb{P}^{n}_{\theta}\}_{\theta\in\Theta}$ 
 satisfies the LAN property at each interior point $\theta\in\Theta$, that is, 
	\begin{equation*}
		\log\frac{\mathrm{d}\mathbb{P}^{n}_{\theta+\Phi_{n}(\theta)u}}{\mathrm{d}\mathbb{P}^{n}_{\theta}}
		=u^{\top} \zeta_{n}(\theta)-\frac{1}{2}u^{\top}\mathcal{I}(\theta)u+o_{\mathbb{P}^{n}_{\theta}}(1)\ \ \mbox{as $n\to\infty$,}
	\end{equation*}
and the normalized score function $\zeta_{n}(\theta)$ satisfies
	\begin{equation}\label{LAN:Score-CLT}
		\mathcal{L}\{\zeta_{n}(\theta)|\mathbb{P}^{n}_{\theta}\}\rightarrow\mathcal{N}(0,\mathcal{I}(\theta))\ \ \mbox{as $n\to\infty$.}
	\end{equation}
\end{theorem}
The proof of Theorem~\ref{Thm:LAN} is given in Section~\ref{Sec:Proof-LAN}. 
In the rest of this section, we give several applications of the LAN property given in Theorem~\ref{Thm:LAN}.

\subsection{Lower Bounds for estimating $\theta$}
\label{Sec:Lower-Bounds}

In this section, we utilize the LAN property to derive asymptotically efficient rates and variances of estimation for $H$, $\sigma$ and $\tau$ respectively. These bounds are based on the H\'ajek's convolution theorem (see
\cite{Hajek-1972, ibragimov1981statistical,LeCam-1972}), recalled below.
\begin{theorem}[Equation~II.12.1 in \cite{ibragimov1981statistical}]
\label{Thm:Hajek-LeCam}
Suppose that the sequence of the famility of distributions $\{\mathbb{P}^{n}_{\theta}\}_{\theta\in\Theta}$ satisfies the LAN property at the point $\theta$ in the interior of $\Theta \subset \mathbb R^{d}$. Then for any sequence of estimators $\widehat{\theta}_n$ and any symmetric nonnegative quasi-convex function $L$ such that $e^{-\varepsilon |z|^2} L(z) \to 0$ as $|z| \to \infty$ and any $\varepsilon > 0$, we have
\begin{equation*}
    \liminf_{c\to\infty}
    \liminf_{n \to \infty}
    \sup_{\substack{\theta^{\prime}\in\Theta\\\left\|\Phi_{n}(\theta)^{-1}(\theta^{\prime}-\theta)\right\|_{\mathbb{R}^{3}}\leq c}}
    \mathbb{E}_{\theta^{\prime}}^{n}\left[
    L\left(\Phi_{n}(\theta)^{-1}(\widehat{\theta}_{n}-\theta)\right)
    \right]
    \geq 
    \frac{1}{(2\pi)^{d/2}}
    \int_{\mathbb{R}^d} L\big(\mathcal{I}(\theta)^{-1/2}z\big) e^{-\frac{|z|^2}{2}} \,\mathrm{d}z.
\end{equation*}
\end{theorem}

\subsubsection{Lower Bounds for estimating $H$ and $\tau$}\label{Sec:Lower-Bounds-H}
In this section, we consider
\begin{equation*}
\varphi_{n}(\theta):=r_{n}^{1}(H)^{-\frac{1}{2}}
\begin{pmatrix}
1&0 \\
-\sigma a_{n}(H)&1
\end{pmatrix},\ \ n\in\mathbb{N},
\end{equation*}
where $a_{n}(H)$ is defined in Assumption~\ref{Assump:RateMat}. 
Note that $\{\Phi_{n}(\theta)\}_{n\in\mathbb{N}}$ satisfies Assumption~\ref{Assump:RateMat}. In that case, $\varphi_{n}(\theta)$ is non-diagonal and depends on both parameters $H$ and $\sigma$ and we have
\begin{equation*}
	\overline{\varphi}_{11}(\theta)=1,\ \  
	\overline{\varphi}_{12}(\theta)=0,\ \ 
	\overline{s}_{21}(\theta)=0,\ \ 
	\overline{s}_{22}(\theta)=2\sigma^{-1},
\end{equation*}
and
\begin{equation*}
	\overline{\varphi}(\theta)=
	\begin{pmatrix}
		1&0\\
		d(\theta)&2\sigma^{-1}
	\end{pmatrix}.
\end{equation*}
Since we have
\begin{equation*}
	\varphi_{n}(\theta)^{-1}=\sqrt{r_{n}^{1}(H)}
	\begin{pmatrix}
	 1&0 \\
	 \sigma a_{n}(H)&1
	\end{pmatrix},
\end{equation*}
Theorems~\ref{Thm:LAN} and~\ref{Thm:Hajek-LeCam} and the computations of the inverse of the Fisher information matrix in Appendix \ref{section:explicit_inverse} give the asymptotic lower bounds
\begin{align}\label{LB-H-tau}
\begin{cases}
	&\liminf_{n\to\infty}
	r_{n}^{1}(H)\mathbb{E}_{\theta}^{n}[(\widehat{H}_{n}-H)^{2}]\geq 
	\mathcal{F}_{22}(\theta)(\mathcal{F}_{11}(\theta)\mathcal{F}_{22}(\theta)-\mathcal{F}_{12}(\theta)^{2})^{-1},\\
	&\liminf_{n\to\infty}r_{n}^{2}(H)
	\mathbb{E}_{\theta}^{n}[(\widehat{\tau}_{n}-\tau)^{2}]\geq\frac{\tau^{2}}{2}
\end{cases}
\end{align}
for any sequences of estimators $\{\widehat{H}_{n}\}_{n\in\mathbb{N}}$ and $\{\widehat{\tau }_{n}\}_{n\in\mathbb{N}}$ by taking $\ell(x,y,z)=x^{2}$ and $\ell(x,y,z)=z^{2}$ in Theorem~\ref{Thm:Hajek-LeCam}. 
In particular, the asymptotically efficient rates of estimation for $H$ and $\tau$ are respectively given by $r_{n}^{1}(H)^{\frac{1}{2}}$ and $r_{n}^{2}(H)^{\frac{1}{2}}$. \\

In \eqref{LB-H-tau}, the optimal convergence rate is derived under the assumption \ref{Assumption:asymptotics_lan} which implies that we are in a large noise asymptotics, i.e. $\nu_n \Delta_n^{-H} \to \infty$. In general, extending the approach of \cite{Szymanski-2022} to study the case where $\nu_n$ is small, we get the following asymptotic lower rate of convergence for the estimation of $H$
\begin{equation*}
    n^{1/2} (\nu_n \Delta_n^{-H})^{-1/(2H+2K+1)} \wedge n^{1/2}.
\end{equation*}
Notably, when $K=0$, this expression aligns with the standard convergence rate for estimating the Hurst index from noisy observations, as discussed in \cite{gloter2007estimation, rosenbaum2008estimation, Szymanski-2022, Chong-Hoffmann-Liu-Rosenbaum-Szymanski-2022-MinMax}. An interesting observation is that we get a better convergence rate when $K$ increases. In the limit $K \to \infty$, the rate converges to the parametric rate $n^{1/2}$. This behaviour is a natural progression, as large $K$ correspond to smoother noises that should, in principle, have a lesser impact on the estimation of $H$.

\subsubsection{Lower Bound for estimating $\sigma$}
\label{Sec:Optimal-sigma}
Throughout this section, we always assume $\inf_{n\in\mathbb{N}}|a_{n}(H)|>0$. In that case, we consider the sequence of matrices $\{ \Phi_{n}(\theta)\}_{n\in\mathbb{N}}$ given by
\begin{align*}
	\varphi_{n}(\theta):=r_{n}^{1}(H)^{-\frac{1}{2}}
	\begin{pmatrix}
	a_{n}(H)^{-1}&1\\
	0&-\sigma a_{n}(H)
	\end{pmatrix},\ \ n\in\mathbb{N}. 
\end{align*}
This sequence satifyies Assumption \ref{Assump:RateMat} and $\varphi_{n}(\theta)$ is non-diagonal and depends on both parameters $H$ and $\sigma$. 
Moreover,
\begin{equation*}
	\overline{\varphi}_{11}(\theta)=0,\ \  
	\overline{\varphi}_{12}(\theta)=1,\ \ 
	\overline{s}_{21}(\theta)=2,\ \ 
	\overline{s}_{22}(\theta)=0,
\end{equation*}
so that we obtain
\begin{equation*}
	\overline{\varphi}(\theta)=
	\begin{pmatrix}
		0&1\\
		2&d(\theta)
	\end{pmatrix}.
\end{equation*}
Since we have
\begin{equation*}
\varphi_{n}(\theta)^{-1}=-\frac{\sqrt{r_{n}^{1}(H)}}{\sigma}
\begin{pmatrix}
-\sigma a_{n}(H) &-1\\
0&a_{n}(H)^{-1}
\end{pmatrix},
\end{equation*}
Theorems~\ref{Thm:LAN} and~\ref{Thm:Hajek-LeCam} and the computations of Appendix \ref{section:explicit_inverse} give the asymptotic lower bound
\begin{equation}\label{LB-sigma}
	\liminf_{n\to\infty}
	\frac{r_{n}^{1}(H)}{\sigma^{2}a_{n}(H)^{2}}\mathbb{E}_{\theta}^{n}[(\widehat{\sigma}_{n}-\sigma)^{2}]\geq 
	\mathcal{F}_{22}(\theta)(\mathcal{F}_{11}(\theta)\mathcal{F}_{22}(\theta)-\mathcal{F}_{12}(\theta)^{2})^{-1}
\end{equation}
for any sequence of estimators $\{\widehat{\sigma}_{n}\}_{n\in\mathbb{N}}$ by taking $\ell(x,y,z)=y^{2}$ in Theorem~\ref{Thm:Hajek-LeCam}. 
This lower bound of estimators implies that the asymptotically efficient rate of estimation for $\sigma$ is given by $r_{n}^{1}(H)^{\frac{1}{2}}|a_{n}(H)|^{-1}$.

\begin{remark}\rm
    Assume that $\Delta_{n}=n^{-1}$, $\nu_{n}=1$ and that $H$ and $\tau$ are \textit{known} as considered in the previous work \cite{Sabel-Schmidt-Hieber-2014} We write $\widetilde{\mathbb{P}}^{n}_{\widetilde{\sigma}}=\mathbb{P}_{(H,\widetilde{\sigma},\tau)}^{n}$ for conciseness. Then, proceeding as in the proof of Theorem~\ref{Thm:LAN}, we can also prove that
	\begin{equation*}
    \log\frac{\mathrm{d}\widetilde{\mathbb{P}}^{n}_{\sigma+r_{n}^{1}(H)^{-1/2}u}}{\mathrm{d}\widetilde{\mathbb{P}}^{n}_{\sigma}}
		=\widetilde{\zeta}_{n}(\sigma)u
		-\frac{1}{2}\widetilde{\mathcal{I}}(\sigma)u^{2}+o_{\widetilde{\mathbb{P}}_{\sigma}^{n}}(1)\ \ \mbox{as $n\to\infty$}
	\end{equation*}
    for each $u\in\mathbb{R}$, where
    \begin{equation*}
		\widetilde{\zeta}_{n}(\widetilde{\sigma}):=r_{n}^{1}(H)^{-1/2}\partial_{\sigma}\ell_{n}(\widetilde{\theta}),\ \ 
		\widetilde{\mathcal{I}}(\widetilde{\sigma}):=
		(2\widetilde{\sigma})^{-2}\mathcal{F}_{22}(\widetilde{\theta}),\ \
        \widetilde{\theta}:=(H,\widetilde{\sigma},\tau).
	\end{equation*}
    Moreover, we have $\mathcal{L}\{\widetilde{\zeta}_{n}(\sigma)|\widetilde{\mathbb{P}}^{n}_{\sigma}\}\rightarrow\mathcal{N}(0,\widetilde{\mathcal{I}}(\sigma))$ as $n\to\infty$. An easy application Theorem~\ref{Thm:Hajek-LeCam} then yield the asymptotic lower bound
	\begin{equation*}
		\liminf_{n\to\infty}
		r_{n}^{1}(H)\mathbb{E}_{\theta}^{n}[(\widehat{\sigma}_{n}-\sigma)^{2}]\geq 
		\widetilde{\mathcal{I}}(\sigma)^{-1}
		=8\sigma^{2-\frac{2}{\Diamond(H)}}\tau^{\frac{2}{\Diamond(H)}}c_{H}^{\frac{1}{\Diamond(H)}}
		\frac{\Diamond(H)^{2}\sin\left(\frac{\pi}{\Diamond(H)}\right)}{\Diamond(H)-1}
	\end{equation*}
	for any sequence of estimators $\{\widehat{\sigma}_{n}\}_{n\in\mathbb{N}}$ by taking a loss function $\ell(x)=x^{2}$ which give an asymptotic lower bound of estimating $\sigma$ similar to the one given in \cite{Sabel-Schmidt-Hieber-2014}.
\end{remark}

\subsubsection{Optimal asymptotic variances}

Using Mathematica, we obtain the closed-form expression of each element of the matrix $\mathcal{F}(\theta)$ and we can compute the optimal asymptotic variances of estimating $H$ and $\sigma$, see \eqref{LB-H-tau} and \eqref{LB-sigma}, as follows:
	\begin{equation*}
		\mathcal{F}_{22}(\theta)(\mathcal{F}_{11}(\theta)\mathcal{F}_{22}(\theta)-\mathcal{F}_{12}(\theta)^{2})^{-1}
		=\frac{
		\Diamond(H)^{4}(\Diamond(H)-1)\sin\left(\frac{\pi}{\Diamond(H)}\right)^{3}
		}{
		\Diamond(H)^{2}\cos\left(\frac{2\pi}{\Diamond(H)}\right)
		+2\pi^{2}(\Diamond(H)-1)^{2}-\Diamond(H)^{2}
		}.
	\end{equation*}
Remarkably, this variance is independent of $\sigma$ and $\tau$, relying solely on $H$ and $K$ and we refer to Figure~\ref{Fig:OptVarEst} for an illustration.

\begin{figure}
\centering
  \includegraphics[width=0.7\textwidth]{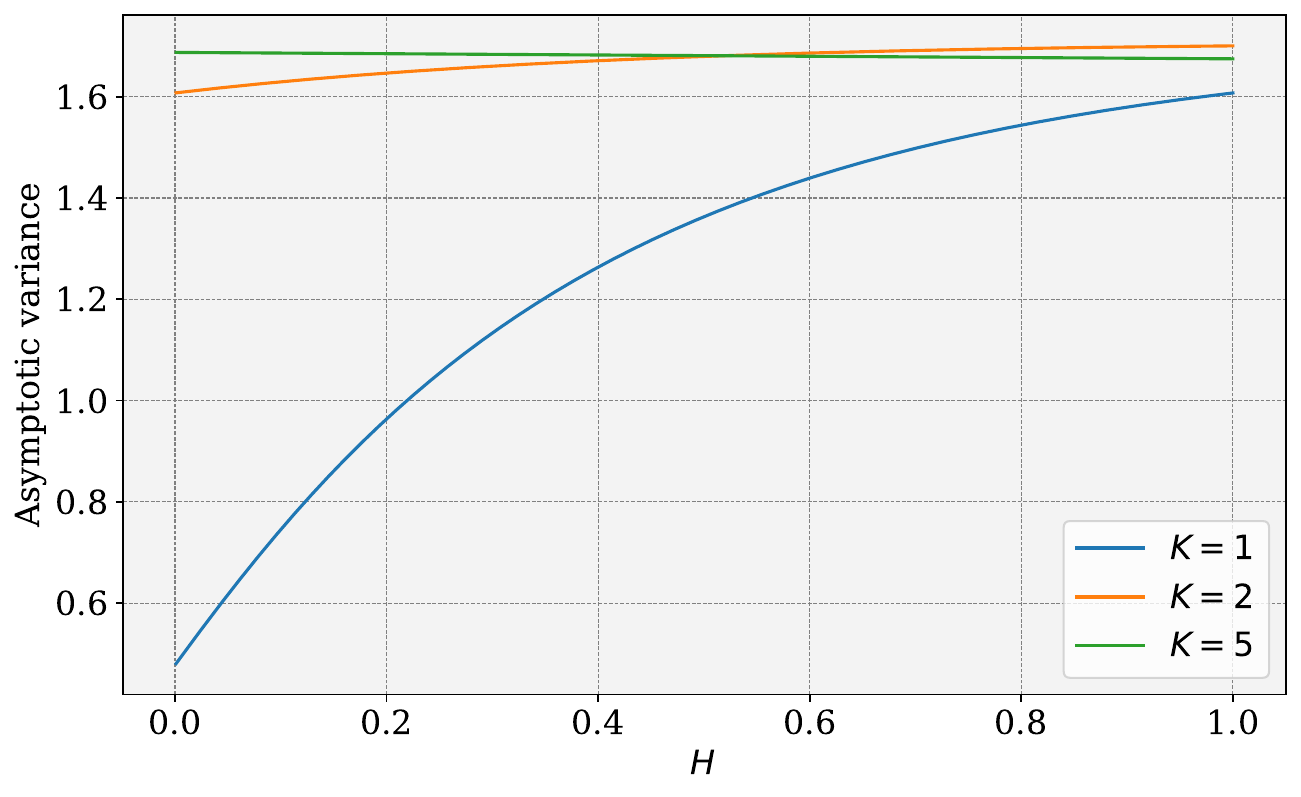}
  \caption{Asymptotic variances for different values of $K$ and $H$}
  \label{Fig:OptVarEst}
\end{figure}

\section{Estimation through Quadratic variations}
\label{Sec:estimation}

In this section, we focus on the estimation of the parameters of the model of Section~\ref{Sec:stat_model} under high-frequency asymptotics. More precisely, we suppose that $\Delta_{n}=1/n$ and $\nu_{n} = \nu$ is constant. To that extent, we first build a rate optimal estimator using a wavelet approach. Then we refine this estimator to reach the efficient convergence rate following a Le Cam’s one-step procedure.

\subsection{Energy levels}
\label{Sec:energy}

We plan to follow a strategy similar to \cite{Chong-Hoffmann-Liu-Rosenbaum-Szymanski-2022-MinMax}. However, due to the shape of the noise, we proceed to a weighed pre-averaging of the data. We consider a continuous function $w: [0,1] \to \mathbb{R}$ representing the weights used in the pre-averaging. The pre-averaging function depends on $K$ and we assume the following. 
\begin{assumption}
\label{assumption:alpha} $w$ is $K$-times continuously differentiable on $[0,1]$ and it satisfies
\begin{equation*}
\forall \, 0 \leq i < K,\; \partial_{x}^{i}w(0) = \partial_{x}^{i}w(1) = 0.
\end{equation*}
\end{assumption}
Of course, there exist functions satisfying Assumption~\ref{assumption:alpha}, for instance one can take $w(x) = x^{K}(1-x)^{K}$.\\

We introduce the wavelet-type coefficients of the fractional Brownian motion $W^H$ pre-averaged using the weight function $w$. We write $p \geq 1$ for the pre-averaging level. For any non-negative integer $j < J_{p,n}^{\star} = (n+1)/p - 2 \leq n/p$, we define
\begin{equation*}
	d_{j,p,n}^{(w)} 
	:=\frac{1}{\sqrt{np}}
	\sum_{\ell=0}^{p-1}w(\tfrac{\ell}{p})
	\left(W^H_{(jp+\ell)n^{-1}}-2W^H_{((j+1)p+\ell)n^{-1}}+W^H_{((j+2)p+\ell)n^{-1}}\right).
\end{equation*}
Their behaviour is summarized in the following Lemma.
\begin{lemma}[Properties of the Wavelet Coefficients]
\label{lem:wavelet:properties}
For any $p \geq 1$ and $0 < H < 1$, we define
\begin{equation*}
	\kappa_{p}^{(w)}(H)
	:=p^{-2}\sum_{\ell_{1}=0}^{p-1}\sum_{\ell_{2}=0}^{p-1}
	w(\ell_{1}/p)w(\ell_{2}/p)\phi_{H}((\ell_{1} - \ell_{2}) / p^{-1}),
\end{equation*}
where $\phi_{H}(x) := \tfrac{1}{2}\sum_{k=0}^{4} (-1)^{k+1} \binom{4}{k} |x+k-2|^{2H}$. The random variables $(d_{j,p,n})_j$ form a stationary Gaussian process whose variance is given by
\begin{equation*}
\mathbb{E}_{\theta}^{n}[(d_{j,p,n}^{(w)})^{2}] = \kappa_{p}^{(w)}(H) (p/n)^{1+2H}.
\end{equation*}
Moreover, the autocorrelation function of the $d_{j,p,n}$ decays quickly at infinity. More precisely, there exists $0 < c_{-}^{(w)} < c_{+}^{(w)}$ independent of $\theta$ such that for all $ j_{1}, j_{2}$ satisfying $|j_{1} - j_{2}| \geq 3$, we have
\begin{equation*}
c_{-}^{(w)}
(1+|j_{1}-j_{2}|)^{2H-4}
\leq 
(n/p)^{(1+2H)}
|\mathbb{E}_{\theta}^{n}{[d^{(w)}_{j_{1},p,n}d^{(w)}_{j_{2},p,n}]}|
\leq 
c_{+}^{(w)}
(1+|j_{1}-j_{2}|)^{2H-4}.
\end{equation*}
\end{lemma}
This result is very similar to Lemmas 17 and 18 in \cite{Chong-Hoffmann-Liu-Rosenbaum-Szymanski-2022-MinMax}. The proofs can be found in Appendix~\ref{Sec:lem:wavelet:properties:proof}. Note that the function $\kappa_{p}^{(w)}$ appearing in Lemma~\ref{lem:wavelet:properties} is explicit but does not have any easy closed formula. It can be seen as the Riemann sum approximation of the integral
\begin{equation*}
\kappa_{\infty}^{(w)}(H) := \int_{0}^1 \int_{0}^1 w(x)
w(y) \phi_{H}(x-y) \,\mathrm{d}x\,\mathrm{d}y
\end{equation*}
defined for all $0 < H < 1$. Its main properties are listed in Lemma~\ref{lem:wavelet:kappa_p} which is proved in Appendix~\ref{Sec:lem:wavelet:kappa_p:proof}.
\begin{lemma}[Properties of the function $\kappa_{p}^{(w)}$]
\mbox{}
\label{lem:wavelet:kappa_p}
\begin{itemize}
\item \textbf{Differentiability: } The functions $\kappa_{p}^{(w)}$ and $\kappa_\infty^{(w)}$ are infinitely differentiable on $(0,1)$.
\item \textbf{Uniform convergence: } There exists $c^{(w)}_0 > 0$ such that for any $H_{-} \leq H \leq H_{+}$, we have
\begin{equation*}
|\kappa_{p}^{(w)}(H) - \kappa_\infty^{(w)}(H) | \leq \tfrac{c^{(w)}_0}{p^{2H\wedge 1}}
\;\; \text{ and } \;\;
|{\partial_{H}\kappa_{p}^{(w)}}(H) - {\partial_{H}\kappa_{\infty}^{(w)}}(H) | \leq \tfrac{c^{(w)}_0}{p^{2H\wedge 1}}.
\end{equation*}
\item \textbf{Uniform control: } We have
\begin{equation}
\label{eq:kappa:uniform_bound}
0 < \inf_{p\geq 2} \inf_{H} \kappa_{p}^{(w)}(H) \leq \sup_{p\geq 1} \sup_{H}  \kappa_{p}^{(w)}(H) <\infty
\end{equation}
and there exists $c^{(w)}_1 > 0$ such that for any $H_{-} \leq H \leq H_{+}$, we have
\begin{equation*}
\left|\partial_{H}\kappa_{\infty}^{(w)}(H)\right|\leq c^{(w)}_1.
\end{equation*}
\end{itemize}
\end{lemma}

As in \cite{gloter2007estimation}, the wavelet coefficients are not observed directly. Instead, we observe
\begin{equation*}
	\widetilde{d}_{j,p,n}^{(w)} 
	:=\frac{1}{\sqrt{np}}
	\sum_{\ell=0}^{p-1}w(\tfrac{\ell}{p})
	\left(X_{jp+\ell}^{n}-2X_{(j+1)p+\ell}^{n}+X_{(j+2)p+\ell}^{n}\right)
\end{equation*}
which relate to the wavelets coefficients through the relation
$
\widetilde{d}_{j,p,n}^{(w)}
=
\sigma d_{j,p,n}^{(w)} + \tau  e_{j,p,n}^{(w)},
$
where 
\begin{equation*}
	e_{j,p,n}^{(w)}:=
	\frac{1}{\sqrt{np}}
	\sum_{\ell=0}^{p-1}w(\tfrac{\ell}{p})
	\left(Y_{jp+\ell}^{n}-2Y_{(j+1)p+\ell}^{n}+Y_{(j+2)p+\ell}^{n}\right).
\end{equation*}

\begin{lemma}[Structure of the noise]
\label{lem:noise:structure}
$\{e_{j,p,n}^{(w)}\}_{j}$ is a stationary Gaussian sequence, whose distribution does not depend on $\theta$. Moreover, its variance is given  by $ \Var[e_{j,p,n}^{(w)}] = \gamma_p^{(w)} n^{-1}$, where $\gamma_p^{(w)} $ depends only on $p$, $w$ and $K$. Moreover, the sequence $\{\gamma_{p}^{(w)}\}_p$ is bounded and $\mathbb{E}_{\theta}^{n} [e_{j_{1},p,n}^{(w)}e_{j_{2},p,n}^{(w)}] = 
0$ provided $|j_{1} - j_{2}| \geq K+4$.

Suppose in addition that $w$ satisfies Assumption~\ref{assumption:alpha}. Then, there exist $0 < c_{-}^{(w)} < c_{+}^{(w)}$ such that 
\begin{equation*}
c_{-}^{(w)} p^{-2K} \leq \gamma_p^{(w)} \leq c_{+}^{(w)}p^{-2K}
\end{equation*}
when $p$ is large enough. (The upper bound is true for any $p$).
\end{lemma}
This Lemma is proved in Section~\ref{Sec:lem:noise:structure:proof} and relies heavily on the choice of $w$ specified in Assumption~\ref{assumption:alpha}. Indeed, the vanishing derivative of $w$ ensures that when plugging \eqref{eq:constructionY} in the definition of $e_{j,p,n}^{(w)}$, most term vanish and do not contribute to the variance. \\

We are now ready to define the energy levels of the observations. 
For each $J\in\{1,\cdots,J_{p,n}^{\star}\}$, we write 
\begin{equation}
    \label{Def:AggEnergy}
	\widehat{Q}_{J,p,n}^{(w)}:=\sum_{j=0}^{J-1}
	\left|\widetilde{d}_{j,p,n}^{(w)}\right|^{2}.
\end{equation}
Using the previous bounds on the wavelet coefficients, we derive an $L^{2}$-concentration property for these energy levels.

\begin{proposition}
\label{prop:deviation:Q}
We have
\begin{equation}
\label{eq:expec:energy}
\mathbb{E}_{\theta}^{n}\left[\widehat{Q}_{J,p,n}^{(w)}\right]
=
J\left(
\sigma^{2}\kappa_{p}^{(w)}(H)(p/n)^{1+2H} + \tau^{2} \gamma_{p}^{(w)}n^{-1}
\right).
\end{equation}
Moreover, there exist $0 < c_{-}^{(w)} < c_{+}^{(w)}$, independent of $n,p,J$ and $\theta$, such that
\begin{equation*}
c_{-}^{(w)}J\bigg(
\Big(\frac{p}{n}\Big)^{2+4H}+\bigg(\frac{\gamma_{p}^{(w)}}{n}\bigg)^{2}
\bigg)
\leq
\Var_{\theta}^{n}\Big[\widehat{Q}_{J,p,n}^{(w)}\Big]
\leq 
c_{+}^{(w)}J\bigg(
\Big(\frac{p}{n}\Big)^{2+4H}+\bigg(\frac{\gamma_{p}^{(w)}}{n}\bigg)^{2}
\bigg).
\end{equation*}
In particular, when $w$ satisfies Assumption~\ref{assumption:alpha}, we have
\begin{equation*}
c_{-}^{(w)}
J
\bigg(
\Big(\frac{p}{n}\Big)^{2+4H}+\frac{1}{n^{2}p^{4K}}\bigg)
\leq
\Var_{\theta}^{n}\left[\widehat{Q}_{J,p,n}^{(w)}\right]
\leq 
c_{+}^{(w)}J
\bigg(
\Big(\frac{p}{n}\Big)^{2+4H}+\frac{1}{n^{2}p^{4K}}\bigg).
\end{equation*}
\end{proposition}

The estimation procedure presented in details in Section~\ref{Sec:QVestimators} is mainly derived from the behaviour of $\mathbb{E}_{\theta}^{n}\left[\widehat{Q}_{J,p,n}^{(w)}\right]$ in \eqref{eq:expec:energy}. Indeed, ignoring the noise term and the presence of $\kappa_p$, it is plausible that
\begin{equation}
\label{eq:idea:guess_est}
	\frac{\widehat{Q}_{J,2p,n}^{(w)}}{\widehat{Q}_{2J,p,n}^{(w)}}
	\approx 2^{2H}
\end{equation}
and the approximation is precise when the variance of $\widehat{Q}_{J,2p,n}^{(w)}
$ is small. Therefore we choose the level $p$ so that this variance term is as small as possible. We define $p_{n}^{\mathrm{opt}}(H)$ by
\begin{equation}
\label{eq:Def:optimalchoice}
 p_{n}^{\mathrm{opt}}(H) := 
 \lfloor n^{2H/\Diamond(H)} \rfloor \vee 2,
\end{equation}
where $\Diamond(H)=2K+2H+1$. However, the presence of a noise term in \eqref{eq:expec:energy} and the fact that $\kappa_{2p}(H) = \kappa_p(H)$ induce an asymptotic bias in the estimation of $H$ that needs to be corrected. Therefore, we introduce the bias-corrected energy levels by
\begin{equation}\label{eq:Def:correctedenergy}
	\widehat{Q}_{J,p,n}^{(w),c}(\theta):=
	\widehat{Q}_{J,p,n}^{(w)}
	-J\left(
	\sigma^{2}\overline{\kappa}_{p}^{(w)}(H)\left(\frac{p}{n}\right)^{1+2H} 
	+\tau^{2}\gamma_{p}^{(w)}n^{-1}
	\right)
\end{equation}
for $\theta=(H,\sigma,\tau)\in\Theta$,
where $\overline{\kappa}_{p}^{(w)}(H):=\kappa_{p}^{(w)}(H)-\kappa_{\infty}^{(w)}(H)$.

\subsection{Rate optimal estimators}
\label{Sec:QVestimators}

\subsubsection*{Uncorrected Estimators}

Recall that $J_{p,n}^{\star}=\lfloor (n+1)/p \rfloor - 2 $ and set $J_{p,n}:=\lfloor J_{p,n}^{\star}/2\rfloor$. Utilizing Proposition~\ref{prop:deviation:Q}, we first build estimators for both $H$ and $\tau$, exhibiting suboptimal convergence rates. We start with $H$ using the ratio of Equation \eqref{eq:idea:guess_est}. To ensure its consistency, we select $p_{n}^{(0)}$ that is independent of the value of $H$ and satisfies $p_{n}^{(0)} > p_{n}^{\mathrm{opt}}(H)$. This choice must ensure that the energy levels concentrate around their mean, as demonstrated in Proposition~\ref{prop:deviation:Q}, even though the rate of concentration may be suboptimal. Furthermore, it must also ensure that the bias induced by $\tau$ on the expected energy levels remains small enough to get consistency. Thus, we define
\begin{equation}
\label{eq:Def:guess_est}
\widehat{H}^{(0)}_{n} := \left( 
\left( 
\frac{1}{2}\log_{2}\left[
\frac{\widehat{Q}_{J_{n}^{(0)},2p_{n}^{(0)},n}^{(w)}}{\widehat{Q}_{2J_{n}^{(0)},p_{n}^{(0)},n}^{(w)}}
\right] 
\right)
\vee H_{-}
\right)
\wedge H_{+}. 
\end{equation}
where $p_{n}^{(0)} := \big \lfloor n^{2/(2K+3)} \big \rfloor \vee 2$ and $J_{n}^{(0)}:=J_{p_{n}^{(0)},n}$. The behaviour of $\widehat{H}^{(0)}_{n}$ is summarized in Lemma~\ref{lemma:uncorrectedQV:H} which is proved in Section~\ref{Sec:lemma:uncorrectedQV:H:proof}. 
\begin{lemma}\label{lemma:uncorrectedQV:H}
	Assume that $w$ satisfies Assumption~\ref{assumption:alpha}. Set
	\begin{equation*}
		v_{n}^{(0)}(H) := n^{(2K+1)/(2(2K+3))} \wedge n^{2(1-H)(2K+1)/(2K+3)} \wedge n^{2\{(2H)\wedge 1\}/(2K+3)}.
	\end{equation*}
	Then $\{v_{n}^{(0)}(H)(\widehat{H}_{n}^{(0)}-H)\}_{n\in\mathbb{N}}$ is bounded in $\mathbb{P}^{n}_{\theta}$-probability, uniformly over $\Theta$. 
\end{lemma}

Then we define a first uncorrected estimator of $\tau$. We always choose a level $p = 1$ here to maximise the importance of the parameter $\tau$ in the expectation of the energy levels in Equation~\eqref{eq:expec:energy}. For the same reason, we always choose a constant pre-averaging function, i.e. $w(x) = 1$, also ensuring that $\gamma_{1}^{(1)} = { 2K+6 \choose K+3} > 0$. We write
\begin{equation}\label{eq:Def:uncorrected:tau}
	\widehat{\tau}^{(0)}_{n} := \left(
	\left[
	n\frac{\widehat{Q}_{n}^{(1)}}{(n-1)\gamma_{1}^{(1)}}
	\right]^{1/2}
	\vee \tau_- \right) \wedge \tau_+.
\end{equation}

\begin{lemma}\label{lemma:uncorrectedQV:tau}
	Let $u_{n}^{(0)}(H) := n^{(1/2)\wedge (2H)}$. 
	Then $\{u_{n}^{(0)}(H)(\widehat{\tau}_{n}^{(0)} - \tau)\}_{n\in\mathbb{N}}$ is bounded in $\mathbb{P}^{n}_{\theta}$-probability, uniformly over $\Theta$.
\end{lemma}

This lemma is proved in Section~\ref{Sec:lemma:uncorrectedQV:tau:proof}. Note that the convergence rate obtained in Lemma~\ref{lemma:uncorrectedQV:tau} is optimal for $H \geq 1/4$ and suboptimal otherwise. This is due to a bias induced by the effect of the term related to the fractional Brownian motion in the expectation of the energy levels. Intuitively, smaller $H$ make the fractional Brownian motion look like a white noise so that it is harder to retrieve the correct intensity of the noise.

\subsubsection*{Bias-Correction Procedure}

We now present the bias-correction procedure used to infer $H$, $\sigma$ and $\tau$ with better rates given a preliminary knowledge, for instance, using a preliminary guess estimator with a sub-optimal convergence rate. In contrast with the uncorrected estimators presented previously, these estimators are based on an adaptive level choice built using preliminary knowledge.\\

We first focus on the pre-averaging level choice. Let $\widetilde{H}\in\Theta_{H}$ represent a guess for the true value of $H$. Since we want to choose a pre-averaging level as close as possible to the optimal level choice $p_{n}^{\mathrm{opt}}(H) = \lfloor n^{2H/\Diamond(H)} \rfloor$ defined in \eqref{eq:Def:optimalchoice}, one might be tempted to use $ \widehat{p}_{n} = \lfloor n^{2\widetilde{H}/\Diamond(\widetilde{H})} \rfloor$.
However, for technical reasons first identified in \cite{Szymanski-2022}, this does not work as any estimator of $H$ cannot be used as a guess. Indeed, the estimators considered as guesses must take their values from a discrete set so that we can obtain deterministic sequences. With this view, we introduce a fixed sequence of positive integers $\{q_{n}\}_{n\in\mathbb{N}}$ satisfying certain assumptions on this sequence that will be specified later on. We first consider the projection $h_{n}(\widetilde{H})$ of $\widetilde{H}$ on a grid discretizing $[H_{-}, H_{+}]$ with mesh $q_{n}^{-1}(H_{+} - H_{-})$, i.e. we write
\begin{equation*}
	h_{n}(\widetilde{H}) := H_{-} + \frac{\widehat{i}_{n}(\widetilde{H})}{q_{n}}( H_{+} - H_{-} ),\text{ where } i_{n}(\widetilde{H}) :=  \Big \lfloor q_{n} \frac{\widetilde{H} - H_{-}}{H_{+} - H_{-}}  \Big \rfloor.
\end{equation*}
Then we imitate the definition of $p_{n}^{\mathrm{opt}}(H)$ by adaptively choosing the level 
\begin{equation}\label{def_p-hat}
	\widehat{p}_{n}(\widetilde{H}) := \lfloor n^{\frac{2h_{n}(\widetilde{H})}{\Diamond(h_{n}(\widetilde{H}))}} \rfloor.
\end{equation}
If the guess $\widetilde{H}$ is a "good" preliminary estimator of $H$, we should have $\widetilde{H} \approx H$ with high probability so that $i_{n}(\widetilde{H}) = i_{n}^{b}(H)$ for some $b \in \{-1, 0, 1\}$ with high probability, where
\begin{equation*}
	i_{n}^{b}(H) := 
	\Big\lfloor q_{n}\frac{H - H_{-}}{H_{+} - H_{-}} \Big\rfloor + b,\ \ b \in \{-1, 0, 1\}.
\end{equation*}
In view of this, we also define
\begin{equation*}
	h_{n}^{b}(H) := H_{-} + \frac{i_{n}^{b}(H)}{q_{n}}( H_{+} - H_{-} )
	\ \ \text{ and }\ \ 
	p_{n}^{b}(H) := \lfloor n^{\frac{2h_{n}^{b}(H)}{\Diamond(h_{n}^{b}(H))}} \rfloor.
\end{equation*}
We state the asymptotic behaviour of $\{\widehat{p}_{n}(\widetilde{H})\}_{n\in\mathbb{N}}$ rigorously as follows. 
\begin{lemma}\label{lem:behaviour_adaptive_level}
	Let $\{v_{n}(H)\}_{n\in\mathbb{N}}$ be a sequence of positive functions satisfying $v_{n}(H)^{-1}\log{n}\to 0$ as $n\to\infty$ uniformly on $\Theta$. 
	Consider a sequences of estimators $\{\widehat{H}_{n}\}_{n\in\mathbb{N}}$ such that $\widehat{H}_{n} \in [H_{-}, H_{+}]$ for each $n\in\mathbb{N}$ and $\{v_{n}(H)(\widehat{H}_{n}-H)\}_{n\in\mathbb{N}}$ is bounded in $\mathbb{P}^{n}_{\theta}$-probability, uniformly over $\Theta$. 
	We also assume that the sequence $\{q_{n}\}_{n\in\mathbb{N}}$ satisfies $q_{n}\to\infty$ and $q_{n}v_{n}(H)^{-1}\to 0$ as $n\to\infty$ uniformly on $\Theta$.
	Then we have
	\begin{equation*}
		\inf_{\theta=(H,\sigma,\tau) \in \Theta}
		\mathbb{P}^{n}_{\theta}\left[
		\widehat{p}_{n}(\widehat{H}_{n})  \in \{ p_{n}^{b}(H); \; b \in \{-1, 0, 1\} \}
		\right] \to 1\ \ \mbox{as $n\to\infty$}.
	\end{equation*}
\end{lemma}

\begin{proof}
By the definition of $\widehat{p}_{n}(\widehat{H}_{n})$, we have
\begin{align*}
\mathbb{P}^{n}_{\theta}\left[
\widehat{p}_{n}(\widehat{H}_{n})  \in \{ p_{n}^{b}(H); \; b \in \{-1, 0, 1\} \}
\right] 
&\geq 
\mathbb{P}^{n}_{\theta}\left[
|\widehat{H}_{n} - H| \leq q_{n}^{-1}(H_{+} - H_{-})
\right]
\\
&\geq 
\mathbb{P}^{n}_{\theta}\left[
v_{n}(H) |\widehat{H}_{n} - H| \leq v_{n}(H)q_{n}^{-1} (H_{+} - H_{-})
\right].
\end{align*}
The conclusion follows from the tightness of $v_{n}(H) |\widehat{H}_{n} - H|$ and $v_{n}(H)q_{n}^{-1} (H_{+} - H_{-}) \to \infty$. 
\end{proof}

A consequence of Lemma~\ref{lem:behaviour_adaptive_level} is that in the following proofs, the adaptive level choices $\widehat{p}_{n}(\widehat{H}_{n})$, which are intrinsically random, can be replaced by the deterministic optimal level choices $p_{n}^{b}(H)$. Indeed, since we have $|h_{n}^{b}(H) - H| \leq 2 q_{n}^{-1} (H_{+} - H_{-})$, we can show 
\begin{equation*}
	\left|
	\frac{2h_{n}^{b}(H)}{\Diamond(h_{n}^{b}(H))} - \frac{2H}{\Diamond(H)}
	\right| \lesssim q_{n}^{-1}
\end{equation*}
uniformly on $\Theta$. Therefore, assuming $q_{n}^{-1}\log{n}\to 0$ as $n\to\infty$, we obtain
\begin{equation*}
n^{-2H/\Diamond(H)} p_{n}^{b}(H) \to 1\ \ \mbox{as $n\to\infty$ uniformly on $\Theta$.}
\end{equation*}

We now use this level choice to build an adaptive bias-corrected estimator of $\sigma$. We define for $\widetilde{H} \in [H_{-}, H_{+}]$ and $\widetilde{\tau} \in [\tau_{-},\tau_{+}]$
\begin{align}
\label{eq:Def:corrected:sigma}
\widehat{\sigma}^{c}_{n}(\widetilde{H},\widetilde{\tau}) := \left( 
	\left[
	\frac{
		\widehat{Q}_{J_{n}^{\star}(\widetilde{H}),\widehat{p}_{n}(\widetilde{H}),n}^{(w)} 
		- 
		J_{n}(\widetilde{H}) \widetilde{\tau}^{2} \gamma_{\widehat{p}_{n}(\widetilde{H})}^{(w)} n^{-1}
	}{
		J_{n}(\widetilde{H}) \kappa_{\widehat{p}_{n}(\widetilde{H})}^{(w)}(\widetilde{H}) (\widehat{p}_{n}(\widetilde{H})/n)^{1+2\widetilde{H}}
	}
\right]^{1/2}
\vee \sigma_-
\right)
\wedge \sigma_+
\end{align}
where
\begin{align}\label{def-Jn-H}
    J_{n}^{\star}(\widetilde{H}):=J^{\star}_{\widehat{p}_{n}(\widetilde{H}),n}
    \;\;\mbox{ and }\;\;
    J_{n}(\widetilde{H}):=J_{\widehat{p}_{n}(\widetilde{H}),n}
\end{align}
and where $\widehat{p}_{n}(\widetilde{H})$ is defined in \eqref{def_p-hat} and $\widehat{Q}_{J,p,n}^{(w)}$ is defined in \eqref{Def:AggEnergy}. The asymptotic behaviour of $\widehat{\sigma}^{c}_{n}$ is precisely stated in Proposition~\ref{prop:corrected:sigma}, which is proved in Section~\ref{sec:prop:corrected:sigma:proof}.
\begin{proposition}
\label{prop:corrected:sigma}
	Let $\{u_{n}(H)\}_{n\in\mathbb{N}}$ and $\{v_{n}(H)\}_{n\in\mathbb{N}}$ be sequences of positive functions satisfying $u_{n}(H)^{-1}\to 0$ and $v_{n}(H)^{-1}\log{n}\to  0$ as $n\to\infty$ uniformly on $\Theta$. 
	Consider two sequences of estimators $\{\widehat{H}_{n}\}_{n\in\mathbb{N}}$ and $\{\widehat{\tau}_{n}\}_{n\in\mathbb{N}}$ such that $\widehat{H}_{n} \in [H_{-}, H_{+}]$ and $\widehat{\tau}_{n} \in [\tau_-, \tau_+]$ for each $n\in\mathbb{N}$ and $\{v_{n}(H)(\widehat{H}_{n}-H)\}_{n\in\mathbb{N}}$ and $\{u_{n}(H)(\widehat{\tau}_{n}-\tau)\}_{n\in\mathbb{N}}$ are bounded in $\mathbb{P}^{n}_{\theta}$-probability, uniformly over $\Theta$. 
	We also assume that the function $w$ satisfies Assumption~\ref{assumption:alpha} and $\{q_{n}\}_{n\in\mathbb{N}}$ satisfies $q_{n}^{-1}\log{n}\to 0$ and $q_{n}v_{n}(H)\to  0$ as $n\to\infty$ uniformly on $\Theta$. 
	Then $\{s_{n}(H)(\widehat{\sigma}^{c}_{n}(\widehat{H}_{n},\widehat{\tau}_{n})-\sigma)\}_{n\in\mathbb{N}}$ is bounded in $\mathbb{P}^{n}_{\theta}$-probability, uniformly over $\Theta$, where
	\begin{align*}
		s_{n}(H) := n^{(2K+1)/(2\Diamond(H))} \wedge u_{n}(H) \wedge \left(v_{n}(H)|\log{n}|^{-1}\right).
	\end{align*}
\end{proposition}

We now get to the estimation of an adaptive bias-corrected estimator of $H$. For $\widetilde{\theta}=(\widetilde{H},\widetilde{\sigma},\widetilde{\tau})\in\Theta$, we write
\begin{equation*}
\widehat{H}^{c}_{n}(\widetilde{\theta})
:= \left( 
\left( 
\frac{1}{2}\log_{2}\left[
\widehat{Q}^{(w),c}_{2J_{n}(\widetilde{H}),\widehat{p}_{n}(\widetilde{H}),n}(\widetilde{\theta}) \, \Big / \,
{\widehat{Q}^{(w),c}_{J_{n}(\widetilde{H}),2\widehat{p}_{n}(\widetilde{H}),n}(\widetilde{\theta})}
\right] 
\right)
\vee H_{-}
\right)
\wedge H_{+},
\end{equation*}
where $\widehat{Q}^{(w),c}_{J,p,n}$ and $J_{n}(\widetilde{H})$ are defined in \eqref{eq:Def:correctedenergy} and \eqref{def-Jn-H} respectively. 
Its convergence rate is precisely stated in Proposition~\ref{prop:corrected:H}, which is proved in Section~\ref{sec:prop:corrected:H:proof}.
\begin{proposition}\label{prop:corrected:H}
	Let $\{u_{n}(H)\}_{n\in\mathbb{N}}$ and $\{v_{n}(H)\}_{n\in\mathbb{N}}$ be sequences of positive functions satisfying $u_{n}(H)^{-1}\to  0$ and $v_{n}(H)^{-1}\log{n}\to  0$ as $n\to\infty$ uniformly on $\Theta$.  
	Consider three sequences of estimators $\{\widehat{H}_{n}\}_{n\in\mathbb{N}}$, $\{\widehat{\sigma}_{n}\}_{n\in\mathbb{N}}$ and $\{\widehat{\tau}_{n}\}_{n\in\mathbb{N}}$ such that $\widehat{H}_{n} \in [H_{-}, H_{+}]$, $\widehat{\sigma}_{n} \in [\sigma_-, \sigma_+]$ and $\widehat{\tau}_{n} \in [\tau_-, \tau_+]$ for each $n\in\mathbb{N}$ and $\{v_{n}(H)(\widehat{H}_{n}-H)\}_{n\in\mathbb{N}}$, $\{v_{n}(H)|\log{n}|^{-1}(\widehat{\sigma}_{n}-\sigma)\}_{n\in\mathbb{N}}$ and $\{u_{n}(H)(\widehat{\tau}_{n}-\tau)\}_{n\in\mathbb{N}}$
	are bounded in $\mathbb{P}^{n}_{\theta}$-probability, uniformly over $\Theta$. 
	We also assume that the function $w$ satisfies Assumption~\ref{assumption:alpha} and $\{q_{n}\}_{n\in\mathbb{N}}$ satisfies $q_{n}^{-1}\log{n}\to 0$ and $q_{n}v_{n}(H)\to  0$ as $n\to\infty$ uniformly on $\Theta$. 
	Set $\widehat{\theta}_{n}=(\widehat{H}_{n},\widehat{\sigma}_{n},\widehat{\tau}_{n})$ for $n\in\mathbb{N}$. 
	Then $\{s_{n}(H)(\widehat{H}^{c}_{n}(\widehat{\theta}_{n})-H)\}_{n\in\mathbb{N}}$ is bounded in $\mathbb{P}^{n}_{\theta}$-probability, uniformly over $\Theta$, where
	\begin{equation*}
		s_{n}(H) :=  n^{(2K+1)/(2\Diamond(H))} \wedge u_{n}(H) \wedge \left(v_{n}(H)|\log{n}|^{-1} n^{\{(2H)\wedge 1\} 2H /\Diamond(H)} \right).
	\end{equation*}
\end{proposition}

We eventually build an adaptive bias-corrected estimator of $\tau$ using preliminary guesses for $H$ and $\sigma$. For $(\widetilde{H},\widetilde{\sigma})\in\Theta_{H}\times\Theta_{\sigma}$, we define 
\begin{align*}
\widehat{\tau}^{c}_{n}(\widetilde{H}, \widetilde{\sigma})
:= \left( 
 \left[
n
\frac{\widehat{Q}_{n} - (n-1)\widetilde{\sigma}^{2}\kappa_{1}^{(1)}(\widetilde{H}) n^{-1-2\widetilde{H}}}{(n-1)\gamma_{1}^{(1)}}
\right]^{1/2}
\vee \tau_-
\right)
\wedge \tau_+,
\end{align*}
where $\gamma_{p}^{(w)}$ is defined in Lemma~\ref{lem:noise:structure}. We study the behaviour of $\widehat{\tau}^{c}_{n}$ in Proposition~\ref{prop:corrected:tau} whose proof is delayed until Section~\ref{sec:prop:corrected:tau:proof}
\begin{proposition}\label{prop:corrected:tau}
	Let $\{v_{n}(H)\}_{n\in\mathbb{N}}$ be a sequence of positive function satisfying $v_{n}(H)^{-1}\log{n}\to 0$ as $n\to\infty$, uniformly on $\Theta$. 
	Consider two sequences of estimators $\{\widehat{H}_{n}\}_{n\in\mathbb{N}}$ and $\{\widehat{\sigma}_{n}\}_{n\in\mathbb{N}}$ such that $\widehat{H}_{n} \in [H_{-}, H_{+}]$ and $\widehat{\sigma}_{n} \in [\sigma_-, \sigma_+]$ for each $n\in\mathbb{N}$ and $\{v_{n}(H)(\widehat{H}_{n}-H)\}_{n\in\mathbb{N}}$ and $\{v_{n}(H)|\log{n}|^{-1}(\widehat{\sigma}_{n}-\sigma)\}_{n\in\mathbb{N}}$ are bounded in $\mathbb{P}^{n}_{\theta}$-probability, uniformly over $\Theta$. 
	Assume also that the sequence $\{q_{n}\}_{n\in\mathbb{N}}$ satisfies $q_{n}^{-1}\log{n}\to 0$ and $q_{n}v_{n}(H)\to  0$ as $n\to\infty$ uniformly on $\Theta$. 
	Then $\{s_{n}(H)(\widehat{\tau}^{c}_{n}(\widehat{H}_{n},\widehat{\sigma}_{n})-\tau)\}_{n\in\mathbb{N}}$ is bounded in $\mathbb{P}^{n}_{\theta}$-probability, uniformly over $\Theta$, where
	\begin{align*}
		s_{n}(H) := n^{\frac{1}{2}}\wedge \big(v_{n}(H)|\log{n}|^{-1} n^{2H}\big).
	\end{align*}
\end{proposition}

\subsubsection*{Iterative Procedure}

We are now ready to build our final estimators. First we define $\widehat{\theta}_{n}^{(0)}=(\widehat{H}^{(0)}_{n},\widehat{\sigma}^{(0)}_{n},\widehat{\tau}^{(0)}_{n})$ where $\widehat{H}^{(0)}_{n}$ and $\widehat{\tau}^{(0)}_{n}$ as in \eqref{eq:Def:guess_est} and \eqref{eq:Def:uncorrected:tau} and where $\widehat{\sigma}^{(0)}_{n} = \widehat{\sigma}^{c}_{n}(\widehat{H}^{(0)}_{n},\widehat{\tau}^{(0)}_{n})$ is defined in \eqref{eq:Def:corrected:sigma}. Then, for $m \geq 1$, we define estimators $\widehat{\theta}_{n}^{(m)}=(\widehat{H}^{(m)}_{n},\widehat{\sigma}^{(m)}_{n}, \widehat{\tau}^{(m)}_{n})$ by
\begin{align*}
\begin{cases}
    \widehat{H}^{(m)}_{n} := \widehat{H}^{c}_{n} (\widehat{\theta}_{n}^{(m-1)}),\\ 
	\widehat{\tau}^{(m)}_{n} := \widehat{\tau}^{c}_{n} (\widehat{H}^{(m)}_{n},\widehat{\sigma}^{(m-1)}_{n}),\\ 
	\widehat{\sigma}^{(m)}_{n} := \widehat{\sigma}^{c}_{n} (\widehat{H}^{(m)}_{n}, \widehat{\tau}^{(m)}_{n}).
\end{cases}
\end{align*}
An adequate study of the convergence rate of each of these estimators shows that a good stopping point is $m_{\mathrm{opt}} = \lfloor (2H_{-}+1)/8H_{-}^{2} \rfloor + 1$. The convergence rates obtained are summarized in Theorem~\ref{thm:rate_optimal_estimator} and proved in Section~\ref{Sec:thm:rate_optimal_estimator:proof}.
\begin{theorem}
\label{thm:rate_optimal_estimator}
Recall that $\Diamond(H)=2K+2H+1$ and $r_{n}(H)=n^{\frac{2K+1}{\Diamond(H)}}$. 
Then $\{r_{n}(H)^{\frac{1}{2}}(\widehat{H}^{(m_{\mathrm{opt}})}_{n}-H)\}_{n\in\mathbb{N}}$, $\{r_{n}(H)^{\frac{1}{2}}\log{n}^{-1}(\widehat{\sigma}^{(m_{\mathrm{opt}})}_{n}-\sigma)\}_{n\in\mathbb{N}}$ and $\{n^{\frac{1}{2}}(\widehat{\tau}^{(m_{\mathrm{opt}})}_{n}-\tau)\}_{n\in\mathbb{N}}$ are bounded in $\mathbb{P}^{n}_{\theta}$-probability, uniformly over $\Theta$. 
\end{theorem}

\begin{remark}
This whole procedure constructs rate-optimal estimators in the case where $H$, $\sigma$, and $\tau$ are unknown and for a general $K$. In practice, it may be difficult to implement due to numerical instability in the iterative step. However, when either $K = 0$ or when $\tau$ is known, one could build rate-optimal estimators without the use of this iterative procedure. We refer to Section~\ref{sec:numerical} for more details.
\end{remark}

\subsection{Asymptotic Efficiency}
In this section, we use the LAN property to establish the asymptotic efficiency of a one-step estimator, following a similar approach as presented in \cite{Brouste-Masuda-2018} and \cite{Brouste-Soltane-Votsi-2020}. Let $\{\varphi_{n}(\theta)\}_{n\in\mathbb{N}}$ be a sequence of rate matrices satisfying Assumption~\ref{Assump:RateMat}. Considering some initial estimators $\{\widetilde{\theta}_{n}\}_{n\in\mathbb{N}}$, we define $\{\widehat{\theta}_{n}\}_{n\in\mathbb{N}}$ by
\begin{equation*}
	\widehat{\theta}_{n}:=\widetilde{\theta}_{n}+\Phi_{n}(\widetilde{\theta}_{n})\mathcal{I}(\widetilde{\theta}_{n})^{-1}\Phi_{n}(\widetilde{\theta}_{n})^{\top}\partial_{\theta}\ell_{n}(\widetilde{\theta}_{n}), \ \ n\in\mathbb{N},
\end{equation*}
where $\partial_{\theta}\ell_{n}(\theta)$ is the score function and $\mathcal{I}(\theta)$ is the positive definite Fisher information matrix defined in Section~\ref{subsec:LAN}. 
\begin{theorem}\label{Theorem:OneStep}
	Assume the sequence of initial estimators $\{\widetilde{\theta}_{n}\}_{n\in\mathbb{N}}$ satisfies
	\begin{equation}\label{Assumption_IniEstimator}
		\Phi_{n}(\theta)^{-1}(\widetilde{\theta}_{n}-\theta)=O_{\mathbb{P}^{n}_{\theta}}(1)\ \ 
		\mbox{as $n\to\infty$.}
	\end{equation}
	Under the same assumptions in Equation~\ref{Thm:LAN}, the sequence of the one-step estimators $\{\widehat{\theta}_{n}\}_{n\in\mathbb{N}}$ is asymptotically normal and asymptotically efficient in the Fisher sense, i.e.
	\begin{align*}
		\mathcal{L}\left\{\Phi_{n}(\theta)^{-1}(\widehat{\theta}_{n}-\theta)\bigl|\mathbb{P}^{n}_{\theta}\right\}\to\mathcal{N}\left(0,\mathcal{I}(\theta)^{-1}\right)\ \ \mbox{as $n\to\infty$}.
	\end{align*}
\end{theorem}
The proof of Theorem~\ref{Theorem:OneStep} is given in Appendix~\ref{Sec:Proof-OneStep}. We also provide a condition for \eqref{Assumption_IniEstimator} in the following lemma, also proved in Appendix~\ref{Sec:Proof-OneStep}.
\begin{lemma}
\label{Lemma:OneStep}
	Suppose $\{\varphi_{n}(\theta)\}_{n\in\mathbb{N}}$ satisfies Assumption \ref{Assump:RateMat}. 
	Then for any estimators $\{\widetilde{\theta}_{n}\}_{n\in\mathbb{N}}$, \eqref{Assumption_IniEstimator} is equivalent to
	\begin{align*}
		\left|\sqrt{r_{n}^{1}(H)}(\widetilde{H}_{n}-H)\right|
		+\left|r_{n}^{2}(H)^{1/2}(\widetilde{\tau}_{n}-\tau)\right| \nonumber
		+\left|\sqrt{r_{n}^{1}(H)}\left(
		a_{n}(H)(\widetilde{H}_{n}-H)+\sigma^{-1}(\widetilde{\sigma}_{n}-\sigma)
		\right)\right|=O_{\mathbb{P}^{n}_{\theta}}(1).
	\end{align*}
\end{lemma}

\section{Numerical Analysis}
\label{sec:numerical}

In this section, we focus on the implementation of the previous estimators when $K=0$ and $\tau$ are known, we refer to Remarks~\ref{rem:Kgeq2} and \ref{rem:tauunknwon} for insights into what happens when $K>0$ or $\tau$ is unknown. The main issue with the procedure described in Section~\ref{Sec:estimation} is the iterative steps that create instability in the results even for relatively large sample sizes due to the slow convergence rate of the first guess estimator. Indeed, the convergence rate of the first estimator is $n^{\frac{1}{6}} \wedge n^{\frac{1}{3}(1-H)} \wedge n^{\frac{2}{3}\{(2H)\wedge 1\}}$ which is relatively small, in particular when $H \to 0$ or $H \to 1$. Therefore, $\widehat{H}_{n}^{(0)}$ is often far from the true value of $H$, see Figure \ref{figure:hist:first_H}. \\
\begin{figure}[htbp]
    \centering
    \begin{subfigure}[b]{0.3\textwidth}
        \includegraphics[width=\textwidth]{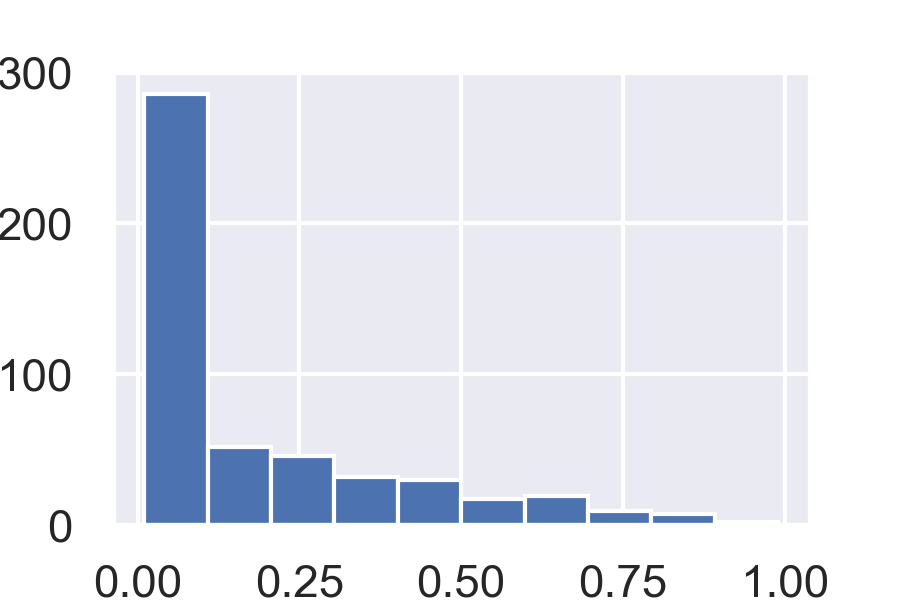}
        \caption{H = 0.15}
    	\label{figure:sub:firstH015}
    \end{subfigure}
    \hfill
    \begin{subfigure}[b]{0.3\textwidth}
        \includegraphics[width=\textwidth]{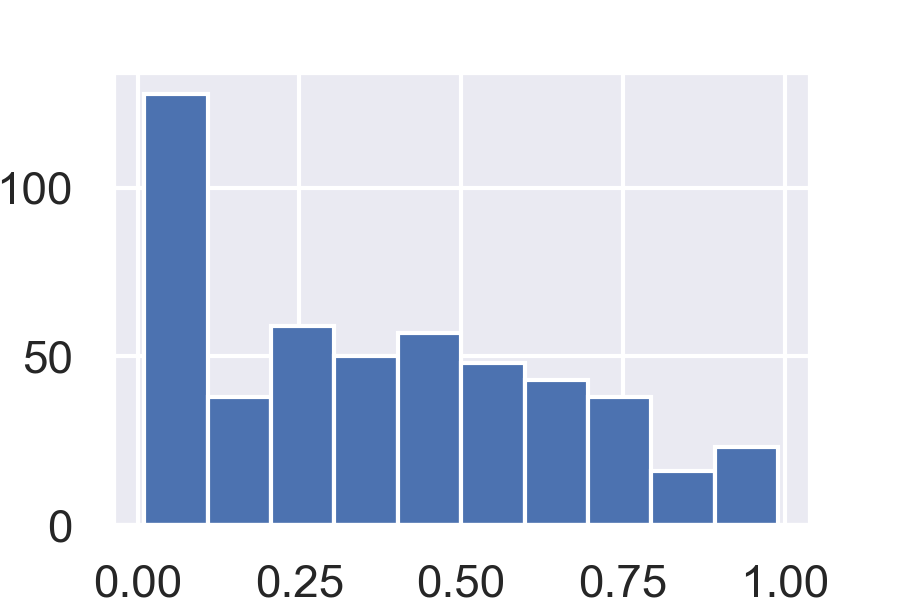}
        \caption{H = 0.5}
    \end{subfigure} 
    \hfill
    \begin{subfigure}[b]{0.3\textwidth}
        \includegraphics[width=\textwidth]{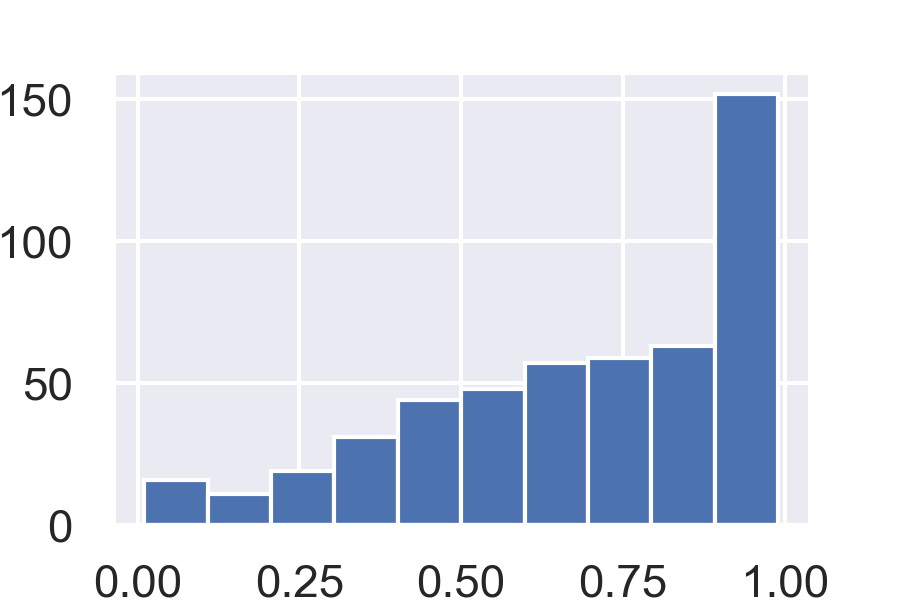}
        \caption{H = 0.85}
    \end{subfigure}
    \caption{Histograms representing the value of $\widehat{H}_{n}^{(0)}$ for different values of $H$ with $\tau = 0.1$, $n=2^{14}$, $\sigma=1$ and $K=0$, based on $500$ simulations. The huge bars near $0$ and $1$ are due to the projection on the interval $[H_-, H_+]$.}.
\label{figure:hist:first_H}
\end{figure}

From a theoretical perspective, these iterations are required for two main reasons. First, it is used to correct the bias induced by the different pre-averaging levels used in the numerator and the denominator. Indeed, we claimed in \eqref{eq:idea:guess_est} that 
$	{\widehat{Q}_{J,2p,n}^{(w)}} / {\widehat{Q}_{2J,p,n}^{(w)}}
	\approx 2^{2H}$ but from Proposition \ref{prop:deviation:Q}, a more accurate version of \eqref{eq:idea:guess_est} is
\begin{equation*}
	\frac{\widehat{Q}_{J,2p,n}^{(w)}}{\widehat{Q}_{2J,p,n}^{(w)}}
	\approx \frac{\kappa_{2p}^{(w)}(H) 2^{2H}}{\kappa_{p}^{(w)}(H)}.
\end{equation*}
The term $\bar{\kappa}_{p}^{(w)}(H)$ in the bias-correction procedure introduced in \eqref{eq:Def:correctedenergy} aims at correcting this effect. The second main reason for using the iterative procedure is to find adaptively the optimal pre-averaring level $p$. \\

To solve this problem, we propose in this Section a new selection rule for the pre-averaging level inspired directly by the level choice of \cite{gloter2007estimation}.
Let
\begin{align*}
\widehat{p^*} := \min \{ p \geq 2: \widehat{Q}_{J_{p,n},2p,n}^{(w)} \geq \nu_0 p^{-1} \}
\end{align*}
for some tuning parameter $\nu_0$. We do not prove that this choice is suitable to our model. However, following Theorem 2 in \cite{Chong-Hoffmann-Liu-Rosenbaum-Szymanski-2022-MinMax}, we can find $\nu_0$ such that with high probability $\widehat{p^*}$ is a good estimate of $p^{\mathrm{opt}}_n(H)$. When $H$ is small, this means that $\widehat{p^*}$ is small so that the bias induced by \eqref{eq:expec:energy} is dominating and we still do not get a good convergence rate. Let $\psi_p(x) := \frac{\kappa_{2p}^{(w)}(H) 2^{2H}}{\kappa_{p}^{(w)}(H)}$. Empirically, it appears that this function is increasing and has an inverse $\psi_p^{-1}$ continuously Lipschitz on any compact subset of $(0,1)$. Therefore in practice, a good estimator of $H$ is
\begin{align*}
\widehat{H^*} := \psi_{\widehat{p^*}}^{-1}\Big(\widehat{Q}_{\widehat{J^*},2\widehat{p^*},n}^{(w)} \, \big/ \, \widehat{Q}_{2\widehat{J^*},\widehat{p^*},n}^{(w)} \Big)
\end{align*}
where $\widehat{J^*} := J_{\widehat{p^*}, n}$. This expression depends on the tuning parameter $\nu_0$, which appears to have only little impact on the resulting estimator $\widehat{H^*}$, we found no significant difference while testing different values in the interval $[2, 6]$, see Figure \ref{fig:MAE_tuning}.  Therefore, we choose $\nu_0 = 2$ in the subsequent tests, which seems to be a good tradeoff. We refer to Figures \ref{figure:hist:best_H} and \ref{figure:table:best_H} for detailed results using $\widehat{H^*}$.

\begin{figure}
\centering
  \includegraphics[width=0.7\textwidth]{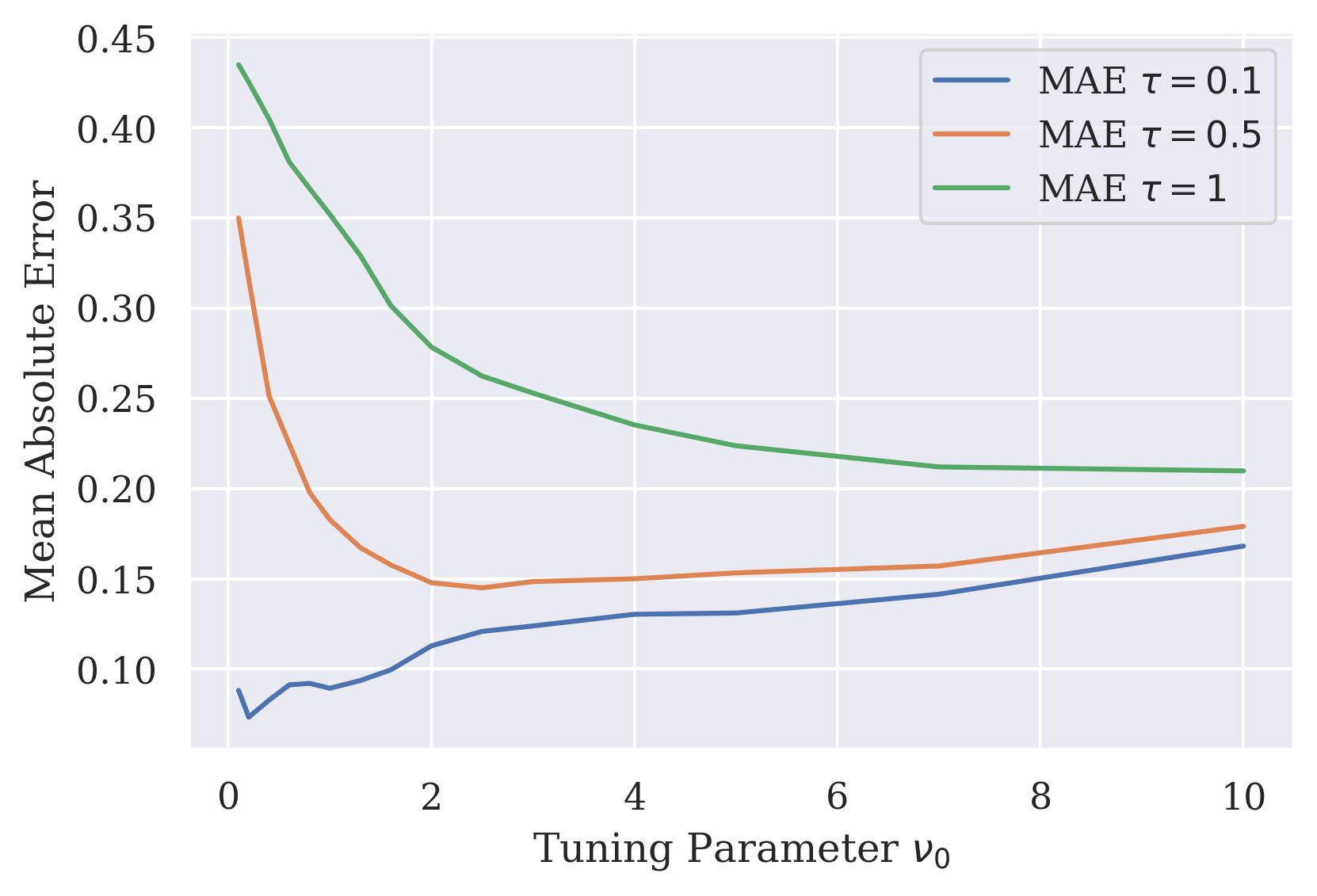}
  \caption{Mean Absolute Error for $\widehat{H^*}$ for different values of $\nu_0$. We used $500$ simulations and the parameters $n=2^{14}$, $\sigma=1$ and $K=0$.
  The different results for large $\tau$ and small $\tau$ illustrate that more pre-averaging is needed when the noise is larger.}
  \label{fig:MAE_tuning}
\end{figure}

\begin{figure}[htbp]
    \centering
    \begin{subfigure}[b]{0.3\textwidth}
        \includegraphics[width=\textwidth]{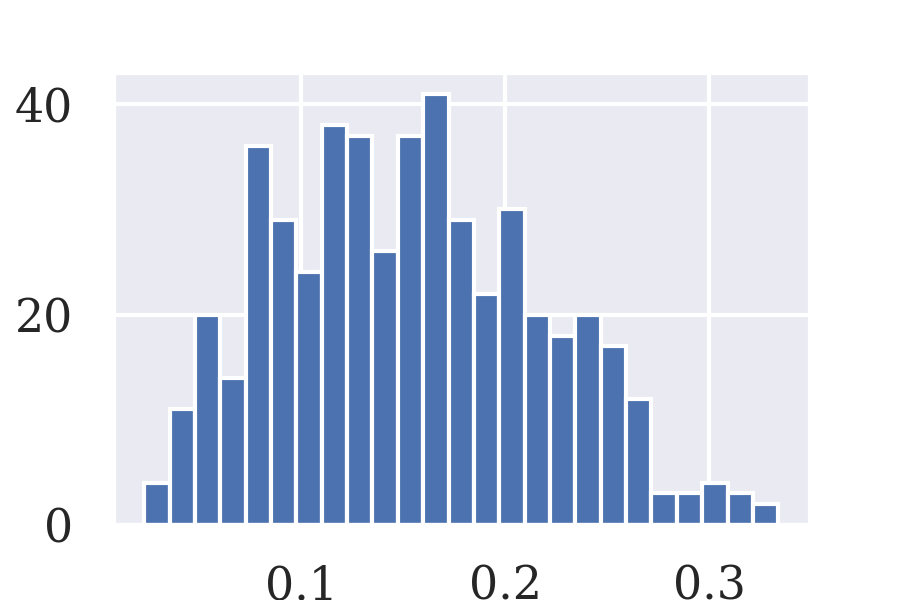}
        \caption{H = 0.15}
    	\label{figure:sub:newH015}
    \end{subfigure}
    \hfill
    \begin{subfigure}[b]{0.3\textwidth}
        \includegraphics[width=\textwidth]{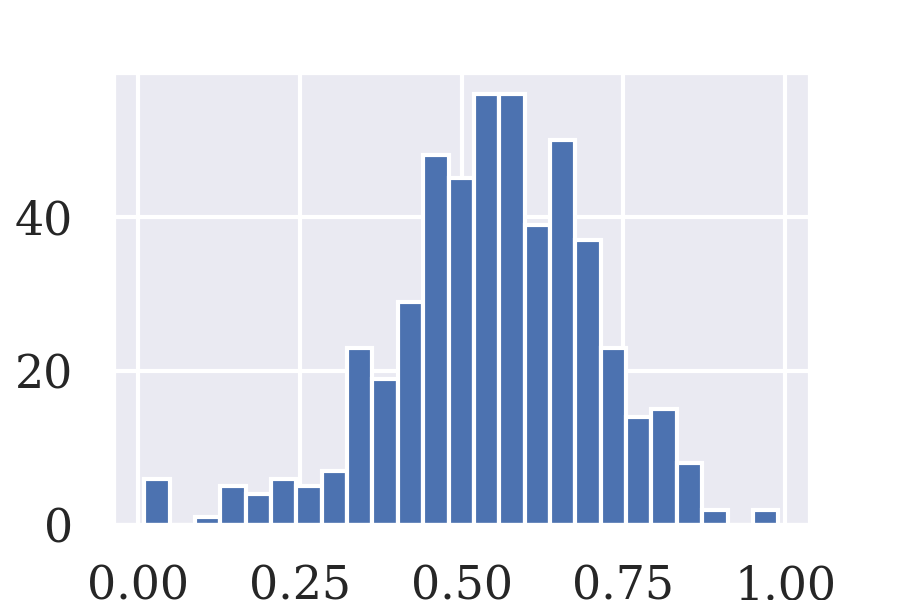}
        \caption{H = 0.5}
    	\label{figure:sub:newH05}
    \end{subfigure} 
    \hfill
    \begin{subfigure}[b]{0.3\textwidth}
        \includegraphics[width=\textwidth]{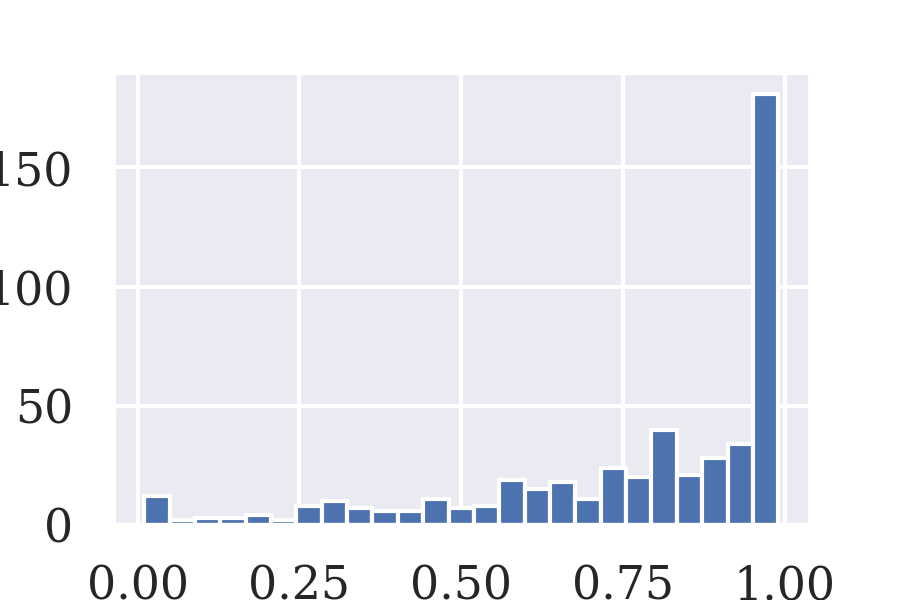}
        \caption{H = 0.85}
    \end{subfigure}
        \caption{Histograms representing the value of $\widehat{H^*}$ for different values of $H$ with $\tau = 0.1$, $n=2^{14}$, $\sigma=1$ and $K=0$, based on $500$ simulations. Note that the shape of Figures \ref{figure:sub:newH015} and \ref{figure:sub:newH05} are similar, although the scale is different due to the improved convergence rate for small $H$. The huge bars near $1$ for $H=0.85$ are due to the projection on the interval $[H_-, H_+]$ and the slow convergence rate in that case.}.
\label{figure:hist:best_H}
\end{figure}

\begin{figure}
\centering
\begin{tabular}{ | c|c|c|c|c|c | }
\hline $H$&$\tau$&Bias &Std dev&RMSE &NaN  \\\hline
\hline
$0.05$&$0.1$&$0.0030$&$0.0198$&$0.0201$&$0.00$ \%\\\hline
$0.05$&$0.5$&$0.0094$&$0.0304$&$0.0318$&$0.00$ \%\\\hline
$0.05$&$1.0$&$0.0516$&$0.0927$&$0.1061$&$0.00$ \%\\\hline
\hline
$0.15$&$0.1$&$0.0031$&$0.0641$&$0.0642$&$0.00$ \%\\\hline
$0.15$&$0.5$&$0.0202$&$0.0870$&$0.0893$&$0.00$ \%\\\hline
$0.15$&$1.0$&$0.0952$&$0.1936$&$0.2157$&$0.00$ \%\\\hline
\hline
$0.3$&$0.1$&$0.0014$&$0.1070$&$0.1070$&$0.00$ \%\\\hline
$0.3$&$0.5$&$0.0309$&$0.1654$&$0.1683$&$0.00$ \%\\\hline
$0.3$&$1.0$&$0.1432$&$0.3015$&$0.3338$&$0.00$ \%\\\hline
\hline
$0.5$&$0.1$&$0.0481$&$0.1586$&$0.1657$&$0.00$ \%\\\hline
$0.5$&$0.5$&$0.1069$&$0.2295$&$0.2532$&$0.00$ \%\\\hline
$0.5$&$1.0$&$0.1039$&$0.3640$&$0.3785$&$0.00$ \%\\\hline
\hline
$0.7$&$0.1$&$-0.0406$&$0.2391$&$0.2425$&$0.00$ \%\\\hline
$0.7$&$0.5$&$-0.0056$&$0.2788$&$0.2789$&$0.00$ \%\\\hline
$0.7$&$1.0$&$-0.1505$&$0.4256$&$0.4514$&$0.00$ \%\\\hline
\hline
$0.85$&$0.1$&$-0.0779$&$0.2558$&$0.2673$&$0.00$ \%\\\hline
$0.85$&$0.5$&$-0.0973$&$0.3175$&$0.3321$&$0.00$ \%\\\hline
$0.85$&$1.0$&$-0.3311$&$0.4421$&$0.5524$&$0.00$ \%\\\hline
\end{tabular}
        \caption{Detailed results of the bias and Error of $\widehat{H^*}$ for different values of $H$ and $\tau$ with $n=2^{14}$, $\sigma=1$ and $K=0$, based on $500$ simulations. }.
\label{figure:table:best_H}
\end{figure}

\begin{remark}
\label{rem:Kgeq2}
The implementation in the case $K>0$ with $\tau$ known is similar, though the computation time is much higher. This is due to the presence of the weight function $w$. Indeed, when $K=0$, we can take $w=1$ and therefore
\begin{align*}
\kappa_{p}^{(w)}(H)
&=p^{-2}\sum_{\ell_1=0}^{p-1}\sum_{\ell_2=0}^{p-1}
	w(\ell_1/p)w(\ell_2/p)\phi_{H}((\ell_1 - \ell_2) / p^{-1})
=p^{-1} \phi_{H}(0)
+
p^{-2}\sum_{\ell=1}^{p-1}
(p-\ell)
	\phi_{H}(\ell / p^{-1}).
\end{align*}
Using this simplification, the computation time of $\kappa_p^{(w)}$ is of the order $O(p)$.
When $K=0$, other choices of $w$ need to be considered, and this simplification does not hold anymore, so that the computation time is of the order $O(p^2)$, which makes the computation of $\psi_p^{-1}$ much slower in practice. Similar problems occur when computing $\gamma$.
\end{remark}

\begin{remark}
\label{rem:tauunknwon}
When $\tau$ is unknown, an efficient implementation could be obtained using another de-biasing method. Indeed, one could replace the energy levels $\widehat{Q}_{j,p,n}^{(w)}$ by
\begin{equation*}
    \widehat{D}_{J,p,n}^{(w)} := 
    \gamma_p^{(w)} \widehat{Q}_{J,2p,n}^{(w)}
    -
    \gamma_{2p}^{(w)} \widehat{Q}_{J,p,n}^{(w)}
\end{equation*}
so that \eqref{eq:expec:energy} gives
\begin{equation*}
\mathbb{E}_{\theta}^{n}\left[\widehat{D}_{J,p,n}^{(w)}\right]
=
J
\sigma^{2}(p/n)^{1+2H}
\Big(
    \gamma_p^{(w)} 2^{1+2H} \kappa_{2p}^{(w)}(H) 
    -
    \gamma_{2p}^{(w)} \kappa_{p}^{(w)}(H) 
\Big).
\end{equation*}
In order to use ratios of energy levels, we need to check that  $
    \gamma_p^{(w)} 2^{1+2H} \kappa_{2p}^{(w)}(H) 
    -
    \gamma_{2p}^{(w)} \kappa_{p}^{(w)}(H) 
$ does not vanish. This is not trivial and can only easily be checked asymptotically when $K=0$ and $w(x) = 1$ since these conditions imply that $\gamma_p^{(w)} = \gamma_{2p}^{(w)} = 6$ and we have $
    2^{1+2H} \kappa_{2p}^{(w)}(H) 
    -
    \kappa_{p}^{(w)}(H) 
\to (2^{1+2H} - 1) \kappa_{\infty}^{(w)}(H) \neq 0$. The same method can then be applied to build an estimator in that case. 
\end{remark}

\section*{Acknowledgments and fundings}

Gr\'egoire Szymanski gratefully acknowledges the financial support of the \'Ecole Polytechnique chairs {\it Deep Finance and Statistics} and {\it Machine Learning and Systematic Methods}.
Tetsuya Takabatake gratefully acknowledges financial support of JSPS KAKENHI (Grant Numbers JP19K23224 and JP23K13016) and the ANR project ``Efficient inference for large and high-frequency data'' (ANR-21-CE40-0021). 
This work was also supported by the Funds for the Development of Human Resources in Science and Technology under MEXT through the Home for Innovative Researchers and Academic Knowledge Users (HIRAKU) consortium.
The authors would like to thank Marc Hoffmann and Mathieu Rosenbaum for helpful remarks and discussions.

\bibliographystyle{alpha}
\bibliography{merged}

\appendix

\section{Analysis of the rate matrix}
\label{Appendix:Rate-Matrix}

\begin{lemma}\label{Lemma_RateMat}
	Under Assumption~\ref{Assump:RateMat}, the following properties hold:
	\begin{enumerate}[$(1)$]
		\item\label{Lemma_RateMat_elements_bdd} 
		Let $j=1,2$. 
		Then we have $|\varphi^{n}_{1j}(\theta)|\lesssim 1$ and $|\varphi^{n}_{2j}(\theta)|\lesssim |a_{n}(H)|$ uniformly on $\Theta$. 
        Moreover, under Assumption~\ref{Assumption:asymptotics_lan}, we also have $|\varphi^{n}_{2j}(\theta)|\lesssim\log{n}$ uniformly on $\Theta$.
		\item\label{Lemma_RateMat_positive-definite}
		The asymptotically non-degenerate property of the matrix $\Phi_{n}(\theta)$:
		\begin{align*}
		  \inf_{\theta\in\Theta}\left|\varphi^{n}_{11}(\theta)\varphi^{n}_{22}(\theta)-\varphi^{n}_{12}(\theta)\varphi^{n}_{21}(\theta)\right|\gtrsim 1.
		\end{align*}
		\item\label{Lemma_RateMat_UGR}
		Uniform growth rate of the normalizing matrix $\varphi_{n}(\theta)$:
		\begin{equation*}
            \sup_{\theta^{\prime}\in\Theta: \left\|\Phi_{n}(\theta)^{-1}(\theta^{\prime}-\theta)\right\|_{\mathbb{R}^{3}}\leq c}
			\left\|\Phi_{n}(\theta^{\prime})^{-1}\Phi_{n}(\theta)-I_{3}\right\|_{\mathrm{F}}\to 0\ \ \mbox{as $n\to\infty$}
		\end{equation*}
		for each $c>0$. 
	\end{enumerate}
\end{lemma}

\begin{proof}[Proof of Lemma~\ref{Lemma_RateMat}]
    Lemma~\ref{Lemma_RateMat}~\eqref{Lemma_RateMat_elements_bdd} and \eqref{Lemma_RateMat_positive-definite} follow directly from Assumption~\ref{Assump:RateMat} and
    \begin{align*}
        \det(\varphi_n(\theta)) &=
        \left|\varphi_{11}^{n}(\theta)\varphi_{22}^{n}(\theta)-\varphi_{12}^{n}(\theta)\varphi_{21}^{n}(\theta)\right| 
        \\
        &=(\sigma/2)\left|
        \varphi_{11}^{n}(\theta)\left(s_{22}^{n}(\theta)-2\varphi_{11}^{n}(\theta)a_{n}(H)\right)
        -\varphi_{12}^{n}(\theta)\left(s_{21}^{n}(\theta)-2\varphi_{12}^{n}(\theta)a_{n}(H)\right)\right| \\
        &=(\sigma/2)\left|
        \varphi_{11}^{n}(\theta)s_{22}^{n}(\theta)
        -\varphi_{12}^{n}(\theta)s_{21}^{n}(\theta)\right|.
    \end{align*}

Before proving Lemma~\ref{Lemma_RateMat}~\eqref{Lemma_RateMat_UGR}, we start with a few preliminary computations. Fix $c>0$ and let $\xi^{\prime}=(H^{\prime},\sigma^{\prime})$. Recall $s^{n}_{2j}(\theta)=2(\varphi^{n}_{1j}(\theta)a_{n}(H)+\sigma^{-1}\varphi^{n}_{2j}(\theta))$ for $j=1,2$ and set
\begin{align}
\label{eq:def:Lphi}
	L_{n}(\xi):=
	\begin{pmatrix}
	 	1& 0\\
	 	2a_{n}(H)& 2\sigma^{-1}
	 \end{pmatrix}\;\;\text{ and }\;\;
	 \widetilde{\varphi}_{n}(\theta):=L_{n}(\xi)\varphi_{n}(\theta)
	 =
	 \begin{pmatrix}
	 	\varphi^{n}_{11}(\theta)&\varphi^{n}_{12}(\theta)\\
	 	s^{n}_{21}(\theta)&s^{n}_{22}(\theta)
	 \end{pmatrix},
\end{align}
where $\varphi_{n}(\theta)$ is defined in Assumption~\ref{Assump:RateMat}. First note that 
\begin{align*}
    \widetilde{\varphi}_{n}(\theta)^{\top}\widetilde{\varphi}_{n}(\theta)
    &=
    \begin{pmatrix}
	    \varphi^{n}_{11}(\theta)&s^{n}_{21}(\theta)\\
	 	\varphi^{n}_{12}(\theta)&s^{n}_{22}(\theta)
	\end{pmatrix}
    \begin{pmatrix}
	    \varphi^{n}_{11}(\theta)&\varphi^{n}_{12}(\theta)\\
	 	s^{n}_{21}(\theta)&s^{n}_{22}(\theta)
	\end{pmatrix}
    \\
    &=
    \begin{pmatrix}
	    \varphi^{n}_{11}(\theta)^{2} +s^{n}_{21}(\theta)^{2}
        &\varphi^{n}_{11}(\theta)\varphi^{n}_{12}(\theta) +s^{n}_{21}(\theta)s^{n}_{22}(\theta) \\
	 	\varphi^{n}_{11}(\theta)\varphi^{n}_{12}(\theta) +s^{n}_{21}(\theta)s^{n}_{22}(\theta)
        &\varphi^{n}_{12}(\theta)^{2} +s^{n}_{22}(\theta)^{2}
	\end{pmatrix}
\end{align*}
so that $\|\widetilde{\varphi}_{n}(\theta)\|_{\mathrm{F}}
=
(\mathrm{Tr}[\widetilde{\varphi}_{n}(\theta)^{\top}\widetilde{\varphi}_{n}(\theta)])^{\frac{1}{2}}
\lesssim 1$
uniformly on $\Theta$ by Assumption~\ref{Assump:RateMat}.
Since $\varphi_{n}(\theta)^{-1}=\widetilde{\varphi}_{n}(\theta)^{-1}L_{n}(\xi)$, we have
\begin{equation*}
	\sqrt{r_{n}^{1}(H)}\varphi_{n}(\theta)^{-1}(\xi^{\prime}-\xi)
	=\widetilde{\varphi}_{n}(\theta)^{-1}
	\begin{pmatrix}
		\sqrt{r_{n}^{1}(H)}(H^{\prime}-H)\\
		2\sqrt{r_{n}^{1}(H)}a_{n}(H)(H^{\prime}-H)
		+2\sigma^{-1}\sqrt{r_{n}^{1}(H)}(\sigma^{\prime}-\sigma)
	\end{pmatrix}.
\end{equation*}
Multiplying both sides by $\varphi_{n}(\theta)$, and using the inequality $\|A\|_{\mathrm{op}}\leq \|A\|_{\mathrm{F}}$, we get
\begin{equation*}
\Bigg \|
	\begin{pmatrix}
		\sqrt{r_{n}^{1}(H)}(H^{\prime}-H)\\
		2\sqrt{r_{n}^{1}(H)}a_{n}(H)(H^{\prime}-H)
		+2\sigma^{-1}\sqrt{r_{n}^{1}(H)}(\sigma^{\prime}-\sigma)
	\end{pmatrix}
 \Bigg \|_{\mathbb{R}^2}
 \leq 
	\|\widetilde{\varphi}_{n}(\theta)\|_{\mathrm{F}}
    \Big \| \sqrt{r_{n}^{1}(H)}\varphi_{n}(\theta)^{-1}(\xi^{\prime}-\xi) \Big \|_{\mathbb{R}^2}
\end{equation*}
and therefore
\begin{align*}
	\sup_{\theta^{\prime}\in\Theta: \left\|\Phi_{n}(\theta)^{-1}(\theta^{\prime}-\theta)\right\|_{\mathbb{R}^{3}}\leq c}
	\left|\sqrt{r_{n}^{1}(H)}(H^{\prime}-H)\right|\vee 
	\left|\sqrt{r_{n}^{1}(H)}\left(
	a_{n}(H)(H^{\prime}-H)+\sigma^{-1}(\sigma^{\prime}-\sigma)
	\right)\right| \lesssim  1
\end{align*}
uniformly on $\Theta$. Then, using that
\begin{equation*}
\frac{\sqrt{r_{n}^{1}(H)}}{a_{n}(H)}(\sigma^{\prime}-\sigma)
=
\frac{\sigma \sqrt{r_{n}^{1}(H)}}{a_{n}(H)}\left(
	a_{n}(H)(H^{\prime}-H)+\sigma^{-1}(\sigma^{\prime}-\sigma)
	\right) - \sigma \sqrt{r_{n}^{1}(H)} (H^{\prime}-H),
\end{equation*}
we obtain
\begin{equation}\label{Tightness-Shrinkage}
	\sup_{\theta^{\prime}\in\Theta: \left\|\Phi_{n}(\theta)^{-1}(\theta^{\prime}-\theta)\right\|_{\mathbb{R}^{3}}\leq c}
	\Big|\sqrt{r_{n}^{1}(H)}(H^{\prime}-H)\Big| \vee 
	\Big|\frac{\sqrt{r_{n}^{1}(H)}}{a_{n}(H)}(\sigma^{\prime}-\sigma)\Big|
	\lesssim 1
\end{equation}
uniformly on $\Theta$.  Moreover, note that we have
\begin{equation*}
\left|\frac{r_{n}^{1}(H^{\prime})}{r_{n}^{1}(H)}-1\right|
=
\left|
\nu_{n}^{-\frac{2}{\Diamond(H^{\prime})}+\frac{2}{\Diamond(H)}}\Delta_{n}^{\frac{2H^{\prime}}{\Diamond(H^{\prime})}-\frac{2H}{\Diamond(H)}}
-
1
\right|
=
\left|
n^{b(H^{\prime})-b(H)}
-
1
\right|
\end{equation*}
for some continuously differentiable function $b(H)$ on $(0,1)$ by  Assumption~\ref{Assumption:asymptotics_lan}. Using the fundamental theorem of the calculus, we have
\begin{align*}
	|n^{b(H^{\prime})-b(H)}-1|
	&=\left|\exp\left((b(H^{\prime})-b(H))\log{n}\right)-1\right|\leq n^{|b(H^{\prime})-b(H)|}|b(H^{\prime})-b(H)|\log{n}.
\end{align*}
Moreover, \eqref{Tightness-Shrinkage} gives
\begin{equation}\label{n_power_bdd}
	\sup_{\theta^{\prime}\in\Theta: \left\|\Phi_{n}(\theta)^{-1}(\theta^{\prime}-\theta)\right\|_{\mathbb{R}^{3}}\leq c}
	\sqrt{r_{n}^{1}(H)}\left|b(H^{\prime})-b(H)\right|\lesssim 1\;\;\mbox{ and }\;\;
	\sup_{\theta^{\prime}\in\Theta: \left\|\Phi_{n}(\theta)^{-1}(\theta^{\prime}-\theta)\right\|_{\mathbb{R}^{3}}\leq c}n^{|b(H^{\prime})-b(H)|}\lesssim 1,
\end{equation}
and therefore we get
\begin{equation}\label{Ratio-Rate}
	\sup_{\theta^{\prime}\in\Theta: \left\|\Phi_{n}(\theta)^{-1}(\theta^{\prime}-\theta)\right\|_{\mathbb{R}^{3}}\leq c}
	\left|\frac{r_{n}^{1}(H^{\prime})}{r_{n}^{1}(H)}-1\right|
	\lesssim  1
\end{equation}
uniformly on $\Theta$. \\

We are now ready to prove Lemma~\ref{Lemma_RateMat}~\eqref{Lemma_RateMat_UGR}. First note that we can write
\begin{align*}
	\Phi_{n}(\theta^{\prime})^{-1}\Phi_{n}(\theta)
	=\mathrm{diag}\left(\frac{r_{n}^{1}(H^{\prime})}{r_{n}^{1}(H)},\frac{r_{n}^{1}(H^{\prime})}{r_{n}^{1}(H)},\frac{r_{n}^{2}(H^{\prime})}{r_{n}^{2}(H)}\right)^{\frac{1}{2}}
	\begin{pmatrix}
		\varphi_{n}(\theta^{\prime})^{-1}\varphi_{n}(\theta)&0_{2\times 1}\\
		0_{1\times 2}&1
	\end{pmatrix}
\end{align*}
so that using \eqref{Ratio-Rate}, Lemma~\ref{Lemma_RateMat}~\eqref{Lemma_RateMat_UGR} follows from
\begin{equation*}
	\sup_{\theta^{\prime}\in\Theta: \left\|\Phi_{n}(\theta)^{-1}(\theta^{\prime}-\theta)\right\|_{\mathbb{R}^{3}}\leq c} 
	\left\|\varphi_{n}(\theta^{\prime})^{-1}\varphi_{n}(\theta)-I_{2}\right\|_{\mathrm{F}}
	=o(1)\ \ \mbox{as $n\to\infty$.}
\end{equation*}

Using the definition of $\widetilde{\varphi}_{n}(\theta)$ in \eqref{eq:def:Lphi}, the triangle inequality and the sub-multiplicative property of the matrix norm $\|\cdot\|_{\mathrm{F}}$, we obtain
\begin{align}
\nonumber
\left\|\varphi_{n}(\theta^{\prime})^{-1}\varphi_{n}(\theta)-I_{2}\right\|_{\mathrm{F}}&\leq
	\left\|\widetilde{\varphi}_{n}(\theta^{\prime})^{-1}
	\left(L_{n}(\xi^{\prime})L_{n}(\xi)^{-1}-I_{2}\right)
	\widetilde{\varphi}_{n}(\theta)\right\|_{\mathrm{F}}
	+\left\|\widetilde{\varphi}_{n}(\theta^{\prime})^{-1}\widetilde{\varphi}_{n}(\theta)-I_{2}\right\|_{\mathrm{F}}\\
	&\leq 
	\left\|\widetilde{\varphi}_{n}(\theta^{\prime})^{-1}\right\|_{\mathrm{F}}
	\left\|L_{n}(\xi^{\prime})L_{n}(\xi)^{-1}-I_{2}\right\|_{\mathrm{F}}
	\left\|\widetilde{\varphi}_{n}(\theta)\right\|_{\mathrm{F}}
	+\left\|\widetilde{\varphi}_{n}(\theta^{\prime})^{-1}\widetilde{\varphi}_{n}(\theta)-I_{2}\right\|_{\mathrm{F}}.
\label{eq:twolines}
\end{align}
Since we have already proved  $\|\widetilde{\varphi}_{n}(\theta)\|_{\mathrm{F}}\lesssim 1$ uniformly on $\Theta$, we only need to focus on the terms $\|\widetilde{\varphi}_{n}(\theta)^{-1}\|_{\mathrm{F}}$, $\|\widetilde{\varphi}_{n}(\theta^{\prime})^{-1}\widetilde{\varphi}_{n}(\theta)-I_{2}\|_{\mathrm{F}}$ and $\|L_{n}(\xi^{\prime})L_{n}(\xi)^{-1}-I_{2}\|_{\mathrm{F}}$. 
Computing explicitely $\widetilde{\varphi}_{n}(\theta)^{\top}\widetilde{\varphi}_{n}(\theta)$, we can show that
\begin{align*}
    \left(\widetilde{\varphi}_{n}(\theta)\widetilde{\varphi}_{n}(\theta)^{\top}\right)^{-1}
    =\frac{1}{\mathrm{det}[\widetilde{\varphi}_{n}(\theta)]^{2}}
    \begin{pmatrix}
	    \varphi^{n}_{12}(\theta)^{2} +s^{n}_{22}(\theta)^{2}
        &-\varphi^{n}_{11}(\theta)\varphi^{n}_{12}(\theta) -s^{n}_{21}(\theta)s^{n}_{22}(\theta) \\
	 	\mathrm{sym.} 
        &\varphi^{n}_{11}(\theta)^{2} +s^{n}_{21}(\theta)^{2}
	\end{pmatrix}
\end{align*}
and therefore, using Assumption~\ref{Assump:RateMat} and Lemma~\ref{Lemma_RateMat}~\eqref{Lemma_RateMat_positive-definite}, we get
\begin{equation}
     \|\widetilde{\varphi}_{n}(\theta)^{-1}\|_{\mathrm{F}}
    \lesssim 1.
    \label{rate-mat:Phi-tilde-bounded}
\end{equation}

Similarly, we have
\begin{align*}
    \widetilde{\varphi}_{n}(\theta^{\prime})^{-1}\widetilde{\varphi}_{n}(\theta)
    &= \frac{1}{\det[\widetilde{\varphi}_{n}(\theta^{\prime})]}
    \begin{pmatrix}
	 	s^{n}_{22}(\theta^{\prime})&-\varphi^{n}_{12}(\theta^{\prime})\\
	 	-s^{n}_{21}(\theta^{\prime})&\varphi^{n}_{11}(\theta^{\prime})
	\end{pmatrix}
    \begin{pmatrix}
	    \varphi^{n}_{11}(\theta)&\varphi^{n}_{12}(\theta)\\
	 	s^{n}_{21}(\theta)&s^{n}_{22}(\theta)
	\end{pmatrix} \\
    &= \frac{1}{\det[\widetilde{\varphi}_{n}(\theta^{\prime})]}
    \begin{pmatrix}
	 	\varphi^{n}_{11}(\theta)s^{n}_{22}(\theta^{\prime}) -\varphi^{n}_{12}(\theta^{\prime})s^{n}_{21}(\theta)
        &s^{n}_{22}(\theta^{\prime})\varphi^{n}_{12}(\theta) -\varphi^{n}_{12}(\theta^{\prime})s^{n}_{22}(\theta) \\
	 	-s^{n}_{21}(\theta^{\prime})\varphi^{n}_{11}(\theta) +\varphi^{n}_{11}(\theta^{\prime})s^{n}_{21}(\theta)
        &-s^{n}_{21}(\theta^{\prime})\varphi^{n}_{12}(\theta) +\varphi^{n}_{11}(\theta^{\prime})s^{n}_{22}(\theta)
	\end{pmatrix}
\end{align*}
so that using Assumption~\ref{Assump:RateMat}, Lemma~\ref{Lemma_RateMat}~\eqref{Lemma_RateMat_positive-definite}, \eqref{Tightness-Shrinkage}, \eqref{Ratio-Rate} and Taylor's theorem, we have
\begin{equation}
    \sup_{\theta^{\prime}\in\Theta: \left\|\Phi_{n}(\theta)^{-1}(\theta^{\prime}-\theta)\right\|_{\mathbb{R}^{3}}\leq c} 
	\|\widetilde{\varphi}_{n}(\theta^{\prime})^{-1}\widetilde{\varphi}_{n}(\theta)-I_{2}\|_{\mathrm{F}} 
    \lesssim 1
    \label{rate-mat:Phi-tilde-error}.
\end{equation}

For the term $\|L_{n}(\xi^{\prime})L_{n}(\xi)^{-1}-I_{2}\|_{\mathrm{F}}$, we can write
\begin{align*}
	L_{n}(\xi^{\prime})L_{n}(\xi)^{-1}=
	\begin{pmatrix}
	 	1& 0\\
	 	2a_{n}(H^{\prime})& 2/\sigma^{\prime}
	 \end{pmatrix}
	 \begin{pmatrix}
	 	1& 0\\
	 	a_{n}(H)\sigma& \sigma/2
	 \end{pmatrix}
	 =
	 \begin{pmatrix}
	 	1& 0\\
	 	2a_{n}(H^{\prime})+2a_{n}(H)\sigma/\sigma^{\prime}& \sigma/\sigma^{\prime}
	 \end{pmatrix}.
\end{align*}
Then using that $a_{n}(H)=\log\Delta_{n}-\log(r_{n}^{1}(H)/n)$, \eqref{Tightness-Shrinkage}, \eqref{Ratio-Rate} and Taylor's theorem, we obtain
\begin{equation}
	\sup_{\theta^{\prime}\in\Theta: \left\|\Phi_{n}(\theta)^{-1}(\theta^{\prime}-\theta)\right\|_{\mathbb{R}^{3}}\leq c} 
	\left\|L_{n}(\xi^{\prime})L_{n}(\xi)^{-1}-I_{2}\right\|_{\mathrm{F}}
	\lesssim 1.
    \label{Ln-Error}
\end{equation}
We conclude by plugging \eqref{rate-mat:Phi-tilde-bounded}, \eqref{rate-mat:Phi-tilde-error} and \eqref{Ln-Error} into \eqref{eq:twolines}
\end{proof}

\section{Proof of Theorem~\ref{Thm:LAN}}\label{Sec:Proof-LAN}

\subsection{Notation and structure of the proof}
\label{Sec:Notation-Proof-LAN}
We first summarize the notation used in this proof.
For $m\in\mathbb{N}$, a $(m\times m)$-matrix $M$ and vectors $x,y\in\mathbb{R}^{m}$, we write $M[x,y]:=
        x^{\top}My$ and $M[x^{\otimes 2}]:=M[x,x]$.
For $i,j\in\{1,2\}$, we write
\begin{align*}
\begin{cases}
 A^{n}_{ij}(\theta):=
 r_{n}^{1}(H)^{-1}\Sigma_{n}(\Delta_{n}^{2H}\partial_{H}^{2-i}f_{\xi})\Sigma_{n}(h_{\theta}^{n})^{-1}\Sigma_{n}(\Delta_{n}^{2H}\partial_{H}^{2-j}f_{\xi})\Sigma_{n}(h_{\theta}^{n})^{-1},\\
 A^{n}_{33}(\theta):=
  r_{n}^{2}(H)^{-1}\Sigma_{n}(g_{\tau}^{n})\Sigma_{n}(h_{\theta}^{n})^{-1}\Sigma_{n}(g_{\tau}^{n})\Sigma_{n}(h_{\theta}^{n})^{-1},\\
 A^{n}_{j3}(\theta):=A^{n}_{3j}(\theta)=
\{r_{n}^{1}(H)r_{n}^{2}(H)\}^{-\frac{1}{2}}\Sigma_{n}(\Delta_{n}^{2H}\partial_{H}^{2-j}f_{\xi})\Sigma_{n}(h_{\theta}^{n})^{-1}\Sigma_{n}(g_{\tau}^{n})\Sigma_{n}(h_{\theta}^{n})^{-1},\\
 B^{n}_{ij}(\theta):=
 r_{n}^{1}(H)^{-1}\Sigma_{n}(\partial_{\sigma}^{i+j-2}\partial_{H}^{4-(i+j)}f_{\xi}^{n})\Sigma_{n}(h_{\theta}^{n})^{-1},\\ 
 B^{n}_{33}(\theta):=
 r_{n}^{2}(H)^{-1}\Sigma_{n}(\partial_{\tau}^{2}g_{\tau}^{n})\Sigma_{n}(h_{\theta}^{n})^{-1},\ \ 
 B^{n}_{j3}(\theta):=B^{n}_{3j}(\theta):=0. 
 \end{cases}
\end{align*}
Then, we define $\mathcal{F}^{n}(\theta)=(\mathcal{F}^{n}_{ij}(\theta))_{i,j=1,2,3}$ and $\mathcal{R}^{n}(\theta)=(\mathcal{R}^{n}_{ij}(\theta))_{i,j=1,2,3}$ by
\begin{align*}
\begin{cases}
    \mathcal{F}^{n}_{ij}(\theta):=
	\mathbf{Z}_{n}^{\top}\Sigma_{n}(h_{\theta}^{n})^{-1}A^{n}_{ij}(\theta)\mathbf{Z}_{n} -\frac{1}{2}\mathrm{Tr}\left[A^{n}_{ij}(\theta)\right],\\
	\mathcal{R}^{n}_{ij}(\theta):=
	\frac{1}{2}\left[\mathbf{Z}_{n}^{\top}\Sigma_{n}(h_{\theta}^{n})^{-1}B^{n}_{ij}(\theta)\mathbf{Z}_{n} -\mathrm{Tr}\left[B^{n}_{ij}(\theta)\right]\right]
 \end{cases}
\end{align*}
for $i,j=1,2,3$, and 
\begin{equation}\label{Def:overline-varphi}
	 \overline{\varphi}_{n}(\theta):=
	 \begin{pmatrix}
	 	1& 0\\
	 	2\log\Delta_{n}& 2\sigma^{-1}
	 \end{pmatrix}
	 \varphi_{n}(\theta)
	 =
	 \begin{pmatrix}
	 	\varphi^{n}_{11}(\theta)&\varphi^{n}_{12}(\theta)\\
	 	s^{n}_{21}(\theta)&s^{n}_{22}(\theta)
	 \end{pmatrix},
\end{equation}
where $s^{n}_{2j}(\theta):=2(\varphi^{n}_{1j}(\theta)\log{\Delta_{n}}+\sigma^{-1}\varphi^{n}_{2j}(\theta))$. 
Moreover, we define
\begin{align}
\label{Def:cn:dn}
\begin{cases}
	c_{n}(\theta):=c(\theta)^{-1}(\nu_{n}\Delta_{n}^{-H})^{\frac{2}{\Diamond(H)}}
	=c(\theta)^{-1}n/r_{n}^{1}(H),\\
	d_{n}(\theta):=\partial_{H}\log{c_{H}}+2\log{c_{n}(\theta)}. 
\end{cases}
\end{align}
Set $\Theta_{\ast}=(0,1)^{2q}\times\left\{[0,\infty)\times(0,\infty)\right\}^{2q}$. For a measurable set $A$ of $\mathbb{R}^{d}$, we write $\Indi_{A}$ the indicator function of $A$, that is $\Indi_{A}(x) = 1$ if $x \in A$ and $0$ otherwise. 
For $\underline{\xi}=(\underline{H},\underline{\sigma})$ with $\underline{H}=(H_{1},\cdots,H_{2q})\in(0,1)^{2q}$ and $\underline{\sigma}=(\sigma_{1},\cdots,\sigma_{2q})\in\left\{[0,\infty)\times(0,\infty)\right\}^{q}$, we define
\begin{align*}
\begin{cases}
	\psi_{q,\pm}(\underline{\xi}):=2\sum_{r=1}^{q}(H_{2r\pm 1}-H_{2r})_{+}\Indi_{(0,\infty)}(\sigma_{2r-1}),\\
	\rho_{q,\pm}(\underline{\xi}):=2\sum_{r=1}^{q}(H_{2r}-H_{2r\pm 1})_{+}\Indi_{(0,\infty)}(\sigma_{2r-1})
\end{cases}
\;\;\;\text{ and }
\;\;\;
\begin{cases}
	\psi_{q}(\underline{\xi}):=\psi_{q,-}(\underline{\xi})\vee\psi_{q,+}(\underline{\xi}),\\
	\rho_{q}(\underline{\xi}):=\rho_{q,-}(\underline{\xi})\vee\rho_{q,+}(\underline{\xi}).
\end{cases}
\end{align*}
Moreover, for $0 < \upsilon < 1$ fixed, we introduce the following restricted parameter space
\begin{align*}
\begin{cases}
	\Theta_{1}^{\underline{\xi}}(\iota):=\{\underline{\xi}\in(0,1)^{2q}\times\left\{[0,\infty)\times(0,\infty)\right\}^{q}:\psi_{q}(\underline{\xi})\leq 1-\upsilon\},\\
	\Theta_{1}(\iota):=\Theta_{1}^{\underline{\xi}}(\iota)\times\left\{[0,\infty)\times(0,\infty)\right\}^{2q}.
\end{cases}
\end{align*}
For $\underline{\theta}=(\underline{H},\underline{\sigma},\underline{\tau})\in \Theta_{\ast}$ with $\underline{H}=(H_{1},\cdots,H_{2q})$, $\underline{\sigma}=(\sigma_{1},\cdots,\sigma_{2q})$ and $\underline{\tau}=(\tau_{1},\cdots,\tau_{2q})$, we write $\theta_{r}=(H_{r},\sigma_{r},\tau_{r})$ for $r\in\{1,2,\cdots,2q\}$. \\

We are now ready to give the structure of the proof. 
Using the change of the variable and Taylor's theorem, we can show
\begin{align*}
	\log\frac{\mathrm{d}\mathbb{P}^{n}_{\theta+\Phi_{n}(\theta)u}}{\mathrm{d}\mathbb{P}^{n}_{\theta}}
	=u^{\top}\zeta_{n}(\theta)-\frac{1}{2}u^{\top}\mathcal{I}_{n}(\theta)u 
	+v_{n}(u,\theta),
\end{align*}
where the remainder term $v_{n}(u,\theta)$ is given by
\begin{equation}\label{Def:LAN-remainder}
	v_{n}(u,\theta):=
	\int_{0}^{1}(1-z)\left(
	\partial_{\theta}^{2}\ell_{n}\left(\theta+z\varphi_{n}(\theta)u\right)
	-\partial_{\theta}^{2}\ell_{n}(\theta)
	\right)
	\left[\left(\varphi_{n}(\theta)u\right)^{\otimes 2}\right]\,\mathrm{d}z.
\end{equation}
Therefore, the LAN property follows once we have verified \eqref{LAN:Score-CLT} and the estimates
\begin{equation}\label{LAN:ObsFisherInfo:Convergence}
	\mathcal{I}_{n}(\theta)=\mathcal{I}(\theta)+o_{\mathbb{P}^{n}_{\theta}}(1)\ \ \mbox{as $n\to\infty$},
\end{equation}
and
\begin{equation}\label{LAN:remainder}
	v_{n}(u,\theta)=o_{\mathbb{P}^{n}_{\theta}}(1)\ \ \mbox{as $n\to\infty$}
\end{equation}
for each $\theta\in\Theta$ and $u\in\mathbb{R}^{3}$. Estimates \eqref{LAN:ObsFisherInfo:Convergence} and \eqref{LAN:remainder} are proved in Section~\ref{Sec:Proof:LAN:ObsFisherInfo:Convergence} and~\ref{Sec:Proof:LAN:Remainder} respectively while \eqref{LAN:Score-CLT} follows from \eqref{LAN:ObsFisherInfo:Convergence} proceeding as for the proof of Lemma 2.5 of \cite{Cohen-Gamboa-Lacaux-Loubes-2013}. These proofs are based on the following result of the trace approximation, which is proved in Section~\ref{Section_Proof_LRR10_Thm5}.

\begin{proposition}\label{Prop:LRR10_Thm5}
		Let $q\in\mathbb{N}$ and $K \geq -1$. Suppose that for each $r\in\{1,2,\cdots,q\}$, $d_{r,\xi}(x)$ and $e_{r}(x)$ are periodic functions on $\mathbb{R}$ with period $2\pi$  satisfying
			\begin{equation}\label{assumption_k}
				\sup_{x\in[-\pi,\pi]\setminus\{0\}}|x|^{2H-1+v-\epsilon_{2}^{\prime}}|\partial_{x}^{v}d_{r,\xi}(x)|\lesssim 1,\ \ 
				\sup_{x\in[-\pi,\pi]\setminus\{0\}}|x|^{-2(K+1)+v}|\partial_{x}^{v}e_{r}(x)|<\infty
			\end{equation}
			uniformly on compact subsets of $(0,1)\times[0,\infty)^{2}$ for any $\epsilon_{2}^{\prime}>0$ and each $v\in\{0,1\}$. Then, for each $\epsilon_{1}^{\prime}>0$, we define $d_{r,\xi}^{n}(x)= d_{r,\xi}^{n,\epsilon_{1}^{\prime}}(x)=\Delta_{n}^{2H-\epsilon_{1}^{\prime}}d_{r,\xi}(x)$, $e_{r,\tau}^{n}(x)=\nu_{n}^{2}\tau^{2}e_{r}(x)$ and $k_{r,\theta}^{n}(x)=d_{r,\xi}^{n}(x)+e_{r,\tau}^{n}(x)$,
		Then we obtain
		\begin{equation*}
			\Delta_{n}^{\rho_{q}(\underline{\xi})+\epsilon_{1}}
			n^{-\psi_{q}(\underline{\xi})-\epsilon_{2}}
			\left|\mathrm{Tr}\left[\prod_{r=1}^{q}\Sigma_{n}(k_{r,\theta_{2r-1}}^{n})\Sigma_{n}(h_{\theta_{2r}}^{n})^{-1}\right]
	   		-\frac{n}{2\pi}\int_{-\pi}^{\pi}\prod_{r=1}^{q}\frac{k_{r,\theta_{2r-1}}^{n}(x)}{h_{\theta_{2r}}^{n}(x)}\,\mathrm{d}x\right|
			=o(1)
		\end{equation*}
		as $n\to\infty$ uniformly on compact subsets of $\Theta_{1}(\iota)$ for any $\epsilon_{1},\epsilon_{2}>0$ and $\iota\in(0,1)$.  
\end{proposition}

\subsection{Proof of \eqref{LAN:ObsFisherInfo:Convergence}}
\label{Sec:Proof:LAN:ObsFisherInfo:Convergence}
First, note that we have
    \begin{align*}
         \partial_{\theta_{j}}\ell_{n}(\theta)
         =-\frac{1}{2}\mathrm{Tr}[\Sigma_{n}(h_{\theta}^{n})^{-1}\Sigma_{n}(\partial_{\theta_{j}}h_{\theta}^{n})]
         +\frac{1}{2}\mathbf{Z}_{n}^{\top}\Sigma_{n}(h_{\theta}^{n})^{-1}\Sigma_{n}(\partial_{\theta_{j}}h_{\theta}^{n})\Sigma_{n}(h_{\theta}^{n})^{-1}\mathbf{Z}_{n}.
    \end{align*}
and
\begin{align*}		
    \partial_{\theta_{k}}\partial_{\theta_{j}}\ell_{n}(\theta)
		&=\frac{1}{2}\mathrm{Tr}[\Sigma_{n}(h_{\theta}^{n})^{-1}\Sigma_{n}(\partial_{\theta_{k}}h_{\theta}^{n})\Sigma_{n}(h_{\theta}^{n})^{-1}\Sigma_{n}(\partial_{\theta_{j}}h_{\theta}^{n})]
		-\frac{1}{2}\mathrm{Tr}[\Sigma_{n}(h_{\theta}^{n})^{-1}\Sigma_{n}(\partial_{\theta_{k}}\partial_{\theta_{j}}h_{\theta}^{n})]\\
		&\;\;\;\;-\mathbf{Z}_{n}^{\top}\Sigma_{n}(h_{\theta}^{n})^{-1}\Sigma_{n}(\partial_{\theta_{k}}h_{\theta}^{n})\Sigma_{n}(h_{\theta}^{n})^{-1}\Sigma_{n}(\partial_{\theta_{j}}h_{\theta}^{n})\Sigma_{n}(h_{\theta}^{n})^{-1}\mathbf{Z}_{n}\\
		&\;\;\;\;+\frac{1}{2}\mathbf{Z}_{n}^{\top}\Sigma_{n}(h_{\theta}^{n})^{-1}\Sigma_{n}(\partial_{\theta_{k}}\partial_{\theta_{j}}h_{\theta}^{n})\Sigma_{n}(h_{\theta}^{n})^{-1}\mathbf{Z}_{n}.
	\end{align*}
	Moreover, we can easily check that
	\begin{align}
 \label{eq:all_der_h}
	\begin{cases}
		\partial_{\sigma}h_{\theta}^{n}=
		\partial_{\sigma}f_{\xi}^{n}=2\sigma^{-1}f_{\xi}^{n},\\ 
		\partial_{\tau}h_{\theta}^{n}=
		\partial_{\tau}g_{\tau}^{n}=2\tau^{-1}g_{\tau}^{n},\\
		\partial_{\tau}^{2}h_{\theta}^{n}=2\tau^{-2}g_{\tau}^{n},\\ 
		\partial_{\sigma}^{2}h_{\theta}^{n}=
		2\sigma^{-2}f_{\xi}^{n}
  \end{cases}
  \;\; \text{ and } \;\;
		\begin{cases}
	\partial_{H}h_{\theta}^{n}=
		\partial_{H}f_{\xi}^{n}=
		f_{\xi}^{n}(2\log\Delta_{n}+\partial_{H}\log{f_{H}}), \\ 
  \partial_{H}^{2}h_{\theta}^{n}=
		f_{\xi}^{n}(2\log\Delta_{n}+\partial_{H}\log{f_{H}})^{2}
		+f_{\xi}^{n}(\partial_{H}^{2}\log{f_{H}}), \\
		\partial_{\sigma}\partial_{H}h_{\theta}^{n}=
		2\sigma^{-1}f_{\xi}^{n}(2\log\Delta_{n}+\partial_{H}\log{f_{H}}),
  \\ 
  		\partial_{\tau}\partial_{H}h_{\theta}^{n}=\partial_{\tau},\partial_{\sigma}h_{\theta}^{n}=0
   \end{cases}
	\end{align}
	so that we have
	\begin{align*}
 \begin{cases}
		\partial_{H}^{2}\ell_{n}(\theta)
		=r_{n}^{1}(H)\big\{-(2\log\Delta_{n})^{2}\mathcal{F}^{n}_{22}(\theta) -(2\log\Delta_{n})(\mathcal{F}^{n}_{12}(\theta)+\mathcal{F}^{n}_{21}(\theta)) -\mathcal{F}^{n}_{11}(\theta) +\mathcal{R}^{n}_{11}(\theta)\big\},
\\
\partial_{H}\partial_{\sigma}\ell_{n}(\theta)
		=r_{n}^{1}(H)\big\{-2\sigma^{-1}(2\log\Delta_{n})\mathcal{F}^{n}_{22}(\theta) -2\sigma^{-1}\mathcal{F}^{n}_{12}(\theta) +\mathcal{R}^{n}_{12}(\theta)\big\},
\\
\partial_{\sigma}^{2}\ell_{n}(\theta)		=r_{n}^{1}(H)\big\{-(2\sigma^{-1})^{2}\mathcal{F}^{n}_{22}(\theta) +\mathcal{R}^{n}_{22}(\theta)\big\},
\\
\partial_{H}\partial_{\tau}\ell_{n}(\theta)		=\sqrt{r_{n}^{1}(H)r_{n}^{2}(H)}\big\{-2\tau^{-1}(2\log\Delta_{n})\mathcal{F}^{n}_{23}(\theta) -2\tau^{-1}\mathcal{F}^{n}_{13}(\theta)\big\},
\\
\partial_{\sigma}\partial_{\tau}\ell_{n}(\theta)		=\sqrt{r_{n}^{1}(H)r_{n}^{2}(H)}\big\{(2\sigma^{-1})(2\tau^{-1})\mathcal{F}^{n}_{23}(\theta)\big\},
\\ 
		\partial_{\tau}^{2}\ell_{n}(\theta)
		=r_{n}^{2}(H)\big\{-(2\tau^{-1})^{2}\mathcal{F}^{n}_{33}(\theta) +\mathcal{R}^{n}_{33}(\theta)\big\}.
\end{cases}
\end{align*}
	Therefore, we obtain the following representation
	\begin{equation*}
		\partial_{\theta}^{2}\ell_{n}(\theta)
		=D_{n}(H)^{\frac{1}{2}}
		\left[
		-M_{n}(\theta)^{\top}\mathcal{F}^{n}(\theta)M_{n}(\theta)
		+\mathcal{R}^{n}(\theta)
		\right]
		D_{n}(H)^{\frac{1}{2}},
	\end{equation*}
	where
	\begin{equation*}
		D_{n}(H):=\mathrm{diag}\left(r_{n}^{1}(H),r_{n}^{1}(H),r_{n}^{2}(H)\right)\;\;\;\text{ and } \;\;\;
		M_{n}(\theta):=
		\begin{pmatrix}
			1&0&0 \\
			2\log\Delta_{n}&2\sigma^{-1}&0 \\
			0&0&2\tau^{-1}
		\end{pmatrix}.
	\end{equation*}
By definition of $D_n(H)$ and $M_n(H)$ and by Assumption \ref{Assump:RateMat}, we get
\begin{equation*}
    \mathcal{I}_{n}(\theta)
    =
	-\Phi_{n}(\theta)^{\top}\partial_{\theta}^{2}\ell_{n}(\theta)\Phi_{n}(\theta)
    =
		\begin{pmatrix}
			\overline{\varphi}_{n}(\theta)& 0_{2\times 1} \\
			0_{1\times 2}& 2\tau^{-1}
		\end{pmatrix}^{\top}
		\mathbb{E}_{\theta}^{n}[\mathcal{F}^{n}(\theta)]
		\begin{pmatrix}
			\overline{\varphi}_{n}(\theta)& 0_{2\times 1} \\
			0_{1\times 2}& 2\tau^{-1}
		\end{pmatrix}
		+\overline{\mathcal{R}}^{n}_{1}(\theta)
		+\overline{\mathcal{R}}^{n}_{2}(\theta),
\end{equation*}
	where
	\begin{align*}
		&\overline{\mathcal{R}}^{n}_{1}(\theta):=
		\begin{pmatrix}
			\overline{\varphi}_{n}(\theta)& 0_{2\times 1} \\
			0_{1\times 2}& 2\tau^{-1}
		\end{pmatrix}^{\top}
		(\mathcal{F}^{n}(\theta)-\mathbb{E}_{\theta}^{n}[\mathcal{F}^{n}(\theta)])
		\begin{pmatrix}
			\overline{\varphi}_{n}(\theta)& 0_{2\times 1} \\
			0_{1\times 2}& 2\tau^{-1}
		\end{pmatrix}
\end{align*}
and
\begin{align*}    
        &\overline{\mathcal{R}}^{n}_{2}(\theta):=
		-\begin{pmatrix}
    	\varphi_{n}(\theta)&0_{2\times 1} \\
    	0_{1\times 2}&1
    	\end{pmatrix}^{\top}
        \mathcal{R}^{n}(\theta)
        \begin{pmatrix}
    	\varphi_{n}(\theta)&0_{2\times 1} \\
    	0_{1\times 2}&1
    	\end{pmatrix}.
	\end{align*}
	Moreover, we have \begin{equation*}
	 \mathrm{Var}_{\theta}^{n}[\mathcal{F}^{n}_{ij}(\theta)]
	 =2\mathrm{Tr}\left[A^{n}_{ij}(\theta)^{2}\right]		\;\;\;\text{ and } \;\;\;
	 \mathrm{Var}_{\theta}^{n}[\mathcal{R}^{n}_{ij}(\theta)]
	 =\frac{1}{2}\mathrm{Tr}\left[B^{n}_{ij}(\theta)^{2}\right].
	\end{equation*}
We need to evaluate these quantities. 
Using the similar calculations used in the proofs of Propositions~\ref{Prop:Limit_Key} and \ref{Prop:FisherInfo-Cross-Terms}, which are stated in the end of this proof, and combining them with \eqref{eq:all_der_h}, Propositions~\ref{Prop:LRR10_Thm5}, Assumption~\ref{Assumption:asymptotics_lan} and using that $\mathbb{E}_\theta^n[\mathcal{R}_{n}(\theta)] = 0$ and $\max_{i,j=1,2}\{|\varphi_{ij}^{n}(\theta)|+|s^{n}_{2j}(\theta)|\}\lesssim|\log{n}|$ under Assumption~\ref{Assumption:asymptotics_lan}, we obtain
\begin{equation*}
		\Var_{\theta}^{n}[\mathcal{F}^{n}_{ij}(\theta)]
		+\Var_{\theta}^{n}[\mathcal{R}^{n}_{ij}(\theta)]
		=o(n^{\epsilon})\ \ \mbox{as $n\to\infty$}
\end{equation*}
and therefore
\begin{equation*}
		\mathcal{I}_{n}(\theta)=
		\begin{pmatrix}
			\overline{\varphi}_{n}(\theta)& 0_{2\times 1} \\
			0_{1\times 2}& 2\tau^{-1}
		\end{pmatrix}^{\top}
		\mathbb{E}_{\theta}^{n}[\mathcal{F}^{n}(\theta)]
		\begin{pmatrix}
			\overline{\varphi}_{n}(\theta)& 0_{2\times 1} \\
			0_{1\times 2}& 2\tau^{-1}
		\end{pmatrix}
		+o_{\mathbb{P}^{n}_{\theta}}(1)\ \ \mbox{as $n\to\infty$}.
\end{equation*}
Note that the right-hand side is no longer random, and therefore this implies that \eqref{LAN:ObsFisherInfo:Convergence}
is a consequence of
\begin{equation}
\label{UCFI1_Key1}
\begin{pmatrix}
			\overline{\varphi}_{n}(\theta)& 0_{2\times 1} \\
			0_{1\times 2}& 2\tau^{-1}
		\end{pmatrix}^{\top}
		\mathbb{E}_{\theta}^{n}[\mathcal{F}^{n}(\theta)]
		\begin{pmatrix}
			\overline{\varphi}_{n}(\theta)& 0_{2\times 1} \\
			0_{1\times 2}& 2\tau^{-1}
		\end{pmatrix}
		=\mathcal{I}(\theta)+o(1)\ \ \mbox{as $n\to\infty$}
	\end{equation}
	for each $\theta\in\Theta$. 
	For conciseness, we write
 \begin{align*}
  \begin{cases}
      &\widetilde{\mathcal{F}}^{n}_{1}(\theta):=
		\begin{pmatrix}
			\mathcal{F}^{n}_{11}(\theta)&\mathcal{F}^{n}_{12}(\theta) \\
			\mathcal{F}^{n}_{21}(\theta)&\mathcal{F}^{n}_{22}(\theta)
		\end{pmatrix},\\ 
	  &\widetilde{\mathcal{F}}^{n}_{2}(\theta):=
		\begin{pmatrix}
			\mathcal{F}^{n}_{13}(\theta)\\
			\mathcal{F}^{n}_{23}(\theta)
		\end{pmatrix},\\
  \end{cases}
  \;\; \text{ and } \;\;
  \begin{cases}
      &D_{n}(\theta):=
		\begin{pmatrix}
  		 1&0 \\
  		 -d_{n}(\theta)&1
  		\end{pmatrix},\\
	  &\varphi^{n}_{\dagger}(\theta):=
	 	D_{n}(\theta)^{-1}
	 	\overline{\varphi}_{n}(\theta)
	 	=
		\begin{pmatrix}
  		 1&0 \\
  		 d_{n}(\theta)&1
  		\end{pmatrix}
	 	\overline{\varphi}_{n}(\theta),
  \end{cases}\end{align*}
 where $\overline{\varphi}_{n}(\theta)$ and $d_{n}(\theta)$ are defined in \eqref{Def:overline-varphi} and \eqref{Def:cn:dn} respectively. Using that
     \begin{align*}
		&\overline{\varphi}_{n}(\theta)^{\top}\mathbb{E}_{\theta}^{n}[\widetilde{\mathcal{F}}^{n}_{1}(\theta)]\overline{\varphi}_{n}(\theta)
		=\varphi^{n}_{\dagger}(\theta)^{\top}D_{n}(\theta)^{\top}\mathbb{E}_{\theta}^{n}[\widetilde{\mathcal{F}}^{n}_{1}(\theta)]D_{n}(\theta)\varphi^{n}_{\dagger}(\theta),
		\\
		&\overline{\varphi}_{n}(\theta)^{\top}\mathbb{E}_{\theta}^{n}[\widetilde{\mathcal{F}}^{n}_{2}(\theta)]
		=\varphi^{n}_{\dagger}(\theta)^{\top}D_{n}(\theta)^{\top}\mathbb{E}_{\theta}^{n}[\widetilde{\mathcal{F}}^{n}_{2}(\theta)],
	\end{align*}
we see that the left-hand side of \eqref{UCFI1_Key1} can be rewritten as
	\begin{align*}
        \begin{pmatrix}
			\overline{\varphi}_{n}(\theta)& 0_{2\times 1} \\
			0_{1\times 2}& 2\tau^{-1}
		\end{pmatrix}^{\top}
		&\mathbb{E}_{\theta}^{n}[\mathcal{F}^{n}(\theta)]
		\begin{pmatrix}
			\overline{\varphi}_{n}(\theta)& 0_{2\times 1} \\
			0_{1\times 2}& 2\tau^{-1}
		\end{pmatrix} \\
        &=
		\begin{pmatrix}
			\overline{\varphi}_{n}(\theta)^{\top}\mathbb{E}_{\theta}^{n}[\widetilde{\mathcal{F}}^{n}_{1}(\theta)]\overline{\varphi}_{n}(\theta)
			&2\tau^{-1}\overline{\varphi}_{n}(\theta)^{\top}\mathbb{E}_{\theta}^{n}[\widetilde{\mathcal{F}}^{n}_{2}(\theta)] \\
			2\tau^{-1}\overline{\varphi}_{n}(\theta)^{\top}\mathbb{E}_{\theta}^{n}[\widetilde{\mathcal{F}}^{n}_{2}(\theta)] & (2\tau^{-1})^{2}\mathbb{E}_{\theta}^{n}[\mathcal{F}^{n}_{33}(\theta)]
		\end{pmatrix} \\
        &=
        \begin{pmatrix}
			\varphi^{n}_{\dagger}(\theta)& 0_{2\times 1} \\
			0_{1\times 2}& 2\tau^{-1}
		\end{pmatrix}^{\top}
		\begin{pmatrix}
			D_{n}(\theta)^{\top}\mathbb{E}_{\theta}^{n}[\widetilde{\mathcal{F}}^{n}_{1}(\theta)]D_{n}(\theta)
			&D_{n}(\theta)^{\top}\mathbb{E}_{\theta}^{n}[\widetilde{\mathcal{F}}^{n}_{2}(\theta)] \\
			D_{n}(\theta)^{\top}\mathbb{E}_{\theta}^{n}[\widetilde{\mathcal{F}}^{n}_{2}(\theta)]& \mathbb{E}_{\theta}^{n}[\mathcal{F}^{n}_{33}(\theta)]
		\end{pmatrix}
		\begin{pmatrix}
			\varphi^{n}_{\dagger}(\theta)& 0_{2\times 1} \\
			0_{1\times 2}& 2\tau^{-1}
		\end{pmatrix}.
	\end{align*}
Using that $\mathbb{E}_{\theta}^{n}[\mathcal{F}^{n}_{33}(\theta)]		=2^{-1}\mathrm{Tr}[A^{n}_{33}(\theta)]$, Proposition~\ref{Prop:LRR10_Thm5} and Assumption~\ref{Assumption:asymptotics_lan}, we also have
	\begin{align*}
		\mathbb{E}_{\theta}^{n}[\mathcal{F}^{n}_{33}(\theta)]
		=&\frac{n}{4\pi r_{n}^{2}(H)}\int_{-\pi}^{\pi}\left|\frac{g_{\tau}^{n}(\lambda)}{h_{\theta}^{n}(\lambda)}\right|^{2}\,\mathrm{d}\lambda +o(1)=\frac{1}{4\pi}\int_{-\pi}^{\pi}\left|\frac{g_{\tau}(\lambda)}{\nu_{n}^{-2}\Delta_{n}^{2H}f_{\xi}(\lambda)+g_{\tau}(\lambda)}\right|^{2}\,\mathrm{d}\lambda +o(1)
	\end{align*}
	as $n\to\infty$, so that we can show $\mathbb{E}_{\theta}^{n}[\mathcal{F}^{n}_{33}(\theta)]\to\mathcal{F}_{33}(\theta)$ as $n\to\infty$ using $\nu_{n}\Delta_{n}^{-H}\to\infty$ as $n\to\infty$ and Lebesgue's convergence theorem. 
	Note that, since we have $d_{n}(\theta)=d(\theta)-2\log(r_{n}^{1}(H)/n)$, Assumption~\ref{Assump:RateMat} implies
    \begin{align*}
        d_{n}(\theta)\varphi^{n}_{1j}(\theta)+s^{n}_{2j}(\theta)
        =d(\theta)\varphi^{n}_{1j}(\theta) 
        +2(\varphi^{n}_{1j}(\theta)a_{n}(H)+\sigma^{-1}\varphi^{n}_{2j}(\theta))
        \to d(\theta)\overline{\varphi}_{1j}(\theta)+\overline{s}_{2j}(\theta)
        \ \ \mbox{as $n\to\infty$}
    \end{align*}
    uniformly on $\Theta$ for each $j\in\{1,2\}$ so that we obtain 
	\begin{align*}
		\varphi^{n}_{\dagger}(\theta)=
	 	\begin{pmatrix}
	 	\varphi^{n}_{11}(\theta) &\varphi^{n}_{12}(\theta)\\
	 	d_{n}(\theta)\varphi^{n}_{11}(\theta)+s^{n}_{21}(\theta)&d_{n}(\theta)\varphi^{n}_{12}(\theta)+s^{n}_{22}(\theta)
		\end{pmatrix}
		\to\overline\varphi(\theta)\ \ \mbox{as $n\to\infty$.}
	\end{align*}
	Therefore, \eqref{UCFI1_Key1} follows once we have proved the following results:
	\begin{equation}\label{Fisher-Convergence-Key1}
		\left[D_{n}(\theta)^{\top}\mathbb{E}_{\theta}^{n}[\widetilde{\mathcal{F}}^{n}_{1}(\theta)]D_{n}(\theta)\right]_{ij}
		\to\mathcal{F}_{ij}(\theta)\;\;\text{ and }\;\;
		\left[D_{n}(\theta)^{\top}\mathbb{E}_{\theta}^{n}[\widetilde{\mathcal{F}}^{n}_{2}(\theta)]\right]_{j3}
		\to\mathcal{F}_{j3}(\theta)
	\end{equation}
	as $n\to\infty$ for each $i,j\in\{1,2\}$. 
	Set $k_{\theta,j}^{n}(\lambda)=f_{\xi}^{n}(\lambda)\left(-d_{n}(\theta)+\partial_{H}\log{f_{H}(\lambda)}\right)^{2-j}$. 
	First note that, since we have $\mathbb{E}_{\theta}^{n}[\mathcal{F}^{n}_{ij}(\theta)]
		=2^{-1}\mathrm{Tr}[A^{n}_{ij}(\theta)]$ for $i,j\in\{1,2,3\}$, we can rewrite
	\begin{align*}
		D_{n}(\theta)^{\top}\mathbb{E}_{\theta}^{n}[\widetilde{\mathcal{F}}^{n}_{1}(\theta)]D_{n}(\theta)
		&=
		\begin{pmatrix}
			\mathbb{E}_{\theta}^{n}[\mathcal{F}^{n}_{11}(\theta)]-2d_{n}(\theta)\mathbb{E}_{\theta}^{n}[\mathcal{F}^{n}_{12}(\theta)]+d_{n}(\theta)^{2}\mathbb{E}_{\theta}^{n}[\mathcal{F}^{n}_{22}(\theta)]&\mathrm{sym.}\\
            \mathbb{E}_{\theta}^{n}[\mathcal{F}^{n}_{21}(\theta)]-d_{n}(\theta)\mathbb{E}_{\theta}^{n}[\mathcal{F}^{n}_{22}(\theta)] 
			&\mathbb{E}_{\theta}^{n}[\mathcal{F}^{n}_{22}(\theta)]
		\end{pmatrix}\\
		&=
		\begin{pmatrix}
			\displaystyle
			\frac{1}{2r_{n}^{1}(H)}
			\mathrm{Tr}\left[\Sigma_{n}(k_{\theta,i}^{n})\Sigma_{n}(h_{\theta}^{n})^{-1}\Sigma_{n}(k_{\theta,j}^{n})\Sigma_{n}(h_{\theta}^{n})^{-1}\right]
		\end{pmatrix}_{i,j=1,2},
	\end{align*}
	and
	\begin{align*}
		D_{n}(\theta)^{\top}\mathbb{E}_{\theta}^{n}[\widetilde{\mathcal{F}}^{n}_{2}(\theta)]
		&=
		\begin{pmatrix}
			\mathbb{E}_{\theta}^{n}[\mathcal{F}^{n}_{13}(\theta)]-d_{n}(\theta)\mathbb{E}_{\theta}^{n}[\mathcal{F}^{n}_{23}(\theta)]\\
			\mathbb{E}_{\theta}^{n}[\mathcal{F}^{n}_{23}(\theta)]
		\end{pmatrix}\\
		&=
		\begin{pmatrix}
			\displaystyle
			\frac{1}{2\sqrt{r_{n}^{1}(H)r_{n}^{2}(H)}}
			\mathrm{Tr}\left[\Sigma_{n}(k_{\theta,i}^{n})\Sigma_{n}(h_{\theta}^{n})^{-1}\Sigma_{n}(g_{\tau}^{n})\Sigma_{n}(h_{\theta}^{n})^{-1}\right]
		\end{pmatrix}_{i=1,2}.
	\end{align*}
	Moreover, Proposition~\ref{Prop:LRR10_Thm5} and Assumption~\ref{Assumption:asymptotics_lan} give
	\begin{equation*}
		\left[D_{n}(\theta)^{\top}\mathbb{E}_{\theta}^{n}[\widetilde{\mathcal{F}}^{n}_{1}(\theta)]D_{n}(\theta)\right]_{ij}
        =
		\frac{n}{2\pi r_{n}^{1}(H)}\int_{0}^{\pi}\frac{f_{\xi}^{n}(\lambda)^{2}\left(-d_{n}(\theta)+\partial_{H}\log{f_{H}(\lambda)}\right)^{4-(i+j)}}{h_{\theta}^{n}(\lambda)^{2}}\,\mathrm{d}\lambda + o(1)
	\end{equation*}
	as $n\to\infty$ for each $i,j=1,2$, and
	\begin{equation*}
		\left[D_{n}(\theta)^{\top}\mathbb{E}_{\theta}^{n}[\widetilde{\mathcal{F}}^{n}_{2}(\theta)]\right]_{j} 
        =
		\frac{n}{2\pi \sqrt{r_{n}^{1}(H)r_{n}^{2}(H)}}\int_{0}^{\pi}\frac{g_{\tau}^{n}(\lambda)f_{\xi}^{n}(\lambda)\left(-d_{n}(\theta)+\partial_{H}\log{f_{H}(\lambda)}\right)^{2-j}}{h_{\theta}^{n}(\lambda)^{2}}\,\mathrm{d}\lambda + o(1)
	\end{equation*}
	as $n\to\infty$ for each $j=1,2$. 
	\eqref{Fisher-Convergence-Key1} eventually
	follows from these two identities and the following propositions, whose proofs are delayed until Appendix~\ref{Section_Proof_Prop_Limit_Key}.

\begin{proposition}\label{Prop:Limit_Key}
	Let $m\in\{0,1,2\}$ and $K \geq -1$. 
	For any $H\in(0,1)$, we obtain
    \begin{align*}
		\lim_{n\to\infty}
		\frac{n}{r_{n}^{1}(H)}\int_{0}^{\pi}\frac{f_{\xi}^{n}(\lambda)^{2}\left(-d_{n}(\theta)+\partial_{H}\log{f_{H}(\lambda)}\right)^{m}}{h_{\theta}^{n}(\lambda)^{2}}\,\mathrm{d}\lambda
		=c(\theta)\int_{0}^{\infty}\frac{\left(-2\log\mu\right)^{m}}{\left(1+\mu^{\Diamond(H)}\right)^{2}}\,\mathrm{d}\mu.
	\end{align*}
\end{proposition}
\begin{proposition}\label{Prop:FisherInfo-Cross-Terms}
	Let $m\in\{1,2\}$ and $K \geq -1$. 
	For any $H\in(0,1)$, we obtain
	\begin{align*}
		\lim_{n\to\infty}
		\frac{n}{\sqrt{r_{n}^{1}(H)r_{n}^{2}(H)}}
		\int_{0}^{\pi}\frac{g_{\tau}^{n}(\lambda)f_{\xi}^{n}(\lambda)\left|-d_{n}(\theta)+\partial_{H}\log{f_{H}(\lambda)}\right|^{m-1}}{h_{\theta}^{n}(\lambda)^{2}}\,\mathrm{d}\lambda=0.
	\end{align*}
\end{proposition}

\subsection{Proof of \eqref{LAN:remainder}}
\label{Sec:Proof:LAN:Remainder}
We first introduce some notation used in this proof.  
Let $m\in\{1,2,3,4\}$ and $q\in\mathbb{N}$. Consider functions $d_{j,\xi}^{n}$ and $e_{j,\tau}^{n}$ satisfying \eqref{assumption_k} for each $j\in\{1,2,3,4\}$. 
Set $k_{j,\theta}^{n,0}:=d_{j,\xi}^{n}$ and $k_{j,\theta}^{n,1}:=e_{j,\tau}^{n}$ for $j\in\{1,2,3,4\}$. 
For conciseness, we write 
$\mathbf{W}_{n}(\theta^{\prime}):=\Sigma_{n}(h_{\theta^{\prime}}^{n})^{-\frac{1}{2}}\mathbf{Z}_{n}$ 
and $S_{n}(k_{j,\theta^{\prime}}^{n}):=\Sigma_{n}(h_{\theta^{\prime}}^{n})^{-\frac{1}{2}}\Sigma_{n}(k_{j,\theta^{\prime}}^{n})\Sigma_{n}(h_{\theta^{\prime}}^{n})^{-\frac{1}{2}}$. 
Moreover, for $\underline{\upsilon}^{m}=(\upsilon_{1},\cdots,\upsilon_{m})\in\{0,1\}^{m}$, we write $|\underline{\upsilon}^{m}|:=\sum_{j=1}^{m}\upsilon_{j}$ and
\begin{equation}\label{def_rn_bar}
	\overline{r}_{n}(H)\equiv\overline{r}_{n}^{m,\underline{\upsilon}^{m}}(H):=\left\{
	\begin{array}{ll}
		r_{n}^{2}(H)&\mbox{if $m-|\underline{\upsilon}^{m}|=0$ with $m\in\{1,2,3,4\}$,}\\
		\sqrt{r_{n}^{1}(H)r_{n}^{2}(H)}&\mbox{if $m-|\underline{\upsilon}^{m}|=1$ with $m\in\{2,3,4\}$,}\\
		r_{n}^{1}(H)&\mbox{otherwise.}
	\end{array}
	\right.
\end{equation}
For conciseness, we write $\rho_{n,1}\equiv\rho_{n,1}(u,\theta):=r_{n}^{1}(H)^{-\frac{1}{2}}\|\Phi_{n}(\theta)\|_{\mathrm{F}}\|u\|_{\mathbb{R}^{3}}$, $\rho_{n,2}\equiv\rho_{n,2}(u,H):=r_{n}^{2}(H)^{-\frac{1}{2}}\|u\|_{\mathbb{R}^{3}}$. 

Recall that $v_{n}(u,\theta)$ is defined in \eqref{Def:LAN-remainder}. 
First note that
\begin{align*}
	|v_{n}(u,\theta)|
	&\leq\sup_{\theta^{\prime}=(\xi^{\prime},\tau^{\prime})\in
	{B}(\xi,\rho_{n,1})\times
	{B}(\tau,\rho_{n,2})}
	\left|\left(
	\partial_{\theta}^{2}\ell_{n}(\theta^{\prime})
	-\partial_{\theta}^{2}\ell_{n}(\xi,\tau^{\prime})
	\right)
	\left[\left(\varphi_{n}(\theta)u\right)^{\otimes 2}\right]\right|\\
	&\quad+\sup_{\tau^{\prime}\in{B}(\tau,\rho_{n,2})}
	\left|\left(
	\partial_{\theta}^{2}\ell_{n}(\xi,\tau^{\prime})
	-\partial_{\theta}^{2}\ell_{n}(\theta)
	\right)
	\left[\left(\varphi_{n}(\theta)u\right)^{\otimes 2}\right]\right|.
\end{align*}
so that using Chebyshev's inequality, we know that for any $M>0$ and $q\geq 1$, $\mathbb{P}^{n}_{\theta}\left[|v_{n}(u,\theta)|>M\right]$ is bounded above by
\begin{equation}
\begin{split}
	\left|\frac{2}{M}\right|^{2q}
	\mathbb{E}_{\theta}^{n}\bigg[
	&\sup_{\theta^{\prime}=(\xi^{\prime},\tau^{\prime})\in
	{B}(\xi,\rho_{n,1})\times
	{B}(\tau,\rho_{n,2})}
	\left|\left(
	\partial_{\theta}^{2}\ell_{n}(\theta^{\prime})
	-\partial_{\theta}^{2}\ell_{n}(\xi,\tau^{\prime})
	\right)
	\left[\left(\varphi_{n}(\theta)u\right)^{\otimes 2}\right]\right|^{2q}\bigg] 
	\\
	&+\left|\frac{2}{M}\right|^{2q}
	\mathbb{E}_{\theta}^{n}\bigg[
	\sup_{\tau^{\prime}\in{B}(\tau,\rho_{n,2})}
	\left|\left(
	\partial_{\theta}^{2}\ell_{n}(\xi,\tau^{\prime})
	-\partial_{\theta}^{2}\ell_{n}(\theta)
	\right)
	\left[\left(\varphi_{n}(\theta)u\right)^{\otimes 2}\right]\right|^{2q}
	\bigg]. 
	\end{split}
\label{LAN:LAN:remainder-Chebyshev}
\end{equation}
We now use the following Sobolev inequality that can be found in the lecture note by \cite{Driver-2003-LectureNote}, see Corollaries 25.10 and 25.11.
\begin{lemma}[Sobolev Inequality]\label{lemma:Sobolev}
    Let $d\in\mathbb{N}$, $\Theta_{\ast}$ be an open subset of $\mathbb{R}^{d}$ and $\{(\mathcal{X}_{n},\mathcal{A}_{n},\mathbb{P}^{n}_{\ast})\}_{n\in\mathbb{N}}$ be a sequence of complete probability spaces.
    Assume that for each $n\in\mathbb{N}$, $\{u_{n}(\theta,\cdot)\}_{\theta\in\Theta_{\ast}}$ is a continuous stochastic process defined on $(\mathcal{X}_{n},\mathcal{A}_{n},\mathbb{P}^{n}_{\ast})$ such that for $\mathbb{P}^{n}_{\ast}$-a.s. $\omega_{n}\in\mathcal{X}_{n}$, $\theta\mapsto u_{n}(\theta,\omega_{n})$ is a continuously differentiable function on $\Theta_{\ast}$. 
    Then for any $m\in\mathbb{N}$ satisfying $m>d$, there exists a constant $C:=C(m,d)>0$ such that for any $n\in\mathbb{N}$, $\theta\in\Theta_{\ast}$ and $\epsilon\in(0,1)$ satisfying $B(\theta,\epsilon)\subset\Theta_{\ast}$,
    \begin{align*}
        &\sup_{\theta^{\prime}\in B(\theta,\epsilon)}\left|u_{n}(\theta^{\prime},\omega_{n})-u_{n}(\theta,\omega_{n})\right|^{m}
        \leq C\epsilon^{m-d}
        \int_{B(\theta,\epsilon)}
        \left\|\partial_{\theta}u_{n}(\theta^{\prime},\omega_{n})\right\|_{\mathbb{R}^{d}}^{m}\,\mathrm{d}\theta^{\prime}, \\
        &\sup_{\theta^{\prime}\in B(\theta,\epsilon)}\left|u_{n}(\theta^{\prime},\omega_{n})\right|^{m}
        \leq C\left[
        \epsilon^{-d}
        \int_{B(\theta,\epsilon)}
        \left|u_{n}(\theta^{\prime},\omega_{n})\right|^{m}\,\mathrm{d}\theta^{\prime}
        +
        \epsilon^{m-d}\int_{B(\theta,\epsilon)}
        \left\|\partial_{\theta}u_{n}(\theta^{\prime},\omega_{n})\right\|_{\mathbb{R}^{d}}^{m}\,\mathrm{d}\theta^{\prime}
        \right]
    \end{align*}
    hold for $\mathbb{P}^{n}_{\ast}$-a.s. $\omega_{n}\in\mathcal{X}_{n}$.
\end{lemma}
For the first term of \eqref{LAN:LAN:remainder-Chebyshev}, Lemma~\eqref{lemma:Sobolev} and Fubini's theorem give that for any $2q>2$, 
\begin{align*}
	\mathbb{E}_{\theta}^{n}\bigg[&
	\sup_{\theta^{\prime}=(\xi^{\prime},\tau^{\prime})\in
	{B}(\xi,\rho_{n,1})\times
	{B}(\tau,\rho_{n,2})}
	\left|\left(
	\partial_{\theta}^{2}\ell_{n}(\theta^{\prime})
	-\partial_{\theta}^{2}\ell_{n}(\xi,\tau^{\prime})
	\right)
	\left[\left(\varphi_{n}(\theta)u\right)^{\otimes 2}\right]\right|^{2q}\bigg] \\
	&\;\;\;\;\;\;\;\;\;\;\;\;\lesssim\left(\rho_{n,1}\right)^{2q-2}
	\int_{{B}(\xi,\rho_{n,1})}
	\mathbb{E}_{\theta}^{n}\left[
	\sup_{\tau^{\prime}\in{B}(\tau,\rho_{n,2})}
	\left|\partial_{\xi}\partial_{\theta}^{2}\ell_{n}(\theta^{\prime})\left[\left(\varphi_{n}(\theta)u\right)^{\otimes 2}\right]\right|^{2q}
	\right]\,\mathrm{d}\xi^{\prime} \\
	&\;\;\;\;\;\;\;\;\;\;\;\;\lesssim\left(\rho_{n,1}\right)^{2q}
	\sup_{\xi^{\prime}\in{B}(\xi,\rho_{n,1})}
	\mathbb{E}_{\theta}^{n}\left[
	\sup_{\tau^{\prime}\in{B}(\tau,\rho_{n,2})}
	\left|\partial_{\xi}\partial_{\theta}^{2}\ell_{n}(\theta^{\prime})\left[\left(\varphi_{n}(\theta)u\right)^{\otimes 2}\right]\right|^{2q}
	\right]
\end{align*}
and 
\begin{align*}
	\mathbb{E}_{\theta}^{n}\bigg[
	\sup_{\tau^{\prime}\in{B}(\tau,\rho_{n,2})}&
	\left|\partial_{\xi}\partial_{\theta}^{2}\ell_{n}(\theta^{\prime})\left[\left(\varphi_{n}(\theta)u\right)^{\otimes 2}\right]\right|^{2q}
	\bigg]\\
	&\lesssim\left(\rho_{n,2}\right)^{-1}
	\int_{{B}(\tau,\rho_{n,2})}
	\mathbb{E}_{\theta}^{n}\left[
	\left|\partial_{\xi}\partial_{\theta}^{2}\ell_{n}(\theta^{\prime})\left[\left(\varphi_{n}(\theta)u\right)^{\otimes 2}\right]\right|^{2q}
	\right]\,\mathrm{d}\tau^{\prime} \\
	&\quad+\left(\rho_{n,2}\right)^{2p-1}
	\int_{{B}(\tau,\rho_{n,2})}
	\mathbb{E}_{\theta}^{n}\left[
	\left|\partial_{\tau}\partial_{\xi}\partial_{\theta}^{2}\ell_{n}(\theta^{\prime})\left[\left(\varphi_{n}(\theta)u\right)^{\otimes 2}\right]\right|^{2q}
	\right]\,\mathrm{d}\tau^{\prime} \\
	&\lesssim
	\sup_{\tau^{\prime}\in{B}(\tau,\rho_{n,2})}
	\mathbb{E}_{\theta}^{n}\left[
	\left|\partial_{\xi}\partial_{\theta}^{2}\ell_{n}(\theta^{\prime})\left[\left(\varphi_{n}(\theta)u\right)^{\otimes 2}\right]\right|^{2q}
	\right] \\
	&\quad+\left(\rho_{n,2}\right)^{2q}
	\sup_{\tau^{\prime}\in{B}(\tau,\rho_{n,2})}
	\mathbb{E}_{\theta}^{n}\left[
	\left|\partial_{\tau}\partial_{\xi}\partial_{\theta}^{2}\ell_{n}(\theta^{\prime})\left[\left(\varphi_{n}(\theta)u\right)^{\otimes 2}\right]\right|^{2q}
	\right]
\end{align*}
uniformly on $\Theta$. Similarly, for the second term of \eqref{LAN:LAN:remainder-Chebyshev}, Lemma~\eqref{lemma:Sobolev} and Fubini's theorem also give that for any $2q>1$, 
\begin{align*}
	\mathbb{E}_{\theta}^{n}&\left[
	\sup_{\tau^{\prime}\in{B}(\tau,\rho_{n,2})}
	\left|\left(
	\partial_{\theta}^{2}\ell_{n}(\xi,\tau^{\prime})
	-\partial_{\theta}^{2}\ell_{n}(\theta)
	\right)
	\left[\left(\varphi_{n}(\theta)u\right)^{\otimes 2}\right]\right|^{2q}
	\right]
	\\
 &\;\;\;\;\;\;\;\;\;\;\lesssim\left(\rho_{n,2}\right)^{2q-1}
	\int_{{B}(\tau,\rho_{n,2})}
	\mathbb{E}_{\theta}^{n}\left[
	\left|\partial_{\tau}\partial_{\theta}^{2}\ell_{n}(\theta^{\prime})\left[\left(\varphi_{n}(\theta)u\right)^{\otimes 2}\right]\right|^{2q}
	\right]\,\mathrm{d}\tau^{\prime} \\
	&\;\;\;\;\;\;\;\;\;\;\lesssim\left(\rho_{n,2}\right)^{2q}
	\sup_{\tau^{\prime}\in{B}(\tau,\rho_{n,2})}
	\mathbb{E}_{\theta}^{n}\left[
	\left|\partial_{\tau}\partial_{\theta}^{2}\ell_{n}(\theta^{\prime})\left[\left(\varphi_{n}(\theta)u\right)^{\otimes 2}\right]\right|^{2q}
	\right] 
\end{align*}
uniformly on $\Theta$.
Therefore \eqref{LAN:remainder} follows once we have proved that for any $2q>2$,
\begin{align}
\label{LAN:remainder-Suff1}
\begin{array}{l}
    \sup_{\theta^{\prime}\in
	{B}(\xi,\rho_{n,1})\times
	{B}(\tau,\rho_{n,2})}
    \mathbb{E}_{\theta}^{n}\Big[
	\Big|r_{n}^{1}(H)^{-\frac{1}{2}}
	a_{n}(H)
	\partial_{\xi}\partial_{\theta}^{2}\ell_{n}(\theta^{\prime})\Big[\Big(\varphi_{n}(\theta)u\Big)^{\otimes 2}\Big]\Big|^{2q}
	\Big]=o(1), 
	\\
    \sup_{\theta^{\prime}\in
	{B}(\xi,\rho_{n,1})\times
	{B}(\tau,\rho_{n,2})}
    \mathbb{E}_{\theta}^{n}\Big[
	\Big|r_{n}^{1}(H)^{-\frac{1}{2}}
	a_{n}(H)
	\partial_{\tau}\partial_{\xi}\partial_{\theta}^{2}\ell_{n}(\theta^{\prime})\Big[\Big(\varphi_{n}(\theta)u\Big)^{\otimes 2}\Big]\Big|^{2q}
	\Big]=o(1), 
	\\
    \sup_{\theta^{\prime}\in
	{B}(\xi,\rho_{n,1})\times
	{B}(\tau,\rho_{n,2})}
    \mathbb{E}_{\theta}^{n}\Big[
	\Big|r_{n}^{2}(H)^{-\frac{1}{2}}\partial_{\tau}\partial_{\theta}^{2}\ell_{n}(\theta^{\prime})\Big[\Big(\varphi_{n}(\theta)u\Big)^{\otimes 2}\Big]\Big|^{2q}
	\Big]=o(1) 
\end{array}
\end{align}
as $n\to\infty$ uniformly on $\Theta$, where $a_{n}(H)$ is defined in Equation~\eqref{def_an(H)}. 
By explicit computation of $\partial_{\xi}\partial_{\theta}^{2}\ell_{n}(\theta)$, $\partial_{\tau}\partial_{\xi}\partial_{\theta}^{2}\ell_{n}(\theta)$ and $\partial_{\tau}\partial_{\theta}^{2}\ell_{n}(\theta)$ and using Lemma~\ref{Lemma_RateMat} \eqref{Lemma_RateMat_elements_bdd} and the linearity $k\mapsto S_{n}(k)$ on $L^{1}([-\pi,\pi])$, we can see that \eqref{LAN:remainder-Suff1} follow once we have proved the following two results:
\begin{equation}\label{LAN:remainder-Suff2-1}
	\sup_{\theta^{\prime}\in
	{B}(\xi,\rho_{n,1})\times
	{B}(\tau,\rho_{n,2})}\left[
    \mathbb{E}_{\theta}^{n}\bigg[
	\bigg|\frac{1}{\overline{r}_{n}(H)}\mathbf{W}_{n}(\theta^{\prime})^{\top}\prod_{j=1}^{m}S_{n}(k_{j,\theta^{\prime}}^{n,\upsilon_{j}})\mathbf{W}_{n}(\theta^{\prime})\bigg|^{2q}
	\bigg]
	+\frac{1}{\overline{r}_{n}(H)}\bigg|
	\mathrm{Tr}\bigg[\prod_{j=1}^{m}S_{n}(k_{j,\theta^{\prime}}^{n,\upsilon_{j}})
	\bigg]\bigg|
    \right] = o(n^{\epsilon})
\end{equation}
as $n\to\infty$ uniformly on $\Theta$ for any $\epsilon>0$ and $(\upsilon_{1},\cdots,\upsilon_{m})\in\{0,1\}^{m}$ with $m\geq 2$, and
\begin{equation}\label{LAN:remainder-Suff2-2}
	\sup_{\theta^{\prime}\in
	{B}(\xi,\rho_{n,1})\times
	{B}(\tau,\rho_{n,2})}\left[
    \mathbb{E}_{\theta}^{n}\left[
	\left|\frac{1}{\overline{r}_{n}(H)}\left(
	\mathbf{W}_{n}(\theta^{\prime})^{\top}S_{n}(k_{1,\theta^{\prime}}^{n,\upsilon_{1}})\mathbf{W}_{n}(\theta^{\prime})
	-\mathrm{Tr}[S_{n}(k_{1,\theta^{\prime}}^{n,\upsilon_{1}})]
	\right)\right|^{2q}\right] 
    \right] = o(n^{\epsilon})\ \ \mbox{as $n\to\infty$}
\end{equation}
for any $\epsilon>0$ uniformly on $\Theta$, $\upsilon_{1}\in\{0,1\}$ and $m=1$, where $\overline{r}_{n}(H)=\overline{r}_{n}^{m,\underline{\upsilon}^{m}}(H)$ is defined in \eqref{def_rn_bar}. 
On the other hand, expressing moments as sum of cumulants and using the closed form expression for cumulants of Gaussian quadratic forms computed in Lemma 2 of \cite{magnus1986exact} and the homogeneous and translation-invariant properties of cumulants, \eqref{LAN:remainder-Suff2-1} and \eqref{LAN:remainder-Suff2-2} follow once we have proved 
\begin{equation}\label{LAN:remainder-Suff3-1}
	\sup_{\theta^{\prime}\in
	{B}(\xi,\rho_{n,1})\times
	{B}(\tau,\rho_{n,2})}
    \overline{r}_{n}(H)^{-q}\bigg|
	\mathrm{Tr}\bigg[\bigg\{
	\Cov^{n}_{\theta}[\mathbf{W}_{n}(\theta^{\prime})]^{w}\prod_{j=1}^{m}S_{n}(k_{j,\theta^{\prime}}^{n,\upsilon_{j}})\bigg\}^{q}\bigg]
	\bigg|=o(n^{\epsilon})\ \ \mbox{as $n\to\infty$}
\end{equation}
uniformly on $\Theta$ for any $\epsilon>0$, $w\in\{0,1\}$, $(\upsilon_{1},\cdots,\upsilon_{m})\in\{0,1\}^{m}$, $m\in\{1,2,3,4\}$ and $q\in\{1,2,\cdots,2p\}$ satisfying $qm\geq 2$, 
and
\begin{equation}\label{LAN:remainder-Suff3-2}
	\sup_{\theta^{\prime}\in
	{B}(\xi,\rho_{n,1})\times
	{B}(\tau,\rho_{n,2})}\frac{1}{\overline{r}_{n}(H)}\left|
	\mathrm{Tr}\left[\Cov^{n}_{\theta}[\mathbf{W}_{n}(\theta^{\prime})]S_{n}(k_{1,\theta^{\prime}}^{n,\upsilon_{1}})\right]
	-\mathrm{Tr}[S_{n}(k_{1,\theta^{\prime}}^{n,\upsilon_{1}})]
	\right|
	=o(1)\ \ \mbox{as $n\to\infty$}
\end{equation}
uniformly on $\Theta$ for each $\upsilon_{1}\in\{0,1\}$ and $m=1$, where we denote $\Cov[\mathbf{W}_{n}(\theta^{\prime})]^{0}=I_{n}$ for conciseness. 
Moreover, since we have $\Cov[\mathbf{W}_{n}(\theta^{\prime})]=\Sigma_{n}(h_{\theta^{\prime}}^{n})^{-\frac{1}{2}}\Sigma_{n}(h_{\theta}^{n})\Sigma_{n}(h_{\theta^{\prime}}^{n})^{-\frac{1}{2}}$, \eqref{LAN:remainder-Suff3-1} and \eqref{LAN:remainder-Suff3-2} follow from Proposition~\ref{Prop:LRR10_Thm5} once we have proved that
\begin{equation}\label{LAN:remainder-Suff4-1}
	\sup_{\theta^{\prime}\in
	{B}(\xi,\rho_{n,1})\times
	{B}(\tau,\rho_{n,2})}
    \frac{n}{\overline{r}_{n}(H)}\int_{-\pi}^{\pi}\bigg[
	\bigg(\frac{h_{\theta}^{n}(\lambda)}{h_{\theta^{\prime}}^{n}(\lambda)}\bigg)^{w}
	\prod_{j=1}^{m}\frac{k_{j,\theta^{\prime}}^{n,\upsilon_{j}}(\lambda)}{h_{\theta^{\prime}}^{n}(\lambda)}
	\bigg]^{q}\,\mathrm{d}\lambda
	=o(n^{\epsilon})\ \ \mbox{as $n\to\infty$}
\end{equation}
uniformly on $\Theta$ for any $\epsilon>0$, $w\in\{0,1\}$, $(\upsilon_{1},\cdots,\upsilon_{m})\in\{0,1\}^{m}$, $m\in\{1,2,3,4\}$ and $q\in\{1,2,\cdots,2p\}$ satisfying $qm\geq 2$, 
and
\begin{equation}\label{LAN:remainder-Suff4-2}
	\sup_{\theta^{\prime}\in
	{B}(\xi,\rho_{n,1})\times
	{B}(\tau,\rho_{n,2})}
    \frac{n}{\overline{r}_{n}(H)}\int_{-\pi}^{\pi}\frac{(h_{\theta}^{n}(\lambda)-h_{\theta^{\prime}}^{n}(\lambda))k_{j,\theta^{\prime}}^{n,\upsilon_{j}}(\lambda)}{h_{\theta^{\prime}}^{n}(\lambda)^{2}}\,\mathrm{d}\lambda
	=o(1)\ \ \mbox{as $n\to\infty$}
\end{equation}
uniformly on $\Theta$ for each $\upsilon_{1}\in\{0,1\}$ and $m=1$. 
We give the proofs of \eqref{LAN:remainder-Suff4-1} and \eqref{LAN:remainder-Suff4-2} in Sections~\ref{Sec:Proof-LAN:remainder-Suff4-1} and~\ref{Sec:Proof-LAN:remainder-Suff4-2} respectively.

\subsubsection{Proof of \eqref{LAN:remainder-Suff4-1}}
\label{Sec:Proof-LAN:remainder-Suff4-1}

We consider separatedly the cases (a) $\upsilon_{j}=1$ for all $j\in\{1,\cdots,m\}$ and (b) there exists $j$ such that $\upsilon_{j}=0$.\\

We firstly consider the case (a). 
Since $\overline{r}_{n}(H)=n$, we can show
\begin{align*}
	&\sup_{\theta^{\prime}\in
	{B}(\xi,\rho_{n,1})\times
	{B}(\tau,\rho_{n,2})}
    \frac{n}{\overline{r}_{n}(H)}\left|\int_{-\pi}^{\pi}\left[
	\left(\frac{h_{\theta}^{n}(\lambda)}{h_{\theta^{\prime}}^{n}(\lambda)}\right)^{w}
	\prod_{j=1}^{m}\frac{k_{j,\theta^{\prime}}^{n,\upsilon_{j}}(\lambda)}{h_{\theta^{\prime}}^{n}(\lambda)}
	\right]^{q}\,\mathrm{d}\lambda
	\right| \\
	&\lesssim 
	\sup_{\theta^{\prime}\in
	{B}(\xi,\rho_{n,1})\times
	{B}(\tau,\rho_{n,2})}
    \int_{0}^{\pi}
	\left(\Delta_{n}^{2(H-H^{\prime})}\lambda^{-2(H-H^{\prime})}+1\right)^{qw}\,\mathrm{d}\lambda\\
	&\lesssim
	\sup_{\theta^{\prime}\in
	{B}(\xi,\rho_{n,1})\times
	{B}(\tau,\rho_{n,2})}
    \Delta_{n}^{-2qw|H-H^{\prime}|}
	\left[
	\int_{0}^{1}\lambda^{-2qw|H-H^{\prime}|}\,\mathrm{d}\lambda
	+\int_{1}^{\pi}\lambda^{2qw|H-H^{\prime}|}\,\mathrm{d}\lambda
	\right]
\end{align*}
uniformly on $\Theta$. 
Then \eqref{LAN:remainder-Suff4-1} in the case (a) follows from \eqref{n_power_bdd} .\\ 

Finally, we consider the case (b). 
Since we have $\overline{r}_{n}(H)\gtrsim r_{n}^{1}(H)=n(\nu_{n}\Delta_{n}^{-H})^{-\frac{2}{\Diamond(H)}}$, we can show that
\begin{align*}
	&\sup_{\theta^{\prime}\in
	{B}(\xi,\rho_{n,1})\times
	{B}(\tau,\rho_{n,2})}
    \frac{n}{\overline{r}_{n}(H^{\prime})}\left|\int_{-\pi}^{\pi}\left[
	\left(\frac{h_{\theta}^{n}(\lambda)}{h_{\theta^{\prime}}^{n}(\lambda)}\right)^{w}
	\prod_{j=1}^{m}\frac{k_{j,\theta^{\prime}}^{n,\upsilon_{j}}(\lambda)}{h_{\theta^{\prime}}^{n}(\lambda)}
	\right]^{q}\,\mathrm{d}\lambda
	\right| \\
	&\lesssim 
	\sup_{\theta^{\prime}\in
	{B}(\xi,\rho_{n,1})\times
	{B}(\tau,\rho_{n,2})}
    \frac{n}{r_{n}^{1}(H^{\prime})}\int_{0}^{\pi}\left[
	\left(\Delta_{n}^{2(H-H^{\prime})}\lambda^{-2(H-H^{\prime})}+1\right)^{w}
	\prod_{j=1}^{m-|\underline{\upsilon}^{m}|}
	\frac{\lambda^{-\epsilon}}
	{1+\nu_{n}(H^{\prime})^{2}\lambda^{\Diamond(H^{\prime})}}
	\right]^{q}\,\mathrm{d}\lambda \\
	&\lesssim 
	\sup_{\theta^{\prime}\in
	{B}(\xi,\rho_{n,1})\times
	{B}(\tau,\rho_{n,2})}
    \left(\frac{r_{n}^{1}(H^{\prime})}{n}\right)^{-q(m-|\underline{\upsilon}^{m}|)\epsilon}
	\int_{0}^{\frac{n\pi}{r_{n}^{1}(H^{\prime})}}\left[
	\left(
	\left(\frac{n\Delta_{n}}{\mu r_{n}^{1}(H^{\prime})}\right)^{2(H-H^{\prime})}
	+1\right)^{w}
	\prod_{j=1}^{m-|\underline{\upsilon}^{m}|}
	\frac{\mu^{-\epsilon}}
	{1+\mu^{\Diamond(H^{\prime})}}
	\right]^{q}\,\mathrm{d}\mu \\
	&\lesssim
	\sup_{\theta^{\prime}\in
	{B}(\xi,\rho_{n,1})\times
	{B}(\tau,\rho_{n,2})}
    \left(\frac{r_{n}^{1}(H^{\prime})}{n}\right)^{-q(m-|\underline{\upsilon}^{m}|)\epsilon-2qw|H-H^{\prime}|}
	\Delta_{n}^{-2qw|H-H^{\prime}|} \\
	&\hspace{3cm}\times\left[
	\int_{0}^{1}\mu^{-2qw|H-H^{\prime}|-q(m-|\underline{\upsilon}^{m}|)\epsilon}\,\mathrm{d}\mu
	+\int_{1}^{\frac{n\pi}{r_{n}^{1}(H^{\prime})}}\mu^{2qw|H-H^{\prime}|-q(m-|\underline{\upsilon}^{m}|)(\Diamond(H^{\prime})+\epsilon)}\,\mathrm{d}\mu 
	\right] 
\end{align*}
uniformly on $\Theta$ and \eqref{LAN:remainder-Suff4-1} in the case (b) follows, using also \eqref{Ratio-Rate}, \eqref{n_power_bdd} and $\inf_{H\in\Theta_{H}}\Diamond(H)>1\geq|q(m-|\underline{\upsilon}^{m}|)|^{-1}$.

\subsubsection{Proof of \eqref{LAN:remainder-Suff4-2}}
\label{Sec:Proof-LAN:remainder-Suff4-2}
First note that
\begin{equation*}
        \int_{-\pi}^{\pi}\frac{\left|(h_{\theta}^{n}(\lambda)-h_{\theta^{\prime}}^{n}(\lambda))k_{j,\theta^{\prime}}^{n,\upsilon_{j}}(\lambda)\right|}{h_{\theta^{\prime}}^{n}(\lambda)^{2}}\,\mathrm{d}\lambda
\leq
        \int_{-\pi}^{\pi}\frac{\left|(f_{\xi}^{n}(\lambda)-f_{\xi^{\prime}}^{n}(\lambda))k_{1,\xi^{\prime}}^{n,\upsilon_{1}}(\lambda)\right|}
	{h_{\theta^{\prime}}^{n}(\lambda)^{2}}\,\mathrm{d}\lambda
+
     \int_{-\pi}^{\pi}\frac{\left|(g_{\tau}^{n}(\lambda)-g_{\tau^{\prime}}^{n}(\lambda))k_{1,\xi^{\prime}}^{n,\upsilon_{1}}(\lambda)\right|}
	{h_{\theta^{\prime}}^{n}(\lambda)^{2}}\,\mathrm{d}\lambda.
\end{equation*}
We start by considering the first integral. We have
\begin{align*}
	\left|f_{\xi}^{n}(\lambda)-f_{\xi^{\prime}}^{n}(\lambda)\right|
	&=f_{\xi^{\prime}}^{n}(\lambda)\left|\frac{f_{\xi}^{n}(\lambda)}{f_{\xi^{\prime}}^{n}(\lambda)}-1\right| \leq f_{\xi^{\prime}}^{n}(\lambda)\left[
	\left|\Delta_{n}^{2(H-H^{\prime})}-1\right|\frac{f_{\xi}(\lambda)}{f_{\xi^{\prime}}(\lambda)}
	+\left|\frac{f_{\xi}(\lambda)}{f_{\xi^{\prime}}(\lambda)}-1\right|
	\right],
\end{align*}
and there exists a constant $C_{1}>0$, independent of $\lambda$, $\theta$, $\theta^{\prime}$ and $n$, such that for any $\xi\in\Theta_{H}\times\Theta_{\sigma}$ and $\xi^{\prime}\in B(\xi,\rho_{n,1})$,
\begin{align*}
	\left|\frac{f_{\xi}(\lambda)}{f_{\xi^{\prime}}(\lambda)}-1\right|
	&=\frac{\left|f_{\xi}(\lambda)-f_{\xi^{\prime}}(\lambda)\right|}{f_{\xi^{\prime}}(\lambda)}\\
	&\leq\|\xi-\xi^{\prime}\|_{\mathbb{R}^{2}}
	\int_{0}^{1}\frac{\left\|\partial_{\xi}f_{\xi^{\prime}+z(\xi-\xi^{\prime})}(\lambda)\right\|_{\mathbb{R}^{2}}}{f_{\xi^{\prime}}(\lambda)}\,\mathrm{d}z\\
	&\leq C_{1} r_{n}^{1}(H)^{-\frac{1}{2}}|\log{n}|
	(1+|\log{\lambda}|)
	\left(|\lambda|^{-2|H-H^{\prime}|}\vee |\lambda|^{2|H-H^{\prime}|}\right)
\end{align*}
using \eqref{Tightness-Shrinkage}. 
Then, we can show that for any $\upsilon_{1}\in\{0,1\}$ and $\epsilon>0$ small enough, 
\begin{align*}
	\sup_{\theta^{\prime}\in
	{B}(\xi,\rho_{n,1})\times
	{B}(\tau,\rho_{n,2})}
    &\frac{n}{\overline{r}_{n}(H)}
	\int_{-\pi}^{\pi}\frac{\left|(f_{\xi}^{n}(\lambda)-f_{\xi^{\prime}}^{n}(\lambda))k_{1,\xi^{\prime}}^{n,\upsilon_{1}}(\lambda)\right|}
	{h_{\theta^{\prime}}^{n}(\lambda)^{2}}\,\mathrm{d}\lambda \\
   &\lesssim
	\sup_{\theta^{\prime}\in
	{B}(\xi,\rho_{n,1})\times
	{B}(\tau,\rho_{n,2})}
    \frac{\log{n}}{r_{n}^{1}(H)^{\frac{1}{2}}}\cdot
	\frac{n}{\overline{r}_{n}(H)}
	\int_{0}^{\pi}\frac{
	\lambda^{-\epsilon(2-\upsilon_{1})}
	(\nu_{n}(H^{\prime})^{2}\lambda^{\Diamond(H^{\prime})})^{\upsilon_{1}}}
	{(1+\nu_{n}(H^{\prime})^{2}\lambda^{\Diamond(H^{\prime})})^{2}}\,\mathrm{d}\lambda 
\end{align*}  
uniformly on $\Theta$. 
Then, for each $\upsilon_{1}\in\{0,1\}$, we can use \eqref{Ratio-Rate} and \eqref{n_power_bdd} and take $\epsilon>0$ small enough to get
\begin{equation}
    	\sup_{\theta^{\prime}\in
	{B}(\xi,\rho_{n,1})\times
	{B}(\tau,\rho_{n,2})}
    \frac{n}{\overline{r}_{n}(H)}\int_{-\pi}^{\pi}\frac{\left|(f_{\xi}^{n}(\lambda)-f_{\xi^{\prime}}^{n}(\lambda))k_{1,\xi^{\prime}}^{n,\upsilon_{1}}(\lambda)\right|}
	{h_{\theta^{\prime}}^{n}(\lambda)^{2}}\,\mathrm{d}\lambda
	=o(1)\ \ \mbox{as $n\to\infty$}.
	\label{LAN:remainder-Suff4-2-1}
\end{equation}
Now, we consider the integral $\int_{-\pi}^{\pi}\frac{\left|(g_{\tau}^{n}(\lambda)-g_{\tau^{\prime}}^{n}(\lambda))k_{1,\xi^{\prime}}^{n,\upsilon_{1}}(\lambda)\right|}
	{h_{\theta^{\prime}}^{n}(\lambda)^{2}}\,\mathrm{d}\lambda$. 
Using that 
\begin{equation*}
    \sup_{\theta^{\prime}\in\Theta: \left\|\Phi_{n}(\theta)^{-1}(\theta^{\prime}-\theta)\right\|_{\mathbb{R}^{3}}\leq c} r_{n}^{2}(H)^{1/2} |\tau^{\prime}-\tau|\lesssim 1,
\end{equation*}
we can see that there exist constants $C_{1},C_{2}>0$, independent of $\lambda$, $\theta$, $\theta^{\prime}$ and $n$, such that for any $\tau\in\Theta_{\tau}$ and $\xi^{\prime}\in B(\tau,\rho_{n,2})$,
\begin{align*}
	\left|g_{\tau}^{n}(\lambda)-g_{\tau^{\prime}}^{n}(\lambda)\right|
	\leq C_{1}|\tau-\tau^{\prime}|
	\nu_{n}^{2}|\lambda|^{2(K+1)}
	\leq C_{2}r_{n}^{2}(H)^{-\frac{1}{2}}\nu_{n}^{2}|\lambda|^{2(K+1)}.
\end{align*}
Therefore, for any $\upsilon_{1}\in\{0,1\}$ and $\epsilon>0$ small enough,  
\begin{align*}
	&\sup_{\theta^{\prime}\in
	{B}(\xi,\rho_{n,1})\times
	{B}(\tau,\rho_{n,2})}
    \frac{n}{\overline{r}_{n}(H)}
	\int_{-\pi}^{\pi}\frac{\big|(g_{\tau}^{n}(\lambda)-g_{\tau^{\prime}}^{n}(\lambda))k_{1,\xi^{\prime}}^{n,\upsilon_{1}}(\lambda)\big|}
	{h_{\theta^{\prime}}^{n}(\lambda)^{2}}\,\mathrm{d}\lambda \\
    &\lesssim
	\sup_{\theta^{\prime}\in
	{B}(\xi,\rho_{n,1})\times
	{B}(\tau,\rho_{n,2})}
    \frac{n}{\overline{r}_{n}(H) r_{n}^{2}(H)^{\frac{1}{2}}}
	\int_{0}^{\pi}\frac{\nu_{n}(H^{\prime})^{2}
	\lambda^{\Diamond(H^{\prime})}
	\lambda^{-\epsilon(1-\upsilon_{1})}
	(\nu_{n}(H^{\prime})^{2}\lambda^{\Diamond(H^{\prime})})^{\upsilon_{1}}}
	{(1+\nu_{n}(H^{\prime})^{2}\lambda^{\Diamond(H^{\prime})})^{2}}\,\mathrm{d}\lambda
\end{align*}
uniformly on $\Theta$. Then for each $\upsilon_{1}\in\{0,1\}$, we get
\begin{equation}
\sup_{\theta^{\prime}\in
	{B}(\xi,\rho_{n,1})\times
	{B}(\tau,\rho_{n,2})}
    \frac{n}{\overline{r}_{n}(H)}\int_{-\pi}^{\pi}\frac{\left|(g_{\tau}^{n}(\lambda)-g_{\tau^{\prime}}^{n}(\lambda))k_{1,\xi^{\prime}}^{n,\upsilon_{1}}(\lambda)\right|}
	{h_{\theta^{\prime}}^{n}(\lambda)^{2}}\,\mathrm{d}\lambda
	=o(1)\ \ \mbox{as $n\to\infty$}
	\label{LAN:remainder-Suff4-2-2}
\end{equation}
using \eqref{Ratio-Rate} and \eqref{n_power_bdd} and taking $\epsilon>0$ small enough. \\

Combining \eqref{LAN:remainder-Suff4-2-1} and \eqref{LAN:remainder-Suff4-2-2} concludes the proof.

\section{Proof of Proposition~\ref{Prop:LRR10_Thm5}}\label{Section_Proof_LRR10_Thm5}
We prove Proposition~\ref{Prop:LRR10_Thm5} in a similar line with the proof of Equation~5 (a) in the full version of \cite{Lieberman-Rosemarin-Rousseau-2012}.  \\

First note that the mapping $f\mapsto\Sigma_{n}(f)$ on is linear $L^{1}([-\pi,\pi])$ and that we can decompose any function $f$ as $f=f_+ - f_-$. Therefore, without loss of generality, we can assume that for each $r\in\{1,2,\cdots,q\}$ and $\theta_{2r-1}=(H_{2r-1},\sigma_{2r-1},\tau_{2r-1})\in(0,1)\times[0,\infty)^{2}$, 
the function $d_{r,H_{2r-1}}$ and $e_{r}$ are non-negative and either $\sigma_{2r-1}$ or $\tau_{2r-1}$ is equal to $0$ and the other is positive. Moreover, we assume that the set $\{x\in\Pi:k_{r,\theta_{2r-1}}^{n}(x)>0\}$ has a positive Lebesgue measure.\\ 

Then we summarize the notation used in the proof of Proposition~\ref{Prop:LRR10_Thm5}. 
Fix $m\in\mathbb{N}$ satisfying $m>K+3/2$. Write
\begin{equation*}
	u_{\theta}^{n,m}(x):=[(2\pi)^{2m}h_{\theta}^{n}(x)]^{-\frac{1}{2m-1}}\;\;\text{ and }\;\;
	u_{\theta}^{n}(x):=u_{\theta}^{n,1}(x),
\end{equation*}
and for each $r\in\{1,2,\cdots,q\}$,
\begin{equation*}
	\Gamma_{n,m}(\theta)
	:=\Sigma_{n}(u_{\theta_{2r}}^{n,m})^{2m-1}.
\end{equation*}
Since $\Sigma_{n}(k_{r,\theta_{2r-1}}^{n})$ is positive definite for each  $r\in\{1,2,\cdots,q\}$ from the above Assumption $(3)$, we define 
\begin{align*}
\begin{cases}
	A_{r}^{n}(\underline{\theta}):=
	\Sigma_{n}(k_{r,\theta_{2r-1}}^{n})\Sigma_{n}(h_{\theta_{2r}}^{n})^{-1},\\ 
	\widetilde{A}_{r}^{n}(\underline{\theta}):=
	\Sigma_{n}(k_{r,\theta_{2r-1}}^{n})^{\frac{1}{2}}\Sigma_{n}(h_{\theta_{2r}}^{n})^{-1}\Sigma_{n}(k_{r+1,\theta_{2r+1}}^{n})^{\frac{1}{2}},\\ 
	B_{r}^{n}(\underline{\theta}):=
	\Sigma_{n}(k_{r,\theta_{2r-1}}^{n})\Gamma_{n,m}(\theta),\\ 
	\widetilde{B}_{r}^{n}(\underline{\theta}):=
	\Sigma_{n}(k_{r,\theta_{2r-1}}^{n})^{\frac{1}{2}}\Gamma_{n,m}(\theta)\Sigma_{n}(k_{r+1,\theta_{2r+1}}^{n})^{\frac{1}{2}},\\
\end{cases}
\;\;\;\text{ and }\;\;\;
	\begin{cases}
D_{r}^{n}(\underline{\theta}):=
	I_{n}-\Sigma_{n}(h_{\theta_{2r}}^{n})\Gamma_{n,m}(\theta),\\ 
	\widetilde{D}_{r}^{n}(\underline{\theta}):=
	I_{n}-\Sigma_{n}(h_{\theta_{2r}}^{n})^{\frac{1}{2}}\Gamma_{n,m}(\theta)\Sigma_{n}(h_{\theta_{2r}}^{n})^{\frac{1}{2}},\\
	V_{1,r}^{n}(\underline{\theta}):=
	\Sigma_{n}(k_{r,\theta_{2r-1}}^{n})^{\frac{1}{2}}\Sigma_{n}(h_{\theta_{2r}}^{n})^{-\frac{1}{2}},\\ 
	V_{2,r}^{n}(\underline{\theta}):=
	\Sigma_{n}(h_{\theta_{2r}}^{n})^{-\frac{1}{2}}\Sigma_{n}(k_{r+1,\theta_{2r+1}}^{n})^{\frac{1}{2}},
\end{cases}
\end{align*}	
where we write $\theta_{2q+1}:=\theta_{1}$ and $k_{2q+1,\theta}^{n}(x):=k_{1,\theta}^{n}(x)$ for conciseness.\\

Moreover, for $\underline{\theta}\in\Theta_{1}(\iota)$, we define
\begin{align*}
	I(\underline{\theta}):=
	(2\pi)^{2q-1}\int_{-\pi}^{\pi}\prod_{r=1}^{q}
	k_{r,\theta_{2r-1}}^{n}(x)u_{\theta_{2r}}^{n}(x)\,\mathrm{d}x 
	=\frac{1}{2\pi}\int_{-\pi}^{\pi}\prod_{r=1}^{q}\frac{k_{r,\theta_{2r-1}}^{n}(x)}{h_{\theta_{2r}}^{n}(x)}\,\mathrm{d}x
\end{align*}
and
\begin{equation*}
	E_{n}(\underline{\theta}):=
	\left|\mathrm{Tr}\left[\prod_{r=1}^{q}\Sigma_{n}(k_{r,\theta_{2r-1}}^{n})\Sigma_{n}(h_{\theta_{2r}}^{n})^{-1}\right]
	   	-\frac{n}{2\pi}\int_{-\pi}^{\pi}\prod_{r=1}^{q}\frac{k_{r,\theta_{2r-1}}^{n}(x)}{h_{\theta_{2r}}^{n}(x)}\,\mathrm{d}x\right|. 
\end{equation*}
Using the triangle inequality, the error $E_{n}(\underline{\theta})$ is divided into the following two terms:
\begin{equation*}
	E_{n}(\underline{\theta})\leq E_{n}^{(1)}(\underline{\theta})+E_{n}^{(2)}(\underline{\theta}),
\end{equation*}
where
\begin{align}
\label{def-En12}	
\begin{cases}
E_{n}^{(1)}(\underline{\theta}):=
	\left|\mathrm{Tr}\left[\prod_{r=1}^{q}B_{r}^{n}(\underline{\theta})\right]
	   	-\frac{n}{2\pi}\int_{-\pi}^{\pi}\prod_{r=1}^{q}\frac{k_{r,\theta_{2r-1}}^{n}(x)}{h_{\theta_{2r}}^{n}(x)}\,\mathrm{d}x\right|,
\\
	E_{n}^{(2)}(\underline{\theta}):=
	\left|\mathrm{Tr}\left[\prod_{r=1}^{q}A_{r}^{n}(\underline{\theta})\right]
	   	-\mathrm{Tr}\left[\prod_{r=1}^{q}B_{r}^{n}(\underline{\theta})\right]\right|.
	\end{cases}
\end{align}
In the rest of the proof, we will derive error bounds of $E_{n}^{(1)}(\underline{\theta})$ and $E_{n}^{(2)}(\underline{\theta})$ separately.

\subsection{Preliminary result}

In this section, we prove the following proposition needed for the proof of Proposition~\ref{Prop:LRR10_Thm5}.

\begin{proposition}\label{Prop:T23_JTSA}
	Let $m\in\mathbb{N}$ satisfying $m>K+3/2$. Under the same assumptions in Proposition~\ref{Prop:LRR10_Thm5}, we obtain
	\begin{equation*}
			\Delta_{n}^{\rho_{q,-}(\underline{\xi})+\epsilon_{1}}n^{-\psi_{q,-}(\underline{\xi})-\epsilon_{2}}
			E_{n}^{(1)}(\underline{\theta})
			=o(1)\ \ \mbox{as $n\to\infty$}
		\end{equation*}
		uniformly on compact subsets of $\Theta_{1}(\iota)$ for any $\epsilon_{1},\epsilon_{2}>0$ and $\iota\in(0,1)$. 
\end{proposition}

	\subsubsection{Outline of Proof of Proposition~\ref{Prop:T23_JTSA}}
	We prove Proposition~\ref{Prop:T23_JTSA} proceeding as the proof of Equation~1 of \cite{Takabatake-2023-JTSA-supplement}.
	We first introduce some notation. Fix $m\in\mathbb{N}$ such that $m>K+3/2$ and $c>1$. We define
\begin{align}
  \label{def-W}  
  \begin{cases}
    &W_{s}:=\left\{\mathbf{x}=(x_{1},\cdots,x_{2mq})\in\mathbb{R}^{2mq}:|\overline{x}_{s}|\leq c|x_{s+1}|\right\},\ \
		s\in\{1,2,\cdots,2mq-1\}, \\
	&W_{2mq}:=\left\{\mathbf{x}=(x_{1},\cdots,x_{2mq})\in\mathbb{R}^{2mq}:|\overline{x}_{2mq}|\leq c|\overline{x}_{2mq}-x_{1}|\right\},\\ 
    &W:=\bigcup_{s=1}^{2mq}W_{s},
  \end{cases}
\end{align}
where we write $\overline{x}_{s}:=\sum_{r=1}^{s}x_{r}$ for $s\in\{1,2,\cdots,2mq\}$. 
	For each $r\in\{1,2,\cdots,q\}$, we write
	\begin{align*}
		&c_{r}^{n,(0)}(\overline{x}_{2m(r-1)+1})=
		k_{r,\theta_{2r-1}}^{n}(\overline{x}_{2m(r-1)+1})
		\prod_{v=2}^{2m}u_{\theta_{2r}}^{n,m}(\overline{x}_{2m(r-1)+1})
		=k_{r,\theta_{2r-1}}^{n}(\overline{x}_{2m(r-1)+1})
		u_{\theta_{2r}}^{n}(\overline{x}_{2m(r-1)+1}),\\
		&c_{r}^{n,(1)}(\overline{x}_{2m(r-1)+1})=k_{r,\theta_{2r-1}}^{n}(\overline{x}_{2m(r-1)+1})
		\left(\prod_{v=2}^{2m}u_{\theta_{2r}}^{n,m}(\overline{x}_{2m(r-1)+v})
		-\prod_{v=2}^{2m}u_{\theta_{2r}}^{n,m}(\overline{x}_{2m(r-1)+1})\right),
	\end{align*} 
	and $I_{n}^{\varpi}(\underline{\theta}):=I_{n,1}^{\varpi}(\underline{\theta})+I_{n,2}^{\varpi}(\underline{\theta})$ where
	\begin{align*}
		&I_{n,1}^{\varpi}(\underline{\theta}):=
		\int_{[-\pi,\pi]^{2mq}\cap{W}}
		\prod_{r=1}^{q}c_{r}^{(\varpi_{r})}(\overline{x}_{2m(r-1)+1})		
		\cdot D_{n}(\overline{x}_{2mq}-x_{1})^{\ast}\prod_{r=2}^{2mq}D_{n}(x_{r})\,\mathrm{d}x_{1}\cdots\,\mathrm{d}x_{2mq},\\
		&I_{n,2}^{\varpi}(\underline{\theta}):=
        \int_{[-\pi,\pi]^{2mq}\cap{W}^{c}}
		\prod_{r=1}^{q}c_{r}^{(\varpi_{r})}(\overline{x}_{2m(r-1)+1})		
		\cdot D_{n}(\overline{x}_{2mq}-x_{1})^{\ast}\prod_{r=2}^{2mq}D_{n}(x_{r})\,\mathrm{d}x_{1}\cdots\,\mathrm{d}x_{2mq}
	\end{align*}
	for $\varpi=(\varpi_{1},\cdots,\varpi_{q})\in\{0,1\}^{q}$, where $D_{n}(x)=\sum_{t=1}^{n}e^{{i}xt}$ for $x\in\mathbb{R}$. 
	We also write $I_{n}(\underline{\theta})=I_{n}^{(0,\cdots,0)}(\underline{\theta})$ for conciseness.\\ 
	
Note in addition that
    \begin{equation*}		k_{r,\theta_{2r-1}}^{n}(\overline{x}_{2m(r-1)+1})
		\prod_{v=2}^{2m}u_{\theta_{2r}}^{n,m}(\overline{x}_{2m(r-1)+v})=c_{r}^{n,(0)}(\overline{x}_{2m(r-1)+1})+c_{r}^{n,(1)}(\overline{x}_{2m(r-1)+1}),
	\end{equation*}
	so we obtain the following upper bound for $E_{n}(\underline{\theta})$ using Lemma 3 of \cite{Takabatake-2023-JTSA-supplement}:
	\begin{equation*}
		E_{n}(\underline{\theta})\lesssim n^{-1}|I_{n}(\underline{\theta})-nI(\underline{\theta})| +n^{-1}\sum|I_{n,1}^{\varpi}(\underline{\theta})| +n^{-1}\sum|I_{n,2}^{\varpi}(\underline{\theta})|.
	\end{equation*}
    uniformly on compact subsets of $\Theta_{\ast}$, where the sums holds over all $\varpi=(\varpi_{1},\cdots,\varpi_{q})\in\{0,1\}^{q}$ such that $\varpi \neq (0,\cdots,0)$. 
	Therefore Proposition~\ref{Prop:T23_JTSA} follows once we have proved the following results.	
	\begin{lemma}\label{key_proposition_lemma}
		Suppose the assumptions given in Proposition~\ref{Prop:T23_JTSA} to hold. 
		Then the following results hold uniformly on compact subsets of $\Theta_{1}(\iota)$ for any $\iota\in(0,1)$. 
		\begin{enumerate}[$(1)$]
			\item\label{key_Prop:W}
			For any $\epsilon_{1},\epsilon_{2}>0$ and $\varpi\in\{0,1\}^{q}$, we have $I_{n,1}^{\varpi}(\underline{\theta})=o(\Delta_{n}^{-\rho_{q}(\underline{\xi})-\epsilon_{1}}
			    n^{\overline{\psi}_{q}(\underline{\xi})+\epsilon_{2}})$ as $n\to\infty$ uniformly on compact subsets of $\Theta_{\ast}$.
			\item\label{key_Prop:Wc}
			For any $\epsilon_{1},\epsilon_{2}>0$ and $\varpi\in\{0,1\}^{q}\setminus\{(0,\cdots,0)\}$, we have $I_{n,2}^{\varpi}(\underline{\theta})=o(\Delta_{n}^{-\rho_{q}(\underline{\xi})-\epsilon_{1}}
			    n^{\psi_{q,-}(\underline{\xi})+\epsilon_{2}})$ as $n\to\infty$ uniformly on compact subsets of $\Theta_{\ast}$.
			\item\label{key_Prop:main_term}
			For any $\epsilon_{1},\epsilon_{2}>0$, we have $|I_{n}(\underline{\theta})-nI(\underline{\theta})|=o(\Delta_{n}^{-\rho_{q}(\underline{\xi})-\epsilon_{1}}
			    n^{\overline{\psi}_{q}(\underline{\xi})+\epsilon_{2}})$ as $n\to\infty$ uniformly on compact subsets of $\Theta_{\ast}$.
		\end{enumerate}
	\end{lemma}

	In the rest of this section, we will prove Lemma~\ref{key_proposition_lemma} (1)-(3) separately.  
	\subsubsection{Proof of Lemma~\ref{key_proposition_lemma} \eqref{key_Prop:W}}\label{Section_proof_key_prop1}
	For any $\epsilon>0$ and each $r\in\{1,2,\cdots,q\}$, $\xi_{2r-1}=(H_{2r-1},\sigma_{2r-1})\in(0,1)\in[0,\infty)$ and $\xi_{2r}=(H_{2r},\sigma_{2r})\in(0,1)\in(0,\infty)$, 
	we take $\eta_{2m(r-1)+v}\equiv\eta_{2m(r-1)+v}(\xi_{2r-1},\xi_{2r},\epsilon)\in(0,1)$ for $v\in\{1,2,\cdots,2m\}$ such that for each $r\in\{1,2,\cdots,q\}$ and $v\in\{2,3,\cdots,m\}$,
	\begin{align*}
		&\eta_{2m(r-1)+1}+\eta_{2m(r-1)+2}
		>\left((2H_{2r-1}-1)-(2m-1)^{-1}(2H_{2r}-1)\right)\Indi_{(0,\infty)}(\sigma_{2r-1}),\\
		&\eta_{2m(r-1)+2v-1}+\eta_{2m(r-1)+2v}>-2(2m-1)^{-1}(2H_{2r}-1)\Indi_{(0,\infty)}(\sigma_{2r-1}),
	\end{align*}
	and
	\begin{equation*}
		\sum_{v=1}^{2m}\eta_{2m(r-1)+v} 
		< 2 (H_{2r-1}-H_{2r})_{+}\Indi_{(0,\infty)}(\tau_{2r-1})
		+\frac{\epsilon}{q}.
	\end{equation*}
	Note that for each $r\in\{1,2,\cdots,q\}$ and $v\in\{2,3,\cdots,m\}$, we also have
	\begin{align*}
		&\eta_{2m(r-1)+1}+\eta_{2m(r-1)+2}>-\frac{4(m-1)(K+1)}{2m-1}, \;\;\text{ and }\;\;
		&\eta_{2m(r-1)+2v-1}+\eta_{2m(r-1)+2v}>\frac{-4(K+1)}{2m-1}.
	\end{align*}
	Moreover, since $\eta_{r}>0$ for all $r\in\{1,2,\cdots,2mq\}$, we can show
	\begin{align*}
		\sum_{v=1}^{2m}\eta_{2m(r-1)+v}>2(H_{2r-1}-H_{2r})_{+}\Indi_{(0,\infty)}(\tau_{2r-1})
	\end{align*}
	so that $\psi_{p,-}(\underline{H})<\sum_{r=1}^{2mq}\eta_{r}<\psi_{p,-}(\underline{H})+\epsilon$ holds for $\underline{H}=(H_{1},\cdots,H_{2q})$.\\

	Let $\eta\in[0,1]$ and $\overline{L}_{\eta}$ be the $2\pi$-periodic function given by $\overline{L}_{\eta}(x):=|x|^{\eta-1}$ on $[-\pi,\pi]$.	Define 
	\begin{equation*}
		f_{n}(\mathbf{x}):=
		\prod_{r=1}^{q}f_{n,r}^{(0)}(\mathbf{x})
		\cdot\prod_{s=1}^{2mq}\overline{L}_{\eta_{r}}(\overline{x}_{s+1}-\overline{x}_{s}),
	\end{equation*}
	where $\mathbf{x}=(x_{1},\cdots,x_{2mq})$ and 
	\begin{align*}
		f_{n,r}^{(0)}(\mathbf{x}):=&\left|\overline{x}_{2m(r-1)+1}\right|^{1-2H_{2r-1}}
		\prod_{v=2}^{2m}\left[
		\left|\overline{x}_{2m(r-1)+v}\right|^{\frac{2H_{2r}-1}{2m-1}}
		+\left|\overline{x}_{2m(r-1)+1}\right|^{\frac{2H_{2r}-1}{2m-1}}
		\right]\\
		&+\left|\overline{x}_{2m(r-1)+1}\right|^{2(K+1)}
		\prod_{v=2}^{2m}\left[
		\left|\overline{x}_{2m(r-1)+v}\right|^{-\frac{2(K+1)}{2m-1}}
		+\left|\overline{x}_{2m(r-1)+1}\right|^{-\frac{2(K+1)}{2m-1}}
		\right]
	\end{align*}
	where we denote by $\overline{x}_{2mq+1}=x_{1}$. Using Lemma 2 of \cite{Takabatake-2023-JTSA-supplement}, 
	we can show that for each $r\in\{1,2,\cdots,q\}$, we have
	\begin{align*}
		\max_{\varpi_{r}\in\{0,1\}}\left|c_{r}^{(\varpi_{r})}(\overline{x}_{2m(r-1)+1})\right|
		&\leq 
		\left|d_{r,\xi_{2r-1}}^{n}(\overline{x}_{2m(r-1)+1})\right|
		\prod_{v=2}^{2m}\left[
		\left\{f_{\xi_{2r}}^{n}(\overline{x}_{2m(r-1)+v})\right\}^{-\frac{1}{2m-1}}
		+\left[f_{\xi_{2r}}^{n}(\overline{x}_{2m(r-1)+1})\right]^{-\frac{1}{2m-1}}
		\right]\\
		&\quad+\left|e_{r,\tau_{2r-1}}^{n}(\overline{x}_{2m(r-1)+1})\right|
		\prod_{v=2}^{2m}\left[
		\left\{g_{\tau_{2r}}^{n}(\overline{x}_{2m(r-1)+v})\right\}^{-\frac{1}{2m-1}}
		+\left\{g_{\tau_{2r}}^{n}(\overline{x}_{2m(r-1)+1})\right\}^{-\frac{1}{2m-1}}
		\right] \\
		&\lesssim 
		\Delta_{n}^{2(H_{2r-1}-H_{2r})}\left|\overline{x}_{2m(r-1)+1}\right|^{1-2H_{2r-1}}
		\prod_{v=2}^{2m}\left[
		\left|\overline{x}_{2m(r-1)+v}\right|^{\frac{2H_{2r}-1}{2m-1}}
		+\left|\overline{x}_{2m(r-1)+1}\right|^{\frac{2H_{2r}-1}{2m-1}}
		\right] \\
		&\quad+\left|\overline{x}_{2m(r-1)+1}\right|^{2(K+1)}
		\prod_{v=2}^{2m}\left[
		\left|\overline{x}_{2m(r-1)+v}\right|^{-\frac{2(K+1)}{2m-1}}
		+\left|\overline{x}_{2m(r-1)+1}\right|^{-\frac{2(K+1)}{2m-1}}
		\right] \\
		&\lesssim 
		\Delta_{n}^{-2(H_{2r}-H_{2r-1})_{+}}f_{n,r}^{(0)}(\mathbf{x})
	\end{align*}
	uniformly on compact subsets of $\Theta_{\ast}$ and therefore
 Then we see
	\begin{equation*}
		|I_{n,1}^{\varpi}(\underline{\theta})|\lesssim  \Delta_{n}^{-\rho_{q,-}(\underline{\xi})-\epsilon_{1}}
		n^{\sum_{r=1}^{2mq}\eta_{r}}\sum_{r=1}^{2mq}\int_{[-\pi,\pi]^{2mq}\cap{W}_{r}}f_{n}(\mathbf{x})\,\mathrm{d}\mathbf{x}
	\end{equation*}
	uniformly on compact subsets of $\Theta_{\ast}$.
 Proceeding as the proof of Proposition~6.1 of \cite{fox1987central}, we can also prove that
	\begin{equation*}
		\sum_{r=1}^{2mq}\int_{[-\pi,\pi]^{2mq}\cap{W}_{r}}{f}_{n}(\mathbf{x})\,\mathrm{d}\mathbf{x}\lesssim 1
	\end{equation*}
	uniformly on compact subsets of $\Theta_{\ast}$ since the assumption $m>K+3/2$ implies $2(K+1)/(2m-1)\in(0,1)$. 
	This completes the proof.
	\subsubsection{Proof of Lemma~\ref{key_proposition_lemma}~\eqref{key_Prop:Wc}}
 \label{Section_proof_key_prop2}
	Before proving Lemma~\ref{key_proposition_lemma}~\eqref{key_Prop:Wc}, we prepare the following lemma. 
	\begin{lemma}\label{key_lemma_W-tilda-c}
		Denote 
		$H_{1}:=\{(x_{1},\cdots,x_{2mq})\in\mathbb{R}^{2mq}:x_{1}=0\}$ and set $A:=[-\pi,\pi]^{2mq}\cap{W}^{c}\cap{H}_{1}^{c}$, where $W$ is defined in \eqref{def-W}. 
		For any $\varpi=(\varpi_{1},\cdots,\varpi_{q})\in\{0,1\}^{q}$, we have
		\begin{equation}\label{key_lemma_W-tilda-c_ineq-main}
		  \Delta_{n}^{\rho_{q,-}(\underline{\xi})+\epsilon_{1}}\left|
		  \prod_{r=1}^{q}c_{r}^{n,(\varpi_{r})}(\overline{x}_{2m(r-1)+1})\right|
		  \lesssim |x_{1}|^{-\psi_{q}(\underline{\xi})-\epsilon_{2}-|\varpi|}\prod_{r=1}^{q}\left[\sum_{w=2}^{2m}\sum_{v=1}^{w-1}\left|x_{2m(r-1)+v+1}\right|\right]^{\varpi_{r}}
		\end{equation}
		uniformly on compact subsets of $A\times\Theta_{\ast}$ for any $\epsilon_{1},\epsilon_{2}>0$, where we write $|\varpi|:=\sum_{r=1}^{q}\varpi_{r}$.
	\end{lemma}
	\begin{proof}[Proof of Lemma~\ref{key_lemma_W-tilda-c}]
		Fix $\varpi=(\varpi_{1},\cdots,\varpi_{q})\in\{0,1\}^{q}$. 
		We can assume $\overline{x}_{s}\neq 0$ for all $s\in\{1,2,\cdots,2mq\}$ without loss of generality. Indeed, $(x_{1},\cdots,x_{2mq})\in{W}^{c}$ implies
		\begin{equation}\label{key_inequality_21}
			c_{-}|\overline{x}_{s-1}|\leq|\overline{x}_{s}|\leq c_{+}|\overline{x}_{s-1}|,\ \ s\in\{2,3,\cdots,2mq\},
		\end{equation} 
		where $c_{\pm}=1\pm c^{-1}$.  
		Note that $c_\pm>0$ Using $c>1$. Therefore $x_{1}\neq 0$ implies $\overline{x}_{s}\neq 0$ for all $s\in\{1,2,\cdots,2mq\}$.\\

By a straightforward calculation, we can see that the inequality \eqref{key_lemma_W-tilda-c_ineq-main} follows once we have proved that for each $r\in\{1,2,\cdots,q\}$ and $v\in\{2,3,\cdots,m\}$ and any $\epsilon_{1},\epsilon_{2}>0$,
\begin{align}
	\left|c_{r}^{n,(0)}(\overline{x}_{2m(r-1)+1})\right|
	&\lesssim 
	1\vee\left(
    \Delta_{n}^{-2(H_{2r}-H_{2r-1})_{+} -\epsilon_{1}}
	|x_{1}|^{-2(H_{2r-1}-H_{2r})_{+} -\epsilon_{2}}
    \Indi_{(0,\infty)}(\sigma_{2r-1}) \right), 
	\label{key_ineq1} \\
	\left|c_{r}^{n,(1)}(\overline{x}_{2m(r-1)+1})\right|
	&\lesssim 
	\left[1\vee\left(
    \Delta_{n}^{-2(H_{2r}-H_{2r-1})_{+} -\epsilon_{1}}
	|x_{1}|^{-2(H_{2r-1}-H_{2r})_{+} -\epsilon_{2}}
    \Indi_{(0,\infty)}(\sigma_{2r-1}) \right)
    \right]
	\sum_{w=2}^{2m}\sum_{v=1}^{w-1}\left|x_{2m(r-1)+v+1}\right|
	\label{key_ineq2}
\end{align}
uniformly on compact subsets of $A\times\Theta_{\ast}$. 
We firstly verify the inequality \eqref{key_ineq1}.  
By a straightforward calculation, we can show that for any $\epsilon_{1},\epsilon_{2}>0$, 
\begin{align*}
	\left|c_{r}^{n,(0)}(\overline{x}_{2m(r-1)+1})\right| 
	&=\left|\frac{k_{r,\theta_{2r-1}}^{n}(\overline{x}_{2m(r-1)+1})}{h_{\theta_{2r}}^{n}(\overline{x}_{2m(r-1)+1})}\right| \\
	&\leq\frac{|d_{r,\theta_{2r-1}}^{n}(\overline{x}_{2m(r-1)+1})|}{d_{\theta_{2r}}^{n}(\overline{x}_{2m(r-1)+1})}\Indi_{(0,\infty)}(\sigma_{2r-1})
	+\frac{|e_{r,\theta_{2r-1}}^{n}(\overline{x}_{2m(r-1)+1})|}{e_{\theta_{2r}}^{n}(\overline{x}_{2m(r-1)+1})}\Indi_{(0,\infty)}(\tau_{2r-1}) \\
	&\lesssim\Delta_{n}^{2(H_{2r-1}-H_{2r})-\epsilon_{1}}|\overline{x}_{2m(r-1)+1}|^{-2(H_{2r-1}-H_{2r})-\epsilon_{2}}\Indi_{(0,\infty)}(\sigma_{2r-1})
	 +\Indi_{(0,\infty)}(\tau_{2r-1}) \\
	&\lesssim
    1\vee\left(
    \Delta_{n}^{-2(H_{2r}-H_{2r-1})_{+} -\epsilon_{1}}
	|x_{1}|^{-2(H_{2r-1}-H_{2r})_{+} -\epsilon_{2}}
    \Indi_{(0,\infty)}(\sigma_{2r-1}) \right)
\end{align*}
uniformly on compact subsets of $A\times\Theta_{\ast}$. 
This completes the proof of \eqref{key_ineq1}.\\

Finally, we verify the inequality \eqref{key_ineq2}. 
First note that $c_{r}^{n,(1)}(\overline{x}_{2m(r-1)+1})$ can rewriten as
\begin{align}
	k_{r,\theta_{2r-1}}^{n}&(\overline{x}_{2m(r-1)+1}) 
	 \cdot\sum_{w=2}^{2m}\left(\prod_{v=w+1}^{2m}u_{\theta_{2r}}^{n,m}(\overline{x}_{2m(r-1)+1})\right)
	\left[
	u_{\theta_{2r}}^{n,m}(\overline{x}_{2m(r-1)+w})
	-u_{\theta_{2r}}^{n,m}(\overline{x}_{2m(r-1)+1})
	\right]
	\prod_{v=2}^{w-1}u_{\theta_{2r}}^{n,m}(\overline{x}_{2m(r-1)+v}) 
	\nonumber \\
	&=k_{r,\theta_{2r-1}}^{n}(\overline{x}_{2m(r-1)+1})
	\left[u_{\theta_{2r}}^{n,m}(\overline{x}_{2m(r-1)+1})\right]^{2m-1}
	\nonumber \\
	&\quad \cdot\sum_{w=2}^{2m} 
	\left[\sum_{v_{1}=1}^{w-1}\left(
	u_{\theta_{2r}}^{n,m}(\overline{x}_{2m(r-1)+v_{1}+1})
	-u_{\theta_{2r}}^{n,m}(\overline{x}_{2m(r-1)+v_{1}})
	\right)
	\right]
	u_{\theta_{2r}}^{n,m}(\overline{x}_{2m(r-1)+1})^{-1}
	\prod_{v_{2}=2}^{w-1}\frac{u_{\theta_{2r}}^{n,m}(\overline{x}_{2m(r-1)+v_{2}})}{u_{\theta_{2r}}^{n,m}(\overline{x}_{2m(r-1)+1})},
	\label{rep_c1}
\end{align}
where, for conciseness, we write
\begin{align*}
	\prod_{v=2m+1}^{2m}u_{\theta_{2r}}^{n,m}(\overline{x}_{2m(r-1)+1})=
	\prod_{v=2}^{1}u_{\theta_{2r}}^{n,m}(\overline{x}_{2m(r-1)+v})=
	\prod_{v=2}^{1}\frac{u_{\theta_{2r}}^{n,m}(\overline{x}_{2m(r-1)+v})}{u_{\theta_{2r}}^{n,m}(\overline{x}_{2m(r-1)+1})}=1.
\end{align*}
Moreover, using the inequality \eqref{key_inequality_21}, we can also show that
\begin{align}
	\prod_{v=2}^{w-1}\frac{h_{\theta_{2r}}^{n}(\overline{x}_{2m(r-1)+v})}{h_{\theta_{2r}}^{n}(\overline{x}_{2m(r-1)+1})} 
	\nonumber
	&\leq\prod_{v=2}^{w-1}\left[
	\frac{d_{\theta_{2r}}^{n}(\overline{x}_{2m(r-1)+v})}{d_{\theta_{2r}}^{n}(\overline{x}_{2m(r-1)+1})}
	+\frac{e_{\theta_{2r}}^{n}(\overline{x}_{2m(r-1)+v})}{e_{\theta_{2r}}^{n}(\overline{x}_{2m(r-1)+1})}
	\right]
	 \\
	&\lesssim 
	\prod_{v=2}^{w-1}\left[
	|\overline{x}_{2m(r-1)+v}|^{1-2H_{2r}}
	|\overline{x}_{2m(r-1)+1}|^{2H_{2r}-1} 
	+|\overline{x}_{2m(r-1)+v}|^{2(K+1)}
	|\overline{x}_{2m(r-1)+1}|^{-2(K+1)}
	\right]
	\lesssim 1
	\label{ineq_ell_ratio}
\end{align}
uniformly on compact subsets of $A\times\Theta_{\ast}$. 
Therefore, using \eqref{key_ineq1}, \eqref{rep_c1} and \eqref{ineq_ell_ratio}, we can see that \eqref{key_ineq2} follows once we have proved
\begin{align}\label{key_ineq1_1}
	\left|
	\sum_{w=2}^{2m}\left[\sum_{v=1}^{w-1}\left(
	u_{\theta_{2r}}^{n,m}(\overline{x}_{2m(r-1)+v+1})
	-u_{\theta_{2r}}^{n,m}(\overline{x}_{2m(r-1)+v})
	\right)\right]
	u_{\theta_{2r}}^{n,m}(\overline{x}_{2m(r-1)+1})^{-1}
	\right| 
	\lesssim 
	|x_{1}|^{-1}
	\sum_{w=2}^{2m}\sum_{v=1}^{w-1}\left|x_{2m(r-1)+v+1}\right|
\end{align}
uniformly on compact subsets of $A\times\Theta_{\ast}$.
Note that $(x_{1},\cdots,x_{2mq})\in{W}^{c}\cap{H}_{1}^{c}$ also implies $|\overline{x}_{s}+z(\overline{x}_{s+1}-\overline{x}_{s})|\geq c_{-}|\overline{x}_{s}|>0$ for any $z\in[0,1]$ and $s\in\{1,2,\cdots,2mq-1\}$. 
Thus, using the inequality \eqref{key_inequality_21} and the mean-value theorem, we can show that for each $v\in\{1,2,\cdots,2m-1\}$,
\begin{align}
	\left|
	u_{\theta_{2r}}^{n,m}(\overline{x}_{2m(r-1)+v+1})
	-u_{\theta_{2r}}^{n,m}(\overline{x}_{2m(r-1)+v})
	\right| 
	\leq
	\left|x_{2m(r-1)+v+1}\right|
	\sup_{c_{-}|\overline{x}_{2m(r-1)+v+1}|\leq y\leq c_{+}|\overline{x}_{2m(r-1)+v+1}|}
	\left|\partial_{x}u_{\theta_{2r}}^{n,m}(y)\right|.
	\label{ineq12_c1}
\end{align}
Moreover, we can also show
\begin{equation*}
	\left|\partial_{x}u_{\theta}^{n,m}(x)\right|
	\lesssim \left|h_{\theta}^{n}(x)^{-\frac{1}{2m-1}-1}\partial_{x}h_{\theta}^{n}(x)\right|
	\lesssim \left(\frac{|\partial_{x}f_{\xi}(x)|}{f_{\xi}(x)}+\frac{|\partial_{x}g_{\tau}(x)|}{g_{\tau}(x)}\right)u_{\theta}^{n,m}(x)
	\lesssim |x|^{-1}u_{\theta}^{n,m}(x)
 \end{equation*}
 and 
 \begin{equation}\label{ineq_ratio-u}
	u_{\theta}^{n,m}(y)u_{\theta}^{n,m}(x)^{-1}
	\lesssim
	\left[\frac{f_{\xi}(x)}{f_{\xi}(y)}+\frac{g_{\tau}(x)}{g_{\tau}(y)}\right]^{\frac{1}{2m-1}}
	\lesssim
	\left[|y|^{1-2H}|x|^{2H-1} + |y|^{2(K+1)}|x|^{-2(K+1)}\right]^{\frac{1}{2m-1}}.
 \end{equation}
uniformly on compact subsets of $A\times\Theta_{\ast}$. 
Thus, using  the inequalities \eqref{key_inequality_21}, \eqref{key_ineq1}, \eqref{ineq12_c1}-\eqref{ineq_ratio-u}, we can also show that
\begin{align}
	&\left|\sum_{w=2}^{2m}\left[\sum_{v=1}^{w-1}\left(
	u_{\theta_{2r}}^{n,m}(\overline{x}_{2m(r-1)+v+1})
	-u_{\theta_{2r}}^{n,m}(\overline{x}_{2m(r-1)+v})
	\right)\right]u_{\theta_{2r}}^{n,m}(\overline{x}_{2m(r-1)+1})^{-1}\right| 
	\nonumber \\
	&\lesssim
	\sum_{w=2}^{2m}\sum_{v=1}^{w-1}\left|x_{2m(r-1)+v+1}\right|
	\sup_{c_{-}|\overline{x}_{2m(r-1)+v+1}|\leq y\leq c_{+}|\overline{x}_{2m(r-1)+v+1}|}
	 \left[\left|\partial_{x}u_{\theta_{2r}}^{n,m}(y)\right|u_{\theta_{2r}}^{n,m}(\overline{x}_{2m(r-1)+v+1})^{-1}\right]
	\nonumber \\
	&\lesssim
	\sum_{w=2}^{2m}\sum_{v=1}^{w-1}\left|x_{2m(r-1)+v+1}\right||\overline{x}_{2m(r-1)+v+1}|^{-1}
	\lesssim
	|x_{1}|^{-1}\sum_{w=2}^{2m}\sum_{v=1}^{w-1}\left|x_{2m(r-1)+v+1}\right|
	\label{trace-app_key_ineq2}
\end{align}
uniformly on compact subsets of $A\times\Theta_{\ast}$. 
Then \eqref{key_ineq1_1} follows from \eqref{ineq12_c1} and \eqref{trace-app_key_ineq2}. 
	\end{proof}
We are now ready to prove Lemma~\ref{key_proposition_lemma} \eqref{key_Prop:Wc}. 
Since $\varpi\neq (0,\cdots,0)$, there exists $r_{0}\in\{1,\cdots,q\}$ such that $\varpi_{r_{0}}=1$. 
Then, using Lemma~\ref{key_lemma_W-tilda-c}, the definition of $W$ in \eqref{def-W} and \eqref{key_inequality_21}, we can show
\begin{equation}\label{key_inequality_Wc}
	\Delta_{n}^{\rho_{q,-}(\underline{\xi})+\epsilon_{1}}\left|\prod_{r=1}^{q}c_{r,n}^{(\varpi_{r})}(\overline{x}_{2m(r-1)+1})\right|
			\lesssim
    |x_{1}|^{-\psi_{q,-}(\underline{\xi})_{+}-\epsilon_{2}-1}
	\sum_{w=2}^{2m}\sum_{v=1}^{w-1}\left|x_{2m(r_{0}-1)+v+1}\right|
\end{equation}
uniformly on compact subsets of $A\times\Theta_{1}(\iota)$ for any $\epsilon_{1},\epsilon_{2}>0$ and $\iota\in(0,1)$ since $\Theta_{1}(\iota)$ is a subset of $\Theta_{\ast}$. 
The rest of the proof of Lemma~\ref{key_proposition_lemma} \eqref{key_Prop:Wc} is similar to the proof of Proposition~4 (2) of \cite{Takabatake-2023-JTSA-supplement} using the inequality \eqref{key_inequality_Wc} instead of the inequality (5) of \cite{Takabatake-2023-JTSA-supplement}. We omit the details for conciseness. 
		
\subsubsection{Proof of Lemma~\ref{key_proposition_lemma}~\eqref{key_Prop:main_term}}
We can also prove Lemma~\ref{key_proposition_lemma} (\ref{key_Prop:main_term}) proceeding as the proof of Proposition~4 (3) of \cite{Takabatake-2023-JTSA-supplement} using Lemma 2 of \cite{Takabatake-2023-JTSA-supplement} and using Lemma~\ref{key_proposition_lemma} \eqref{key_Prop:W}-\eqref{key_Prop:Wc} instead of Proposition~4 (1)-(2) of \cite{Takabatake-2023-JTSA-supplement}. 
We again omit the details for conciseness.

\subsection{Proof of Proposition~\ref{Prop:LRR10_Thm5}}
		 Using Proposition~\ref{Prop:T23_JTSA}, it suffices to prove that
		 \begin{equation}\label{order-En2}
		 	\Delta_{n}^{\rho_{q}(\underline{\xi})+\epsilon_{1}}n^{-\psi_{p}(\underline{\xi})-\epsilon_{2}}
			E_{n}^{(2)}(\underline{\theta})
			=o(1)\ \ \mbox{as $n\to\infty$}
		 \end{equation}
		 uniformly on compact subsets of $\Theta_{1}(\iota)$ for any $\epsilon_{1},\epsilon_{2}>0$, where $E_{n}^{(2)}(\underline{\theta})$ is defined in \eqref{def-En12}. By definition of $A_{r}^{n}$ and $B_{r}^{n}$, we have
		  \begin{equation*}
		 	A_{r}^{n}(\underline{\theta})-B_{r}^{n}(\underline{\theta})=A_{r}^{n}(\underline{\theta})D_{r}^{n}(\underline{\theta})=(A_{r}^{n}(\underline{\theta})-B_{r}^{n}(\underline{\theta}))D_{r}^{n}(\underline{\theta})+B_{r}^{n}(\underline{\theta})D_{r}^{n}(\underline{\theta})
		 \end{equation*}
		 and therefore
		 \begin{align*}
		 	E_{n}^{(2)}(\underline{\theta})
		 	&=\bigg|\sum_{r=1}^{q}\mathrm{Tr}\bigg[\bigg(\prod_{s=1}^{r-1}B_{s}^{n}(\underline{\theta})\bigg)(A_{r}^{n}(\underline{\theta})-B_{r}^{n}(\underline{\theta}))\bigg(\prod_{s=r+1}^{q}A_{s}^{n}(\underline{\theta})\bigg)\bigg]\bigg|\leq\sum_{r=1}^{q}\bigg(v_{r,1}^{n}(\underline{\theta})+v_{r,2}^{n}(\underline{\theta})+v_{r,3}^{n}(\underline{\theta})+v_{r,4}^{n}(\underline{\theta})\bigg)
		 \end{align*}
where we write
\begin{align*}
   \begin{cases}
		 	v_{r,1}^{n}(\underline{\theta}):=\left|\mathrm{Tr}\left[\left(\prod_{s=1}^{r-1}A_{s}^{n}(\underline{\theta})\right)(A_{r}^{n}(\underline{\theta})-B_{r}^{n}(\underline{\theta}))D_{r}^{n}(\underline{\theta})\left(\prod_{s=r+1}^{q}A_{s}^{n}(\underline{\theta})\right)\right]\right|,\\
		 	v_{r,2}^{n}(\underline{\theta}):=\left|\mathrm{Tr}\left[\left(\prod_{s=1}^{r-1}B_{s}^{n}(\underline{\theta})-\prod_{s=1}^{r-1}A_{s}^{n}(\underline{\theta})\right)(A_{r}^{n}(\underline{\theta})-B_{r}^{n}(\underline{\theta}))D_{r}^{n}(\underline{\theta})\left(\prod_{s=r+1}^{q}A_{s}^{n}(\underline{\theta})\right)\right]\right|,\\
		 	v_{r,3}^{n}(\underline{\theta}):=\left|\mathrm{Tr}\left[\left(\prod_{s=1}^{r-1}B_{s}^{n}(\underline{\theta})\right)B_{r}^{n}(\underline{\theta})D_{r}^{n}(\underline{\theta})\left(\prod_{s=r+1}^{q}B_{s}^{n}(\underline{\theta})\right)\right]\right|,\\ 
		 	v_{r,4}^{n}(\underline{\theta}):=\left|\mathrm{Tr}\left[\left(\prod_{s=1}^{r-1}B_{s}^{n}(\underline{\theta})\right)B_{r}^{n}(\underline{\theta})D_{r}^{n}(\underline{\theta})\left(\prod_{s=r+1}^{q}A_{s}^{n}(\underline{\theta})-\prod_{s=r+1}^{q}B_{s}^{n}(\underline{\theta})\right)\right]\right|,
\end{cases}
\end{align*}
and where $\prod_{s=p+1}^{q}A_{s}^{n}(\underline{\theta}):=I_{n}$ and $\prod_{s=p+1}^{q}B_{s}^{n}(\underline{\theta}):=I_{n}$ for conciseness. In the rest of this proof, we evaluate the four terms $v_{r,1}^{n}(\underline{\theta})$, $v_{r,2}^{n}(\underline{\theta})$, $v_{r,3}^{n}(\underline{\theta})$ and $v_{r,4}^{n}(\underline{\theta})$ as $n\to\infty$ for each $r\in\{1,2,\cdots,q\}$ separately.\\
\par\noindent
\textit{Estimate of $v_{r,1}^{n}(\underline{\theta})$}: First note that 
\begin{equation}
		 	(A_{r}^{n}(\underline{\theta})-B_{r}^{n}(\underline{\theta}))D_{r}^{n}(\underline{\theta})
		 	=A_{r}^{n}(\underline{\theta})D_{r}^{n}(\underline{\theta})^{2} =\Sigma_{n}(k_{r,\theta_{2r-1}}^{n})^{\frac{1}{2}}V_{1,r}^{n}(\underline{\theta})\widetilde{D}_{r}^{n}(\underline{\theta})^{2}V_{2,r}^{n}(\underline{\theta})\Sigma_{n}(k_{r+1,\theta_{2r+1}}^{n})^{-\frac{1}{2}}. \label{decomposition_LRR10_2}
		 \end{equation}
Moreover, using that $\widetilde{A}_{r}^{n}(\underline{\theta}) =  V_{1,r}^{n}(\underline{\theta}) V_{2,r}^{n}(\underline{\theta})$ and the properties of Frobenius and operator norms, we have
		 \begin{align*}
		 	v_{r,1}^{n}(\underline{\theta})
		 	&=\bigg|\mathrm{Tr}\bigg[\bigg(\prod_{s=1}^{r-1}\widetilde{A}_{s}^{n}(\underline{\theta})\bigg)
		 	V_{1,r}^{n}(\underline{\theta})\widetilde{D}_{r}^{n}(\underline{\theta})^{2}V_{2,r}^{n}(\underline{\theta})
		 	\bigg(\prod_{s=r+1}^{q}\widetilde{A}_{s}^{n}(\underline{\theta})\bigg)\bigg]\bigg|\\
		 	&\leq\bigg\|\bigg(\prod_{s=1}^{r-1}\widetilde{A}_{s}^{n}(\underline{\theta})\bigg)V_{1,r}^{n}(\underline{\theta})\widetilde{D}_{r}^{n}(\underline{\theta})\bigg\|_{\mathrm{F}}
		 	\bigg\|\widetilde{D}_{r}^{n}(\underline{\theta})V_{2,r}^{n}(\underline{\theta})\bigg(\prod_{s=r+1}^{q}\widetilde{A}_{s}^{n}(\underline{\theta})\bigg)\bigg\|_{\mathrm{F}}\\
		 	&\leq\left\|\widetilde{D}_{r}^{n}(\underline{\theta})\right\|_{\mathrm{F}}^{2}\prod_{s=1}^{q}\left\|J_{1,s}^{n}(\theta)\right\|_{\mathrm{op}}\left\|J_{2,s}^{n}(\theta)\right\|_{\mathrm{op}}.
		 \end{align*}
		 Therefore, we obtain
		 \begin{equation}\label{upper-bound-aj}
			\Delta_{n}^{\rho_{q}(\underline{\xi})+\epsilon_{1}}n^{-\psi_{p}(\underline{\xi})-\epsilon_{2}}
			v_{r,1}^{n}(\underline{\theta})=o(1)\ \ \mbox{as $n\to\infty$}
		 \end{equation}
		 uniformly on compact subsets of $\Theta_{\ast}$ for any $\epsilon_{1},\epsilon_{2}>0$
		  and $r\in\{1,2,\cdots,q\}$ using the two following lemmas.

\begin{lemma}\label{Lemma_T22_QWLE_Ext}
	Let $m\in\mathbb{N}$ satisfying $m\geq K+3/2$. 
	Under the same assumptions in Proposition~\ref{Prop:LRR10_Thm5}, we obtain
	\begin{equation*}
		\left\|I_{n}-\Sigma_{n}(h_{\theta}^{n})^{\frac{1}{2}}\Gamma_{n,m}(\theta)\Sigma_{n}(h_{\theta}^{n})^{\frac{1}{2}}\right\|_{\mathrm{F}}^{2}=o(n^{\epsilon})\ \ \mbox{as $n\to\infty$}
	\end{equation*}
	uniformly on compact subsets of $(0,1)\times(0,\infty)^{2}$ for any $\epsilon>0$.
\end{lemma}
\begin{proof}[Proof of Lemma~\ref{Lemma_T22_QWLE_Ext}]
	First note that we can write
	\begin{align*}
		\left\|I_{n}-\Sigma_{n}(h_{\theta}^{n})^{\frac{1}{2}}\Gamma_{n,m}(\theta)\Sigma_{n}(h_{\theta}^{n})^{\frac{1}{2}}\right\|_{\mathrm{F}}^{2}
  &=n-2\mathrm{Tr}\left[\Sigma_{n}(h_{\theta}^{n})\Gamma_{n,m}(\theta)\right]+\mathrm{Tr}\left[\Sigma_{n}(h_{\theta}^{n})\Gamma_{n,m}(\theta)\Sigma_{n}(h_{\theta}^{n})\Gamma_{n,m}(\theta)\right]\\\
		&=-2\left(\mathrm{Tr}\left[\Sigma_{n}(h_{\theta}^{n})\Gamma_{n,m}(\theta)\right]-n\right)
		+\left(\mathrm{Tr}\left[\Sigma_{n}(h_{\theta}^{n})\Gamma_{n,m}(\theta)\Sigma_{n}(h_{\theta}^{n})\Gamma_{n,m}(\theta)\right]-n\right).
	\end{align*}
	Therefore, the conclusion follows from Proposition~\ref{Prop:T23_JTSA}.
\end{proof}
\begin{lemma}\label{Lemma_Operator_{n}orm}
	Under the same assumptions in Proposition~\ref{Prop:LRR10_Thm5}, we obtain
	\begin{equation*}
		\bigl\|V_{w,r}^{n}(\underline{\theta})\bigr\|_{\mathrm{op}}^{2}
		\lesssim 
		\begin{cases}
		   1& \mbox{if $\sigma_{2r-1}=0$},\\
           \Delta_{n}^{-2(H_{2r}-H_{2r+2w-1})_{+} -\epsilon_{1}}n^{2(H_{2r+2w-1}-H_{2r})_{+} +\epsilon_{2}}
           & \mbox{if $\sigma_{2r-1}>0$}
		\end{cases}
	\end{equation*}
	 as $n\to\infty$ uniformly on compact subsets of $\Theta_{\ast}$ for any $w\in\{0,1\}$, $r\in\{1,2,\cdots,q\}$ and $\epsilon_{1},\epsilon_{2}>0$.
\end{lemma}
\begin{proof}[Proof of Lemma~\ref{Lemma_Operator_{n}orm}]
	Fix $r\in\{1,2,\cdots,q\}$. 	
	Since $d_{r,\xi}^{n}(x)$ and $e_{r,\tau}^{n}(x)$ are non-negative functions and we have $\bigl\|V_{2,r}^{n}(\underline{\theta})\bigr\|_{\mathrm{op}}=\bigl\|V_{2,r}^{n}(\underline{\theta})^{\top}\bigr\|_{\mathrm{op}}$, straightforward calculations provide that for each $w\in\{0,1\}$,
	\begin{align*}
		\bigl\|V_{w,r}^{n}(\underline{\theta})\bigr\|_{\mathrm{op}}^{2}
		&=
		\sup_{u\in\mathbb{R}^{n}\setminus\{0\}}\frac{u^{\top}\Sigma_{n}(k_{r+w,\theta_{2r+2w-1}}^{n})u}{u^{\top}\Sigma_{n}(h_{\theta_{2r}}^{n})u}\\
		&\leq
		\Indi_{(0,\infty)}(\sigma_{2r-1})\sup_{u\in\mathbb{R}^{n}\setminus\{0\}}\frac{u^{\top}\Sigma_{n}(d_{r+w,\xi_{2r+2w-1}}^{n})u}{u^{\top}\Sigma_{n}(f_{\xi_{2r}}^{n})u} +\Indi_{(0,\infty)}(\tau_{2r-1})\sup_{u\in\mathbb{R}^{n}\setminus\{0\}}\frac{u^{\top}\Sigma_{n}(e_{r+w,\tau_{2r+2w-1}}^{n})u}{u^{\top}\Sigma_{n}(g_{\tau_{2r}}^{n})u}\\
		&\lesssim 
		\Indi_{(0,\infty)}(\sigma_{2r-1})\Delta_{n}^{2(H_{2r+2w-1}-H_{2r})-\epsilon_{1}}
		\left\|\Sigma_{n}(d_{r+w,H_{2r+2w-1}})^{\frac{1}{2}}\Sigma_{n}(f_{H_{2r}})^{-\frac{1}{2}}\right\|_{\mathrm{op}}^{2} \\
		&\quad +\Indi_{(0,\infty)}(\tau_{2r-1})\left\|\Sigma_{n}(e_{r+w})^{\frac{1}{2}}\Sigma_{n}(g)^{-\frac{1}{2}}\right\|_{\mathrm{op}}^{2}
	\end{align*}
	uniformly on compact subsets of $\Theta_{\ast}$. 
	Then the conclusion follows from Lemma 5.3 of \cite{Dahlhaus-1989} and Lemma 2 in the full version of \cite{Lieberman-Rosemarin-Rousseau-2012}. 
\end{proof}
    
		 \par\noindent
		 \textit{Estimate of $v_{r,2}^{n}(\underline{\theta})$}: Using \eqref{decomposition_LRR10_2}, that $\widetilde{A}_{r}^{n}(\underline{\theta}) = V_{1,r}^{n}(\underline{\theta}) V_{2,r}^{n}(\underline{\theta})$, the properties of Frobenius and operator norms, and the Cauchy-Schwarz inequality, we can show 
		 \begin{align*}
		 	v_{r,2}^{n}(\underline{\theta})&=\left|\mathrm{Tr}\left[\left(\prod_{s=1}^{r-1}\widetilde{B}_{s}^{n}(\underline{\theta})-\prod_{s=1}^{r-1}\widetilde{A}_{s}^{n}(\underline{\theta})\right)
		 	V_{1,r}^{n}(\underline{\theta})\widetilde{D}_{r}^{n}(\underline{\theta})^{2}V_{2,r}^{n}(\underline{\theta})
		 	\left(\prod_{s=r+1}^{q}\widetilde{A}_{s}^{n}(\underline{\theta})\right)\right]\right|\\
		 	&\leq\left\|\prod_{s=1}^{r-1}\widetilde{A}_{s}^{n}(\underline{\theta})-\prod_{s=1}^{r-1}\widetilde{B}_{s}^{n}(\underline{\theta})\right\|_{\mathrm{F}}
		 	\left\|V_{1,r}^{n}(\underline{\theta})\widetilde{D}_{r}^{n}(\underline{\theta})^{2}V_{2,r}^{n}(\underline{\theta})
		 	\left(\prod_{s=r+1}^{q}\widetilde{A}_{s}^{n}(\underline{\theta})\right)\right\|_{\mathrm{F}}\\
		 	&\leq\left\|\prod_{s=1}^{r-1}\widetilde{A}_{s}^{n}(\underline{\theta})-\prod_{s=1}^{r-1}\widetilde{B}_{s}^{n}(\underline{\theta})\right\|_{\mathrm{F}}\left\|\widetilde{D}_{r}^{n}(\underline{\theta})\right\|_{\mathrm{F}}^{2}
		 	\prod_{s=r}^{q}\left\|J_{1,s}^{n}(\theta)\right\|_{\mathrm{op}}\left\|J_{2,s}^{n}(\theta)\right\|_{\mathrm{op}}.
		 \end{align*}
		 Therefore, we obtain
		 \begin{equation}\label{upper-bound-bj}
		 	\Delta_{n}^{\rho_{q}(\underline{\xi})+\epsilon_{1}}
		 	n^{-\psi_{p}(\underline{\xi})-\epsilon_{2}}
			v_{r,2}^{n}(\underline{\theta})=o(1)\ \ \mbox{as $n\to\infty$}
		 \end{equation}
		 uniformly on compact subsets of $\Theta_{\ast}$ for any $\epsilon_{1},\epsilon_{2}>0$ and $r\in\{1,2,\cdots,q\}$ using Lemmas~\ref{Lemma_T22_QWLE_Ext}, \ref{Lemma_Operator_{n}orm} and the following lemma.

\begin{lemma}\label{LRR_lemmaB5_3}
		Let $q\in\mathbb{N}$. For any $q_{0}\in\mathbb{N}$ satisfying $1\leq q_{0}\leq q$ and $\epsilon_{1},\epsilon_{2}>0$, we obtain
		\begin{equation}\label{LRR_lemmaB5_Key3}
			\Delta_{n}^{\rho_{q_{0}}(\underline{\xi})+\epsilon_{1}}
			n^{-\psi_{q_{0}}(\underline{\xi})-\epsilon_{2}}
			\left\|\prod_{r=1}^{q_{0}}\widetilde{A}_{r}^{n}(\underline{\theta})-\prod_{r=1}^{q_{0}}\widetilde{B}_{r}^{n}(\underline{\theta})\right\|_{\mathrm{F}}=o(1)\ \ \mbox{as $n\to\infty$} 
		\end{equation}
		uniformly on compact subsets of $\Theta_{\ast}$. 
	\end{lemma}
\begin{proof}[Proof of Lemma~\ref{LRR_lemmaB5_3}]
First note that we can write 
\begin{equation*}
	\widetilde{A}_{r}^{n}(\underline{\theta})-\widetilde{B}_{r}^{n}(\underline{\theta})=V_{1,r}^{n}(\underline{\theta})\widetilde{D}_{r}^{n}(\underline{\theta})V_{2,r}^{n}(\underline{\theta}),
\end{equation*}
so that Lemmas~\ref{Lemma_T22_QWLE_Ext} and~\ref{Lemma_Operator_{n}orm} provide that for any $r\in\{1,2,\cdots,q\}$ and $\epsilon_{1},\epsilon_{2}>0$,
\begin{align*}
	\left\|\widetilde{A}_{r}^{n}(\underline{\theta})-\widetilde{B}_{r}^{n}(\underline{\theta})\right\|_{\mathrm{F}}
	&\leq\left\|\widetilde{D}_{r}^{n}(\underline{\theta})\right\|_{\mathrm{F}}\prod_{w=1}^{2}\left\|V_{w,r}^{n}(\underline{\theta})\right\|_{\mathrm{op}}\\
	&\lesssim 
	\Delta_{n}^{-\sum_{w=0}^{1}(H_{2r}-H_{2r+2w-1})_{+}\Indi_{(0,\infty)}(\sigma_{2r+2w-1}) -\epsilon_{1}}
	n^{\sum_{i=0}^{1}(H_{2r+2w-1}-H_{2r})_{+}\Indi_{(0,\infty)}(\sigma_{2r+2w-1}) +\epsilon_{2}}
\end{align*}
uniformly on compact subsets of $\Theta_{\ast}$ as $n\to\infty$. 
Then we conclude \eqref{LRR_lemmaB5_Key3} in the case $q_{0}=1$. 
In the rest of this proof, we will prove \eqref{LRR_lemmaB5_Key3} in the case $q_{0}\geq 2$ by mathematical induction. 
In the following, we assume that for all $q_{1}\in\mathbb{N}$ satisfying $1\leq q_{1}<q_{0}$ and $\epsilon_{1},\epsilon_{2}>0$, it holds
\begin{equation}\label{LRR_lemB5_Key3_mathematical_induction}
	\Delta_{n}^{\rho_{q_{1}}(\underline{\xi})+\epsilon_{1}}
	n^{-\psi_{q_{1}}(\underline{\xi})-\epsilon_{2}}
	\left\|\prod_{r=1}^{q_1}\widetilde{A}_{r}^{n}(\underline{\theta})-\prod_{r=1}^{q_1}\widetilde{B}_{r}^{n}(\underline{\theta})\right\|_{\mathrm{F}}=o(1)\ \ \mbox{as $n\to\infty$}
\end{equation}
uniformly on compact subsets of $\Theta_{\ast}$. Using the same decomposition used in the equation (13) of \cite{Dahlhaus-1989} and elementary properties of the Frobenius norm and the operator norm, we can show
\begin{align}
	&\left\|\prod_{r=1}^{q_{1}}\widetilde{A}_{r}^{n}(\underline{\theta})-\prod_{r=1}^{q_{1}}\widetilde{B}_{r}^{n}(\underline{\theta})\right\|_{\mathrm{F}} 
	\nonumber =\left\|\sum_{r=1}^{q_{1}}\left(\prod_{s=1}^{r-1}\widetilde{B}_{s}^{n}(\underline{\theta})\right)\left(\widetilde{A}_{r}^{n}(\underline{\theta})-\widetilde{B}_{r}^{n}(\underline{\theta})\right)\prod_{s=r+1}^{q_{1}}\widetilde{A}_{s}^{n}(\underline{\theta})\right\|_{\mathrm{F}} 
	\nonumber \\
	&\quad\leq\sum_{r=1}^{q_{1}}\left\|\prod_{s=1}^{r-1}\widetilde{B}_{s}^{n}(\underline{\theta})\right\|_{\mathrm{op}}
	\left\|\left(\widetilde{A}_{r}^{n}(\underline{\theta})-\widetilde{B}_{r}^{n}(\underline{\theta})\right)\prod_{s=r+1}^{q_{1}}\widetilde{A}_{s}^{n}(\underline{\theta})\right\|_{\mathrm{F}} 
	\nonumber \\
	&\quad\leq\sum_{r=1}^{q_{1}}\left(\prod_{s=1}^{r-1}\left\|\widetilde{A}_{s}^{n}(\underline{\theta})\right\|_{\mathrm{op}} +\left\|\prod_{s=1}^{r-1}\widetilde{A}_{s}^{n}(\underline{\theta})-\prod_{s=1}^{r-1}\widetilde{B}_{s}^{n}(\underline{\theta})\right\|_{\mathrm{F}}\right)
	\left\|\widetilde{A}_{r}^{n}(\underline{\theta})-\widetilde{B}_{r}^{n}(\underline{\theta})\right\|_{\mathrm{F}}\prod_{s=r+1}^{q_{1}}\left\|\widetilde{A}_{s}^{n}(\underline{\theta})\right\|_{\mathrm{op}}. 
	\label{ineq_lem5_1}
\end{align}
Therefore, we conclude \eqref{LRR_lemmaB5_Key3} in the case $q_{0}\geq 2$ from Lemma~\ref{Lemma_Operator_{n}orm}, the Assumption \eqref{LRR_lemB5_Key3_mathematical_induction} of the mathematical induction, Equation \eqref{ineq_lem5_1} and 
\begin{equation*}
	\left\|\widetilde{A}_{r}^{n}(\underline{\theta})\right\|_{\mathrm{op}}\leq\left\|V_{1,r}^{n}(\underline{\theta})\right\|_{\mathrm{op}}\left\|V_{2,r}^{n}(\underline{\theta})\right\|_{\mathrm{op}}.
\end{equation*}
This completes the proof. 
\end{proof}

		 \par\noindent
		 \textit{Estimate of $v_{r,3}^{n}(\underline{\theta})$}: First note that
		 \begin{align*}
		 	v_{r,3}^{n}(\underline{\theta})=&\left|\mathrm{Tr}\left[\prod_{s=1}^{q}B_{s}^{n}(\underline{\theta})\right]-\mathrm{Tr}\left[\left(\prod_{s=1}^{r}B_{s}^{n}(\underline{\theta})\right)\Sigma_{n}(h_{\theta_{2r}}^{n})\Gamma_{n,m}(\theta)\left(\prod_{s=r+1}^{q}B_{s}^{n}(\underline{\theta})\right)\right]\right|\\
		 	\leq&\left|\mathrm{Tr}\left[\prod_{s=1}^{q}B_{s}^{n}(\underline{\theta})\right] 
		 	-\frac{n}{2\pi}\int_{-\pi}^{\pi}\prod_{r=1}^{q}\frac{k_{r,\theta_{2r-1}}^{n}(x)}{h_{\theta_{2r}}^{n}(x)}\,\mathrm{d}x\right|\\
		 	&+\left|\mathrm{Tr}\left[\left(\prod_{s=1}^{r}B_{s}^{n}(\underline{\theta})\right)\Sigma_{n}(h_{\theta_{2r}}^{n})\Gamma_{n,m}(\theta)\left(\prod_{s=r+1}^{q}B_{s}^{n}(\underline{\theta})\right)\right] 
		 	-\frac{n}{2\pi}\int_{-\pi}^{\pi}\prod_{r=1}^{q}\frac{k_{r,\theta_{2r-1}}^{n}(x)}{h_{\theta_{2r}}^{n}(x)}\,\mathrm{d}x\right|.
		 \end{align*}
		 Therefore, 
		 using Proposition~\ref{Prop:T23_JTSA}, we obtain
		 \begin{equation}\label{upper-bound-cj}
		 	\Delta_{n}^{\rho_{q}(\underline{\xi})+\epsilon_{1}}
		 	n^{-\psi_{p}(\underline{\xi})-\epsilon_{2}}
			v_{r,3}^{n}(\underline{\theta})=o(1)\ \ \mbox{as $n\to\infty$}
		 \end{equation}
		uniformly on compact subsets of $\Theta_{1}(\iota)$ for any $\epsilon_{1},\epsilon_{2}>0$, $\iota\in(0,1)$ and $r\in\{1,2,\cdots,q\}$.\\
		 \par\noindent
		 \textit{Estimate of $v_{r,4}^{n}(\underline{\theta})$}: Using the Cauchy-Schwarz inequality, we can show
		 \begin{equation}\label{upper-bound-dj-key1}
		 	v_{r,4}^{n}(\underline{\theta})
		 	\leq\left\|\left(\prod_{s=1}^{r}\widetilde{B}_{s}^{n}(\underline{\theta})\right)\widetilde{D}_{r}^{n}(\underline{\theta})\right\|_{\mathrm{F}}
		 	\left\|\prod_{s=r+1}^{q}\widetilde{A}_{s}^{n}(\underline{\theta})-\prod_{s=r+1}^{q}\widetilde{B}_{s}^{n}(\underline{\theta})\right\|_{\mathrm{F}}.
		 \end{equation}
		 Moreover, we can also prove
		 \begin{equation}\label{upper-bound-dj-key2}
		 	\left\|\left(\prod_{s=1}^{r}\widetilde{B}_{s}^{n}(\underline{\theta})\right)\widetilde{D}_{r}^{n}(\underline{\theta})\right\|_{\mathrm{F}}^{2}=o(n^{\epsilon})\ \ \mbox{as $n\to\infty$}
		 \end{equation}
		 uniformly on compact subsets of $\Theta_{1}(\iota)$ for any $\epsilon>0$, $\iota\in(0,1)$ and $r\in\{1,2,\cdots,q\}$ using Proposition~\ref{Prop:T23_JTSA} proceeding as the proof of the term $v_{r,3}^{n}(\underline{\theta})$. 
		 Therefore, using \eqref{upper-bound-dj-key1}, \eqref{upper-bound-dj-key2} and Lemma~\ref{LRR_lemmaB5_3}, we can show
		 \begin{equation}\label{upper-bound-dj}
		 	\Delta_{n}^{\rho_{q}(\underline{\xi})+\epsilon_{1}}
		 	n^{-\psi_{p}(\underline{\xi})-\epsilon_{2}}
			v_{r,4}^{n}(\underline{\theta})=o(1)\ \ \mbox{as $n\to\infty$}
		 \end{equation}
		 uniformly on compact subsets of $\Theta_{1}(\iota)$ for any $\epsilon_{1},\epsilon_{2}>0$, $\iota\in(0,1)$ and $r\in\{1,2,\cdots,q\}$.\\

We eventually obtain \eqref{order-En2} by combining the bounds \eqref{upper-bound-aj}, \eqref{upper-bound-bj}, \eqref{upper-bound-cj} and \eqref{upper-bound-dj}, which proves Proposition~\ref{Prop:LRR10_Thm5}.

\section{Proofs of Propositions~\ref{Prop:Limit_Key} and~\ref{Prop:FisherInfo-Cross-Terms}}\label{Section_Proof_Prop_Limit_Key}
\subsection{Proof of Proposition~\ref{Prop:Limit_Key}}
We first introduce some notation used in this Section. 
For sequences of functions $\{a_{\theta}^{n}(\lambda)\}_{n\in\mathbb{N}}$ and $\{b_{\theta}^{n}(\lambda)\}_{n\in\mathbb{N}}$ on $[-\pi,\pi]\times\Theta$ which are not equal to zero on $(0,\pi]\times\Theta$, write $a_{\theta}^{n}(\lambda)\sim b_{\theta}^{n}(\lambda)$ as $|\lambda|\to 0$ uniformly on $\Theta$ if
\begin{align*}
	\lim_{|\lambda|\to 0}\limsup_{n\to\infty}\sup_{\theta\in\Theta}\left|\frac{a_{\theta}^{n}(\lambda)}{b_{\theta}^{n}(\lambda)}\right|=0.
\end{align*}
For $\lambda\in[-\pi,\pi]\setminus\{0\}$, we write
$f_{H}^{\downarrow}(\lambda):=c_{H}|\lambda|^{1-2H}$, $g^{\downarrow}(\lambda):=|\lambda|^{2(K+1)}$ and 
\begin{align*}
	f_{\xi}^{n,\downarrow}(\lambda):=\sigma^{2}\Delta_{n}^{2H}c_{H}|\lambda|^{1-2H},\ \ 
	h_{\theta}^{n,\downarrow}(\lambda):=\sigma^{2}\Delta_{n}^{2H}c_{H}|\lambda|^{1-2H}+\tau^{2}\nu_{n}^{2}|\lambda|^{2(K+1)}.
\end{align*}
Note that $f_{H}(\lambda)\sim f_{H}^{\downarrow}(\lambda)$, $f_{\xi}^{n}(\lambda)\sim f_{\xi}^{n,\downarrow}(\lambda)$ and $h_{\theta}^{n}(\lambda)\sim h_{\theta}^{n,\downarrow}(\lambda)$ as $|\lambda|\to 0$ uniformly on $\Theta$.\\

Before proving Proposition~\ref{Prop:Limit_Key}, we prepare the following three lemmas.  
\begin{lemma}\label{PropC1_Ext1_20200812}
	Let $m\in\{0,1,2\}$ and $\{r_{n}^{+}(H)\}_{n\in\mathbb{N}}$ be a positive sequence satisfying
	\begin{equation*}
		r_{n}^{+}(H)=o(n) \;\;\text{ and }\;\; \left(\frac{r_{n}^{1}(H)}{r_{n}^{+}(H)}\right)^{2\Diamond(H)-1}\left|\log\left(\frac{n}{r_{n}^{1}(H)}\right)\right|^{2}=o(1) \mbox{ as } n\to\infty
	\end{equation*}
	uniformly on $\Theta_{H}$.
    Under the same assumptions in Proposition~\ref{Prop:Limit_Key}, we obtain
	\begin{align*}
		\frac{n}{r_{n}^{1}(H)}\int_{\frac{r_{n}^{+}(H)}{n}}^{\pi}\frac{f_{\xi}^{n}(\lambda)^{2}\left(-d_{n}(\theta)+\partial_{H}\log{f_{H}(\lambda)}\right)^{m}}{h_{\theta}^{n}(\lambda)^{2}}\,\mathrm{d}\lambda=o(1)\ \ \mbox{as $n\to\infty$}
	\end{align*}
    uniformly on $\Theta$.
\end{lemma}
\begin{proof}[Proof of Lemma~\ref{PropC1_Ext1_20200812}]
	Set $r_{n}^{+}(H) n^{-1}=r_{n}^{+}(H)/n$. 
	First note that by defintion of $r_{n}^{+}(H)$, we have $r_{n}^{+}(H) n^{-1}=o(1)$ and $r_{n}^{1}(H)=o(r_{n}^{+}(H))$ as $n\to\infty$ uniformly on $\Theta_{H}$ since  $\inf_{H\in\Theta_{H}}\Diamond(H)>1/2$. Therefore, $(n/r_{n}^{+}(H))\lesssim (n/r_{n}^{1}(H))$ and
	\begin{equation*}
		\sup_{\lambda\in(r_{n}^{+}(H) n^{-1},\pi]}\left|\partial_{H}\log{f_{H}(\lambda)}\right|
		\lesssim \sup_{\lambda\in(r_{n}^{+}(H) n^{-1},\pi]}\left|\log\lambda\right|
		\lesssim -\log{r_{n}^{+}(H) n^{-1}}
		\lesssim \log\left(\frac{n}{r_{n}^{1}(H)}\right).
	\end{equation*}
	uniformly on $\Theta_{H}$. Since we have $|d_{n}(\theta)|\lesssim \log\left(n/r_{n}^{1}(H)\right)$ uniformly on $\Theta$, we obtain
	\begin{align*}
		\int_{r_{n}^{+}(H) n^{-1}}^{\pi}\frac{f_{\xi}^{n}(\lambda)^{2}\left(-d_{n}(\theta)+\partial_{H}\log{f_{H}(\lambda)}\right)^{m}}{h_{\theta}^{n}(\lambda)^{2}}\,\mathrm{d}\lambda
		&\lesssim 
		(\nu_{n}\Delta_{n}^{-H})^4\left|\log\left(\frac{n}{r_{n}^{1}(H)}\right)\right|^{2}\int_{r_{n}^{+}(H) n^{-1}}^{\pi}\lambda^{-2\Diamond(H)}\,\mathrm{d}\lambda\\
		&\lesssim 
		(\nu_{n}\Delta_{n}^{-H})^4\left|\log\left(\frac{n}{r_{n}^{1}(H)}\right)\right|^{2}\left(1+(r_{n}^{+}(H) n^{-1})^{1-2\Diamond(H)}\right)\\
		&\lesssim 
		(\nu_{n}\Delta_{n}^{-H})^4\left|\log\left(\frac{n}{r_{n}^{1}(H)}\right)\right|^{2}\left(\frac{r_{n}^{+}(H)}{n}\right)^{1-2\Diamond(H)}
	\end{align*}
	uniformly on $\Theta$. 
    Then the conclusion follows from the assumption of $r^+_n(H)$ using that $(\nu_{n}\Delta_{n}^{-H})^{2}=(r_{n}^{1}(H)/n)^{\Diamond(H)}$ so that
	\begin{align*}
		\frac{n}{r_{n}^{1}(H)}\cdot(\nu_{n}\Delta_{n}^{-H})^4\left|\log\left(\frac{n}{r_{n}^{1}(H)}\right)\right|^{2}\left(\frac{r_{n}^{+}(H)}{n}\right)^{1-2\Diamond(H)}
		=\left(\frac{r_{n}^{1}(H)}{r_{n}^{+}(H)}\right)^{2\Diamond(H)-1}\left|\log\left(\frac{n}{r_{n}^{1}(H)}\right)\right|^{2}.
	\end{align*}
	This completes the proof.	
\end{proof}
\begin{lemma}\label{Lemma_Limit_Key2}
	Let $m\in\{0,1,2\}$ and $\{r_{n}^{+}(H)\}_{n\in\mathbb{N}}$ be a positive sequence satisfying
	\begin{equation*}
		\mbox{$r_{n}^{1}(H)=o(r_{n}^{+}(H))$ and $r_{n}^{+}(H)=o(n)$ as $n\to\infty$}
	\end{equation*}
	uniformly on $\Theta_{H}$. 
    Under the same assumptions in Proposition~\ref{Prop:Limit_Key}, we obtain
	\begin{align*}
		\lim_{n\to\infty}
		\frac{n}{r_{n}^{1}(H)}\int_{0}^{\frac{r_{n}^{+}(H)}{n}}\frac{f_{\xi}^{n,\downarrow}(\lambda)^{2}\left(-d_{n}(\theta)+\partial_{H}\log{f_{H}^{\downarrow}(\lambda)}\right)^{m}}{h_{\theta}^{n,\downarrow}(\lambda)^{2}}\,\mathrm{d}\lambda
		=c(\theta)\int_{0}^{\infty}\frac{\left(-2\log\mu\right)^{m}}{\left(1+\mu^{\Diamond(H)}\right)^{2}}\,\mathrm{d}\mu
	\end{align*}
    uniformly on $\Theta$.
\end{lemma}
\begin{proof}[Proof of Lemma~\ref{Lemma_Limit_Key2}]
    Recall that $c_{n}(\theta)= c(\theta)^{-1}(\nu_{n}\Delta_{n}^{-H})^{\frac{2}{\Diamond(H)}}= c(\theta)^{-1}n/r_{n}^{1}(H)$. Using the change of variable $\mu=c_{n}(\theta)\lambda$, we obtain
	\begin{align*}
		\frac{n}{r_{n}^{1}(H)}&\int_{0}^{r_{n}^{+}(H) n^{-1}}\frac{f_{\xi}^{n,\downarrow}(\lambda)^{2}\left(-d_{n}(\theta)+\partial_{H}\log{f_{H}^{\downarrow}(\lambda)}\right)^{m}}{h_{\theta}^{n,\downarrow}(\lambda)^{2}}\,\mathrm{d}\lambda\\
        &
        =\frac{n}{r_{n}^{1}(H)}\int_{0}^{r_{n}^{+}(H) n^{-1}}\frac{\left(-d_{n}(\theta)+\partial_{H}\log{c_{H}}-2\log\lambda\right)^{m}}{\left(1+(\sigma^{2}\tau^{-2}c_{H})^{-1}(\nu_{n}\Delta_{n}^{-H})^{2}\lambda^{\Diamond(H)}\right)^{2}}\,\mathrm{d}\lambda\\
		&=c(\theta)\int_{0}^{r_{n}^{+}(H) n^{-1}c_{n}(\theta)}\frac{\left(-d_{n}(\theta)+\partial_{H}\log{c_{H}}-2\log{c_{n}(\theta)^{-1}}-2\log\mu\right)^{m}}{\left(1+\mu^{\Diamond(H)}\right)^{2}}\,\mathrm{d}\mu\\
		&=c(\theta)\int_{0}^{r_{n}^{+}(H) n^{-1}c_{n}(\theta)}\frac{\left(-2\log\mu\right)^{m}}{\left(1+\mu^{\Diamond(H)}\right)^{2}}\,\mathrm{d}\mu
		\stackrel{n\to\infty}{\to}
		c(\theta)\int_{0}^{\infty}\frac{\left(-2\log\mu\right)^{m}}{\left(1+\mu^{\Diamond(H)}\right)^{2}}\,\mathrm{d}\mu 
	\end{align*}
 Moreover, we can see from the definition of $d_{n}(\theta)$, see \eqref{Def:cn:dn} that $-d_{n}(\theta)+\partial_{H}\log{c_{H}}-2\log{c_{n}(\theta)^{-1}}=0$ and therefore
 \begin{equation*}
\frac{n}{r_{n}^{1}(H)}\int_{0}^{r_{n}^{+}(H) n^{-1}}\frac{f_{\xi}^{n,\downarrow}(\lambda)^{2}\left(-d_{n}(\theta)+\partial_{H}\log{f_{H}^{\downarrow}(\lambda)}\right)^{m}}{h_{\theta}^{n,\downarrow}(\lambda)^{2}}\,\mathrm{d}\lambda
=
c(\theta)\int_{0}^{r_{n}^{+}(H) n^{-1}c_{n}(\theta)}\frac{\left(-2\log\mu\right)^{m}}{\left(1+\mu^{\Diamond(H)}\right)^{2}}\,\mathrm{d}\mu.
\end{equation*}
Then, using the convergence $\inf_{\theta=(H,\sigma,\tau)\in\Theta}r_{n}^{+}(H) n^{-1}c_{n}(\theta)\to\infty$ as $n\to\infty$, we get
\begin{equation*}
    \frac{n}{r_{n}^{1}(H)}\int_{0}^{r_{n}^{+}(H) n^{-1}}\frac{f_{\xi}^{n,\downarrow}(\lambda)^{2}\left(-d_{n}(\theta)+\partial_{H}\log{f_{H}^{\downarrow}(\lambda)}\right)^{m}}{h_{\theta}^{n,\downarrow}(\lambda)^{2}}\,\mathrm{d}\lambda
    \to
    \stackrel{n\to\infty}{\to}
		c(\theta)\int_{0}^{\infty}\frac{\left(-2\log\mu\right)^{m}}{\left(1+\mu^{\Diamond(H)}\right)^{2}}\,\mathrm{d}\mu
\end{equation*}
uniformly on $\Theta$, which concludes the proof.
\end{proof}

\begin{lemma}
\label{PropB1_20200812}
	Let $m\in\{0,1,2\}$ and $\{r_{n}^{+}(H)\}_{n\in\mathbb{N}}$ be a positive sequence satisfying 
	\begin{equation}\label{Assumption_rn-p_Key3}
		r_{n}^{+}(H)=o(n) \;\; \mbox{ and } \;\; \left(\frac{r_{n}^{+}(H)}{n}\right)\left|\log\left(\frac{n}{r_{n}^{1}(H)}\right)\right|\left|\log\left(\frac{n}{r_{n}^{+}(H)}\right)\right|=o(1)\;\; \text{ as } n\to\infty.
	\end{equation}
	uniformly on $\Theta_{H}$. Under the same assumptions in Proposition~\ref{Prop:Limit_Key}, we obtain
	\begin{align*}
		&\frac{n}{r_{n}^{1}(H)}\int_{0}^{\frac{r_{n}^{+}(H)}{n}}\frac{f_{\xi}^{n}(\lambda)^{2}\left(-d_{n}(\theta)+\partial_{H}\log{f_{H}(\lambda)}\right)^{m}}{h_{\theta}^{n}(\lambda)^{2}}\,\mathrm{d}\lambda\\
		&=\frac{n}{r_{n}^{1}(H)}\int_{0}^{\frac{r_{n}^{+}(H)}{n}}\frac{f_{\xi}^{n,\downarrow}(\lambda)^{2}\left(-d_{n}(\theta)+\partial_{H}\log{f_{H}^{\downarrow}(\lambda)}\right)^{m}}{h_{\theta}^{n,\downarrow}(\lambda)^{2}}\,\mathrm{d}\lambda +o(1)\ \ \mbox{as $n\to\infty$}
	\end{align*}
    uniformly on $\Theta$.
\end{lemma}

\begin{proof}[Proof of Lemma~\ref{PropB1_20200812}]
Recall that $c(\theta)=(\sigma^{2}\tau^{-2}c_{H})^{\frac{1}{\Diamond(H)}}$, $c_{n}(\theta)= c(\theta)^{-1}(\nu_{n}\Delta_{n}^{-H})^{\frac{2}{\Diamond(H)}}= c(\theta)^{-1}n/r_{n}^{1}(H)$ and $d_{n}(\theta)=\partial_{H}\log{c_{H}}+2\log{c_{n}(\theta)}$. 
For $m\in\{0,1,2\}$, we also set
	\begin{equation*}
		L_{n}^{(m)}(\theta):=\frac{n}{r_{n}^{1}(H)}\int_{0}^{r_{n}^{+}(H) n^{-1}}\frac{f_{\xi}^{n,\downarrow}(\lambda)^{2}\left|-d_{n}(\theta)+\partial_{H}\log{f_{H}^{\downarrow}(\lambda)}\right|^{m}}{h_{\theta}^{n,\downarrow}(\lambda)^{2}}\,\mathrm{d}\lambda
	\end{equation*}
and for $m\in\{1,2\}$, we write
	\begin{align*}
 \begin{cases}
		q_{f,n}^{(m)}(\theta)
		:=\sup_{\lambda\in(0,r_{n}^{+}(H) n^{-1}]}\left|\left(\frac{f_{H}(\lambda)}{f_{H}^{\downarrow}(\lambda)}\right)^{m}-1\right|,\\ 
		q_{g,n}^{(m)}(\theta)
		:=\sup_{\lambda\in(0,r_{n}^{+}(H) n^{-1}]}\left|\left(\frac{g(\lambda)}{g^{\downarrow}(\lambda)}\right)^{m}-1\right|,\\
		q_{h,n}^{(m)}(\theta)
		:=\sup_{\lambda\in(0,r_{n}^{+}(H) n^{-1}]}\left|\left(\frac{h_{\theta}^{n}(\lambda)}{h_{\theta}^{n,\downarrow}(\lambda)}\right)^{m}-1\right|
 \end{cases}
	\end{align*}
 and 
 \begin{equation*}
    R_{n}^{(m)}(\theta)
		:=\sup_{\lambda\in(0,r_{n}^{+}(H) n^{-1}]}\left|\left(\partial_{H}\log{f_{H}(\lambda)}\right)^{m}-\left(\partial_{H}\log{f_{H}^{\downarrow}(\lambda)}\right)^{m}\right|.
 \end{equation*}
 
    Note that
	\begin{align*}
		&\frac{n}{r_{n}^{1}(H)}\int_{0}^{r_{n}^{+}(H) n^{-1}}\left[
		\frac{f_{\xi}^{n}(\lambda)^{2}\left(-d_{n}(\theta)+\partial_{H}\log{f_{H}(\lambda)}\right)^{m}}{h_{\theta}^{n}(\lambda)^{2}}
		-\frac{f_{\xi}^{n,\downarrow}(\lambda)^{2}\left(-d_{n}(\theta)+\partial_{H}\log{f_{H}^{\downarrow}(\lambda)}\right)^{m}}{h_{\theta}^{n,\downarrow}(\lambda)^{2}}
		\right]\,\mathrm{d}\lambda\\
		&=\frac{n}{r_{n}^{1}(H)}\int_{0}^{r_{n}^{+}(H) n^{-1}}\frac{f_{\xi}^{n}(\lambda)^{2}}{h_{\theta}^{n}(\lambda)^{2}}\left[\left(-d_{n}(\theta)+\partial_{H}\log{f_{H}(\lambda)}\right)^{m} - \left(-d_{n}(\theta)+\partial_{H}\log{f_{H}^{\downarrow}(\lambda)}\right)^{m}\right]\,\mathrm{d}\lambda\\
		&\quad+\frac{n}{r_{n}^{1}(H)}\int_{0}^{r_{n}^{+}(H) n^{-1}}\frac{f_{\xi}^{n,\downarrow}(\lambda)^{2}}{h_{\theta}^{n}(\lambda)^{2}}\left(\frac{f_{\xi}^{n}(\lambda)^{2}}{f_{\xi}^{n,\downarrow}(\lambda)^{2}}-1\right)\left(-d_{n}(\theta)+\partial_{H}\log{f_{H}^{\downarrow}(\lambda)}\right)^{m}\,\mathrm{d}\lambda\\
		&\quad+\frac{n}{r_{n}^{1}(H)}\int_{0}^{r_{n}^{+}(H) n^{-1}}\frac{f_{\xi}^{n,\downarrow}(\lambda)^{2}}{h_{\theta}^{n}(\lambda)^{2}}\left(1-\frac{h_{\theta}^{n}(\lambda)^{2}}{h_{\theta}^{n,\downarrow}(\lambda)^{2}}\right)\left(-d_{n}(\theta)+\partial_{H}\log{f_{H}^{\downarrow}(\lambda)}\right)^{m}\,\mathrm{d}\lambda\\
		&=:J_{n,1}^{(m)}(\theta)+J_{n,2}^{(m)}(\theta)+J_{n,3}^{(m)}(\theta).
	\end{align*}
    From the definitions of $q_{f,n}^{(m)}(\theta)$, $q_{f,n}^{(m)}(\theta)$, $L_{n}^{(m)}(\theta)$ and $R_{n}^{(m)}(\theta)$, we can show
	\begin{equation}\label{inequality_J2_J3}
		J_{n,2}^{(m)}(\theta)\lesssim q_{f,n}^{(2)}(\theta)L_{n}^{(m)}(\theta),\ \ 
		J_{n,3}^{(m)}(\theta)\lesssim q_{h,n}^{(2)}(\theta)L_{n}^{(m)}(\theta),\ \ \forall m\in\{0,1,2\},
	\end{equation}
	and
	\begin{equation*}
		J_{n,1}^{(m)}(\theta)\lesssim \left\{
		\begin{array}{cc}
		  	0&\mbox{if $m=0$},\\
		  	R_{n}^{(1)}(\theta)L_{n}^{(2)}(\theta)&\mbox{if $m=1$},\\
		  	\left(2|d_{n}(\theta)|R_{n}^{(1)}(\theta)+R_{n}^{(2)}(\theta)\right)L_{n}^{(2)}(\theta)&\mbox{if $m=2$}
		\end{array}
		\right.
	\end{equation*}
    uniformly on $\Theta$. 
    Proceeding as the proof of Lemma~\ref{Lemma_Limit_Key2}, we also have
	\begin{equation*}
		L_{n}^{(m)}(\theta)\lesssim 1,\ \ \forall m\in\{0,1,2\}.
	\end{equation*}
	uniformly on $\Theta$. 
	Using the definitions of $f_H^{\downarrow}$ and $g_H^{\downarrow}$ and Taylor's theorem, we have 
	\begin{align}
\label{ineq:q-LT2}
	\begin{split}
		R_{n}^{(1)}(\theta)+R_{n}^{(2)}(\theta)=o(1)\ \ \mbox{as $n\to\infty$},
        \\
        q_{f,n}^{(2)}(\theta)+q_{h,n}^{(2)}(\theta)=o(1)\ \ \mbox{as $n\to\infty$}
    \end{split}
	\end{align}
	uniformly on $\Theta$.
    Moreover, Taylor's theorem also yields
    \begin{align*}
        q_{f,n}^{(1)}(\theta) +q_{g,n}^{(1)}(\theta)
        +\sup_{\lambda\in(0,r_{n}^{+}(H) n^{-1}]}\left|\frac{f_{H}(\lambda)}{f_{H}^{\downarrow}(\lambda)}\right|
        +\sup_{\lambda\in(0,r_{n}^{+}(H) n^{-1}]}
        \left|\frac{g(\lambda)}{g^{\downarrow}(\lambda)}\right| = o(1)\ \ \mbox{as $n\to\infty$}
    \end{align*}
    uniformly on $\Theta$. On the other hand, we can show $q_{h,n}^{(1)}(\theta)\leq q_{f,n}^{(1)}(\theta)+q_{g,n}^{(1)}(\theta)$ and
    \begin{align*}
        &q_{f,n}^{(2)}(\theta)
        \leq q_{f,n}^{(1)}(\theta)
        \sup_{\lambda\in(0,r_{n}^{+}(H) n^{-1}]}
        \left|\frac{f_{\xi}^{n}(\lambda)+f_{\xi}^{n,\downarrow}(\lambda)}{f_{\xi}^{n,\downarrow}(\lambda)}\right|
        \leq
        q_{f,n}^{(1)}(\theta)
        \left(
        1+\sup_{\lambda\in(0,r_{n}^{+}(H) n^{-1}]}
        \left|\frac{f_{H}(\lambda)}{f_{H}^{\downarrow}(\lambda)}\right|
        \right),\\ 
        &q_{h,n}^{(2)}(\theta)
        \leq q_{h,n}^{(1)}(\theta)
        \sup_{\lambda\in(0,r_{n}^{+}(H) n^{-1}]}
        \left|\frac{h_{\theta}^{n}(\lambda)+h_{\theta}^{n,\downarrow}(\lambda)}{h_{\theta}^{n,\downarrow}(\lambda)}\right|
        \leq
        q_{h,n}^{(1)}(\theta)
        \left(
        1+\sup_{\lambda\in(0,r_{n}^{+}(H) n^{-1}]}
        \left|\frac{f_{H}(\lambda)}{f_{H}^{\downarrow}(\lambda)}\right|
        +\sup_{\lambda\in(0,r_{n}^{+}(H) n^{-1}]}
        \left|\frac{g(\lambda)}{g^{\downarrow}(\lambda)}\right|
        \right).
    \end{align*}
    Thus we conclude \eqref{ineq:q-LT2} from the above calculations. 
    Then, using the inequalities~\eqref{inequality_J2_J3}-\eqref{ineq:q-LT2}, we also conclude
	\begin{align*}
		J_{n,1}^{(m)}(\theta)+J_{n,2}^{(m)}(\theta)+J_{n,3}^{(m)}(\theta)=o(1)\ \ \mbox{as $n\to\infty$}
	\end{align*}
	uniformly on $\Theta$. This completes the proof.
\end{proof}
\begin{proof}[Proof of Proposition~\ref{Prop:Limit_Key}]
Consider $r_{n}^{+}(H)$ defined by
	\begin{align*}
		r_{n}^{+}(H)=r_{n}^{1}(H){(\nu_{n}\Delta_{n}^{-H})}^{\epsilon}\ \ \mbox{for some $\epsilon\in(0,2\Diamond(H)^{-1})$.}
	\end{align*}
This sequence satisfies the conditions of Lemmas~\ref{PropC1_Ext1_20200812}-\ref{PropB1_20200812}.
	and the conclusion immediately follows from and application of these Lemmas.
\end{proof}
\subsection{Proof of Proposition~\ref{Prop:FisherInfo-Cross-Terms}}
Since $r_{n}^{2}(H)=n$ and $\inf_{H\in\Theta_{H}}\Diamond(H)>1$, we have
	\begin{align*}
		&\frac{n}{\sqrt{r_{n}^{1}(H)r_{n}^{2}(H)}}
		\int_{0}^{\pi}\frac{g_{\tau}^{n}(\lambda)f_{\xi}^{n}(\lambda)\left|-d_{n}(\theta)+\partial_{H}\log{f_{H}(\lambda)}\right|^{m-1}}{h_{\theta}^{n}(\lambda)^{2}}\,\mathrm{d}\lambda\\
		&\lesssim
		\sqrt{\frac{r_{n}^{1}(H)}{r_{n}^{2}(H)}}\cdot\frac{n}{r_{n}^{1}(H)}
		\int_{0}^{\pi}\frac{(\nu_{n}\Delta_{n}^{-H})^{2}|\lambda|^{\Diamond(H)}\left||d_{n}(\theta)|+|\log\lambda|\right|^{m-1}}{\left(1+(\nu_{n}\Delta_{n}^{-H})^{2}|\lambda|^{\Diamond(H)}\right)^{2}}\,\mathrm{d}\lambda\\
		&\lesssim
		\sqrt{\frac{r_{n}^{1}(H)}{r_{n}^{2}(H)}}
		\left|\log{\frac{r_{n}^{1}(H)}{n}}\right|^{m-1}
		\int_{0}^{\infty}\frac{\mu^{\Diamond(H)}\left|1+|\log\lambda|\right|^{m-1}}{\left(1+\mu^{\Diamond(H)}\right)^{2}}\,\mathrm{d}\mu\\
		&\leq
		\sqrt{\frac{r_{n}^{1}(H)}{r_{n}^{2}(H)}}
		\left|\log{\frac{r_{n}^{1}(H)}{n}}\right|^{m-1}
		\left(
		\int_{0}^{1}\mu^{\Diamond(H)}\left|1+|\log\mu|\right|^{m-1}\,\mathrm{d}\mu
		+\int_{1}^{\infty}\mu^{-\Diamond(H)}\left|1+|\log\mu|\right|^{m-1}\,\mathrm{d}\mu\right)
	\end{align*}
	uniformly on $\Theta$. 
    Moreover, since we have $\inf_{H\in\Theta_{H}}\Diamond(H)>0$ and $\lim_{n\to\infty}\inf_{H\in\Theta_{H}}\nu_{n}\Delta_{n}^{-H}=\infty$ from Assumption~\ref{Assump:RateMat}, we obtain
 \begin{align*}
     \sqrt{\frac{r_{n}^{1}(H)}{r_{n}^{2}(H)}}
		\left|\log{\frac{r_{n}^{1}(H)}{n}}\right|^{m-1}
     =(\nu_{n}\Delta_{n}^{-H})^{-\frac{1}{\Diamond(H)}}
		\left|\log{(\nu_{n}\Delta_{n}^{-H})^{-\frac{2}{\Diamond(H)}}}\right|^{m-1}
     \to 0 
 \end{align*}
 as $n\to\infty$ uniformly on $\Theta$. This completes the proof.

\section{Explicit Expressions of $\left[\mathcal{I}(\theta)^{-1}\right]_{jj}$}
\label{section:explicit_inverse}

In this section, we derive explicit expressions of $\left[\mathcal{I}(\theta)^{-1}\right]_{jj}$ for $j\in\{1,2,3\}$ when $\nu(H)=\infty$ with $\Diamond(H)>1/2$. 
For conciseness, we write
\begin{align*}
	\overline{\mathcal{F}_{1}}(\theta)=
	\begin{pmatrix}
		\mathcal{F}_{11}(\theta)&\mathcal{F}_{12}(\theta)\\
		\mathcal{F}_{21}(\theta)&\mathcal{F}_{22}(\theta)
	\end{pmatrix},\ \ 
	\overline{s}_{2j}(\theta)=
	\overline{\varphi}_{2j}(\theta)+d(\theta)\overline{\varphi}_{1j}(\theta),\ \ 
	j\in\{1,2\}.
\end{align*} 
By definition, we have
\begin{align*}
	\mathcal{I}(\theta)=
	\begin{pmatrix}
		\overline{\varphi}(\theta)^{\top}\overline{\mathcal{F}_{1}}(\theta)\overline{\varphi}(\theta)
		&0_{2\times 1}\\
		0_{1\times 2}&(2\tau^{-1})^{2}\mathcal{F}_{33}(\theta)
	\end{pmatrix}
\end{align*}
and
\begin{align*}
	\left[\overline{\varphi}(\theta)^{\top}\overline{\mathcal{F}_{1}}(\theta)\overline{\varphi}(\theta)\right]_{jj}
	=\overline{\varphi}_{1j}(\theta)^{2}\mathcal{F}_{11}(\theta)
	+2\overline{\varphi}_{1j}(\theta)\overline{s}_{2j}(\theta)\mathcal{F}_{12}(\theta)
	+\overline{s}_{2j}(\theta)^{2}\mathcal{F}_{22}(\theta)
\end{align*}
for $j\in\{1,2\}$. Moreover
\begin{align*}
	\mathrm{det}\left[\mathcal{I}(\theta)\right]
	&=(2\tau^{-1})^{2}\mathcal{F}_{33}(\theta)
	\mathrm{det}\left[\overline{\varphi}(\theta)^{\top}\overline{\mathcal{F}_{1}}(\theta)\overline{\varphi}(\theta)\right]\\
	&=(2\tau^{-1})^{2}\mathcal{F}_{33}(\theta)
	\mathrm{det}[\overline{\varphi}(\theta)]^{2}
	\mathrm{det}\left[\overline{\mathcal{F}_{1}}(\theta)\right],\\
\end{align*}
and
\begin{equation*}
    \mathrm{det}[\overline{\varphi}(\theta)]
	=\overline{\varphi}_{11}(\theta)\overline{\varphi}_{22}(\theta)-\overline{\varphi}_{12}(\theta)\overline{\varphi}_{21}(\theta),
\end{equation*}
so we obtain
\begin{align*}	
\left[\mathcal{I}(\theta)^{-1}\right]_{11}
	=&\mathrm{det}\left[\mathcal{I}(\theta)\right]^{-1}
	\mathrm{det}\left[\mathrm{diag}\left(
	\left[\overline{\varphi}(\theta)^{\top}\overline{\mathcal{F}_{1}}(\theta)\overline{\varphi}(\theta)\right]_{22},
	(2\tau^{-1})^{-2}\mathcal{F}_{33}(\theta)
	\right)
	\right] \\
	=&
	\left(\overline{\varphi}_{11}(\theta)\overline{\varphi}_{22}(\theta)-\overline{\varphi}_{12}(\theta)\overline{\varphi}_{21}(\theta)\right)^{-2}
	\mathrm{det}\left[\overline{\mathcal{F}_{1}}(\theta)\right]^{-1} \\
	&\quad\times
	\left(\overline{\varphi}_{12}(\theta)^{2}\mathcal{F}_{11}(\theta)
	+2\overline{\varphi}_{12}(\theta)\overline{s}_{22}(\theta)\mathcal{F}_{12}(\theta)
	+\overline{s}_{22}(\theta)^{2}\mathcal{F}_{22}(\theta)\right).
\end{align*}
Similarly, we have
\begin{align*}
	\left[\mathcal{I}(\theta)^{-1}\right]_{22}
	=&\mathrm{det}\left[\mathcal{I}(\theta)\right]^{-1}
	\mathrm{det}\left[\mathrm{diag}\left(
	\left[\overline{\varphi}(\theta)^{\top}\overline{\mathcal{F}_{1}}(\theta)\overline{\varphi}(\theta)\right]_{11},
	(2\tau^{-1})^{-2}\mathcal{F}_{33}(\theta)
	\right)
	\right] \\
	=&
	\left(\overline{\varphi}_{11}(\theta)\overline{\varphi}_{22}(\theta)-\overline{\varphi}_{12}(\theta)\overline{\varphi}_{21}(\theta)\right)^{-2}
	\mathrm{det}\left[\overline{\mathcal{F}_{1}}(\theta)\right]^{-1} \\
	&\quad\times
	\left(\overline{\varphi}_{11}(\theta)^{2}\mathcal{F}_{11}(\theta)
	+2\overline{\varphi}_{11}(\theta)\overline{s}_{21}(\theta)\mathcal{F}_{12}(\theta)
	+\overline{s}_{21}(\theta)^{2}\mathcal{F}_{22}(\theta)\right)
\end{align*}
and
\begin{equation*}
	\left[\mathcal{I}(\theta)^{-1}\right]_{33}
	=\mathrm{det}\left[\mathcal{I}(\theta)\right]^{-1}
	\mathrm{det}\left[\overline{\varphi}(\theta)^{\top}\overline{\mathcal{F}_{1}}(\theta)\overline{\varphi}(\theta)\right]
	=\frac{\tau^{2}}{2}.
\end{equation*}

\section{Proof of the results of Section~\ref{Sec:energy}}

\subsection{Proof of Lemma~\ref{lem:wavelet:properties}}
\label{Sec:lem:wavelet:properties:proof}

First, note that by self-similarity and the fact that $W^H$ has stationary increments, we have
\begin{equation*}
\mathbb{E}_{\theta}^{n}[(d_{j,p,n}^{(w)})^{2}]=
\frac{1}{n^{1+2H}p^{1-2H}} 
\sum_{\ell_1=0}^{p-1}
\sum_{\ell_2=0}^{p-1}
w(\tfrac{\ell_1}{p})
w(\tfrac{\ell_2}{p})
\mathbb{E}_{\theta}^{n}\big[
\left(
W^H_{(\ell_1-\ell_2)/p}
-
2W^H_{1+(\ell_1-\ell_2)/p}
+
W^H_{2+(\ell_1-\ell_2)/p}
\right)
(
W^H_{2}
-
2W^H_{1}
)
\big].
\end{equation*}
It follows that $\kappa_{p}^{(w)}(H)$ is well defined for all values of $H$ and $p$. Moreover, a direct computation of the expectation above yields
\begin{equation*}
\kappa_{p}^{(w)}(H) = 
p^{-2}
\sum_{\ell_1=0}^{p-1}
\sum_{\ell_2=0}^{p-1}
w(\tfrac{\ell_1}{p})
w(\tfrac{\ell_2}{p})
 \phi_{H}((\ell_1 - \ell_2) p^{-1}),
\end{equation*}
where $\phi_{H}(x) = \tfrac{1}{2}\sum_{k=0}^{4} (-1)^{k+1} \binom{4}{k} |x+k-2|^{2H}$.
\\

We now focus on the decay of the autocorrelation function. By symmetry, it suffices to show this decay only for $j_{1} - j_{2} \geq 3$. Since $W^H$ is self-similar and has stationary increments, we can compute explicitly this autocorrelation function and we have
\begin{align}
\label{eq:explicit_correl_d}
\mathbb{E}_{\theta}^{n}\big[d_{j_{1},p,n}^{(w)}d_{j_{2},p,n}^{(w)}\big]
=
(p/n)^{1+2H} \cdot
p^{-2}
\sum_{\ell_1=0}^{p-1}
\sum_{\ell_2=0}^{p-1}
w(\tfrac{\ell_1}{p})
w(\tfrac{\ell_2}{p})
 \phi_{H}( j_{1}-j_{2} + (\ell_1 - \ell_2) p^{-1}),
\end{align}
where $\phi_{H}(x) = \tfrac{1}{2}\sum_{k=0}^{4} (-1)^{k+1} \binom{4}{k} |x+k-2|^{2H}$. For $x > 2$, the absolute values appearing in the expression of $\phi_{H}(x)$ can be removed and we can use Taylor's formula to expand $|x+k-2|^{2H}$ around $|x|^{2H}$. Let $F(x,t) = |x+t|^{2H}$,  we have
\begin{align*}
\phi_{H}(x) &= \frac{1}{2\cdot 4!} \int_{0}^1 (1-t)^3 \left( 
	16 \partial^{4}_t F(x,-2t)
	-4\partial^{4}_t F(x,-t)
	-4\partial^{4}_t F(x,t)
	+16 \partial^{4}_t F(x,2t)
\right)\,\mathrm{d}t
\\
&= \frac{1}{2\cdot 4!} \int_{-1}^1 (1-|t|)^3 \left( 
	-4\partial^{4}_t F(x,t)
	+16 \partial^{4}_t F(x,2t)
\right)\,\mathrm{d}t.
\end{align*}
Then we infer $C_- |x-2|^{2H-4} \leq |\phi_{H}(x)| \leq C_+ |x-2|^{2H-4}$ for some constant $C_+$ and $C_-$ independent of $H \in [H_{-}, H_{+}]$. Summing over $\ell_1$ and $\ell_2$ in \eqref{eq:explicit_correl_d} yields the result.

\subsection{Proof of Lemma~\ref{lem:wavelet:kappa_p}}
\label{Sec:lem:wavelet:kappa_p:proof}

It is cleat from its definition that $(H, x) \mapsto \phi_H(x)$ is differentiable in $H$ and therefore $\kappa_p^{(w)}$ is differentiable. Using usual theorems for derivation under the integral, we also know that $\kappa_\infty^{(w)}$ is differentiable. 
\\

Moreover, since $w$ is Lipschitz continuous (since it is differentiable with bounded derivative on $[0,1]$) and $|\phi_{H}(x) - \phi_{H}(y)| \leq C |x-y|^{2H\wedge 1}$  for any $-1 \leq x, y \leq 1$, we deduce that
\begin{align*}
|\kappa_{p}^{(w)}(H) &- \kappa_{\infty}^{(w)}(H)|
\leq
\Big| p^{-2}
\sum_{\ell_1=0}^{p-1}
\sum_{\ell_2=0}^{p-1}
w(\tfrac{\ell_1}{p})
w(\tfrac{\ell_2}{p})
 \phi_{H}((\ell_1 - \ell_2) p^{-1}) - \int_{0}^1 \int_{0}^1 w(x)
w(y) \phi_{H}(x-y) \,\mathrm{d}x\mathrm{d}y
\Big |
\\
&\leq
\sum_{\ell_1=0}^{p-1}
\sum_{\ell_2=0}^{p-1}
\int_{\tfrac{\ell_1}{p}}^{\tfrac{\ell_1+1}{p}}
\int_{\tfrac{\ell_2}{p}}^{\tfrac{\ell_2+1}{p}}
\big| 
w(\tfrac{\ell_1}{p})
w(\tfrac{\ell_2}{p})
 \phi_{H}((\ell_1 - \ell_2) p^{-1})
-  
w(x)
w(y) 
\phi_{H}(x-y) 
\big |
\,\mathrm{d}x\mathrm{d}y
\\
&\leq
C
\sum_{\ell_1=0}^{p-1}
\sum_{\ell_2=0}^{p-1}
\int_{\tfrac{\ell_1}{p}}^{\tfrac{\ell_1+1}{p}}
\int_{\tfrac{\ell_2}{p}}^{\tfrac{\ell_2+1}{p}}
\big| 
\tfrac{\ell_1}{p}
-
x
\big |
+
\big| 
\tfrac{\ell_2}{p}
-
y
\big |
+
\big| 
\tfrac{\ell_1 - \ell_2}{p}
-
(x-y)
\big |^{2H\wedge 1}
\,\mathrm{d}x\mathrm{d}y
\\
&\leq
C
\sum_{\ell_1=0}^{p-1}
\sum_{\ell_2=0}^{p-1}
\int_{\tfrac{\ell_1}{p}}^{\tfrac{\ell_1+1}{p}}
\int_{\tfrac{\ell_2}{p}}^{\tfrac{\ell_2+1}{p}}
\tfrac{1}{p}
+
\tfrac{1}{p^{2H\wedge 1}}
\,\mathrm{d}x\mathrm{d}y
\\
&\leq
\tfrac{C}{p^{2H\wedge 1}}.
\end{align*}
We show similarly the bound $|{\partial_{H}\kappa_{p}^{(w)}}(H) - {\partial_{H}\kappa_{\infty}^{(w)}}(H) | \leq c^{(w)} p^{-2H\wedge 1}$. \\

It remains to prove the uniform control on $\kappa_p^{(w)}$ and on $\partial_H \kappa_\infty^{(w)}$. Recall from Lemma~\ref{lem:wavelet:properties} that $
\kappa_{p}^{(w)}(H)
=
(n/p)^{1+2H} \mathbb{E}_{\theta}^{n}[(d_{j,p,n}^{(w)})^{2}]
$. Moreover, $\mathbb{E}_{\theta}^{n}[(d_{j,p,n}^{(w)})^{2}]>0$ whenever there exists $0 \leq \ell \leq p-1$ such that $w(\ell/p) \neq 0$. This is the case whenever $p \geq 2$ by Assumption~\ref{assumption:alpha}. Using also continuity of $\kappa_{p}^{(w)}$ for any fixed $p$, we know that $\inf_{H} \kappa_{p}^{(w)}(H) > 0
$ for all $p \geq 2$
 and that $\sup_{H}  \kappa_{p}^{(w)}(H) <\infty
$
 for all $p \geq 1$. To get the uniformity in $p$ of \eqref{eq:kappa:uniform_bound}, we use the uniform convergence $\kappa_{p} \to \kappa_\infty$ (by Lemma~\ref{lem:wavelet:kappa_p}), the fact that $\kappa_{\infty}^{(w)}$ and that
\begin{equation*}
\kappa_{\infty}^{(w)}(H) = \mathbb{E}\left[
\left(\int_0^1 w(u) (W^H_u - 2 W^H_{u+1} + W^H_{u+2})\,\mathrm{d}u\right)^{2}
\right] > 0.
\end{equation*}

\subsection{Proof of Lemma~\ref{lem:noise:structure}}
\label{Sec:lem:noise:structure:proof}
Under Assumption~\ref{Assumption_Y}, we have $Y_{j} = \sum_{l=0}^{K}{K \choose l} (-1)^{l} \varepsilon_{j-l}$ in distribution for some standard Gaussian variables $\{\varepsilon_{j}\}_{j\in\mathbb{Z}}$ 
\begin{equation}
\label{eq:expression:e}
	e_{j,p,n}^{(w)}=
	\frac{\tau^{2}\nu_{n}^{2}}{\sqrt{np}}
	\sum_{l=-K}^{p-1}\alpha_{l,p}
	\left(\varepsilon_{jp+l}-2\varepsilon_{(j+1)p+l}+\varepsilon_{(j+2)p+l}\right)
\end{equation}
in distribution, where 
\begin{equation}
\label{eq:Def:beta}
	\alpha_{l,p} :=
	\sum_{i=0\vee(-l)}^{K\wedge(p-1-l)}{K\choose i} (-1)^{i} w(\tfrac{l+i}{p}).
\end{equation}
Since the $\varepsilon_{l}$ are independent standard Gaussian variables, we can compute the variance of $e_{j,p,n}^{(w)}$ from \eqref{eq:expression:e} and in particular, $\gamma_{p}^{(w)}= n\mathbb{E}_{\theta}^{n}[(e_{j,p,n}^{(w)})^{2}]$ is well defined. The precise computation of $\gamma_{p}^{(w)}$ require delicate handling because of the crossovers between $\varepsilon_{(j+1)p+l,n}$ for $l<0$ and $\varepsilon_{jp+l,n}$ for $l$ close to $p-1$. Indeed, we obtain
\begin{align*}
\gamma_{p}^{(w)}&=
\frac{1}{p}
\left(
\sum_{l=-K}^{p-K-1}
\alpha_{l,p}^{2}
+
\sum_{l=p-K}^{p-1}
(\alpha_{l,p}-2\alpha_{l-p,p})^{2}
+
\sum_{l=p}^{2p-K-1}
4\alpha_{l-p,p}^{2}
\right.\\&
\left.
\;\;\;\;\;\;\;\;\;\;\;\;\;\;\;\;+
\sum_{l=p-K}^{2p-1}
(2\alpha_{l-p,p}-\alpha_{l-2p,p})^{2}
+
\sum_{l=2p}^{3p-1}
\alpha_{l-2p,p}^{2}
\right)
\\
&=
\frac{1}{p}
\left(
6\sum_{l=-K}^{p-1}
\alpha_{l,p}^{2}
-
8
\sum_{l=p-K}^{p-1}
\alpha_{l,p}\alpha_{l-p,p}
\right).
\end{align*}
From \eqref{eq:Def:beta}, the coefficients $\alpha_{l,p}$ are bounded uniformly so that the sequence $(\gamma_p)_p$ is bounded.\\

Suppose now that Assumption~\ref{assumption:alpha} holds. Note that since the $K-1$ first derivatives of $w$ vanish at $0$ and at $1$, we have $|\alpha_{l,p}| \leq Cp^{-K}$ whenever $-K \leq l < 0$ of $p-K \leq l < p$. For $0 \leq l < p-K$, we have by Taylor-Lagrange's formula
\begin{align*}
	\alpha_{l,p}
	=\sum_{i=0}^{K}{K\choose i}(-1)^{i}
	\left(
	\sum_{k=0}^{K-1} \tfrac{1}{k!} (i/p)^{k} \partial_{x}^{k}w(l/p) 
	+\tfrac{1}{K!}(i/p)^{K} \partial_{x}^{K}w\left( \frac{l+i\theta_{i,l}}{p}\right) 
	\right)
\end{align*}
for some $0 \leq \theta_{i,l} \leq 1$. Using also that $\sum_{i=0}^{K} {K\choose i}i^{k}(-1)^{i}= 0$ whenever $k\in\{0,1,\cdots,K-1\}$, we get 
\begin{equation}\label{eq:simplified:beta}
	\alpha_{l,p}
	=
	\tfrac{\displaystyle 1}{\displaystyle K!p^{K}}
	\sum_{i=0}^{K}
	{K \choose i} (-1)^{i} i^{K} \partial_{x}^{K}w\left( \frac{l+i\theta_{i,l}}{p}\right).
\end{equation}
Using finally that $\partial_{x}^{K}w(x)$ is continuous, we retrieve the same bound $|\alpha_{l,p}| \leq Cp^{-K}$ for $0 \leq l < p-K$. Combining this with the explicit expression for $\gamma_{p}^{(w)}$ obtained previously yields the upper bound on $\gamma_{p}^{(w)}$ in Lemma~\ref{lem:noise:structure}.\\

We eventually focus on the lower bound. Note from \eqref{eq:simplified:beta} that since $\partial_{x}^{K}w$ is continuously differentiable on $[0,1]$, we have
\begin{align*}
	\alpha_{l,p}
	=&\frac{ 1}{K!p^{K}}
	\partial_{x}^{K}w( l p^{-1} )
	\sum_{i=0}^{K}
	{K \choose i} (-1)^{i} i^{K} 
	+O(p^{-K-1})\\ 
	=& 
	\frac{ (-1)^n}{ p^{K}}
	\partial_{x}^{K}w( l p^{-1} )
	+O(p^{-K-1}),
\end{align*}
where the $O$ is uniform in $l$. Note in addition that for $p>2K$, we have
\begin{align*}
\gamma_{p}^{(w)}&\geq
\frac{1}{p}
\sum_{l=K}^{p-K-1}
\alpha_{l,p}^{2}
\\
&\geq 
\frac{1}{p}
\sum_{l=K}^{p-K-1}
\left(
\frac{1}{p^{2K}}
\partial_{x}^{K}w\left( l/p\right)^{2}
+O(p^{-2K-1})
\right)
\\
&\geq 
\frac{1}{p^{2K+1}}
\sum_{l=K}^{p-K-1}
\partial_{x}^{K}w\left(l/p\right)^{2}
+O(p^{-2K-1}).
\end{align*}
Finally, using that $x\mapsto\partial_{x}^{K}w(x)$ is continuously differentiable, we have
\begin{align*}
	\gamma_{p}^{(w)}
	\geq 
	\frac{1}{p^{2K}}
	\norm{\partial_{x}^{K}w}_{L^{2}([0,1])}^{2}
	+O(p^{-2K-1}),
\end{align*}
which yields the lower bound ion $\gamma_{p}^{(w)}$ in Lemma~\ref{lem:noise:structure} when $p$ is large enough.

\subsection{Proof of Proposition~\ref{prop:deviation:Q}}
\label{Sec:prop:deviation:Q:proof}
Recall that
$\widehat{Q}_{J,p,n}^{(w)} = \sum_{j=0}^{J-1} |\widetilde{d}_{j,p,n}^{(w)}|^{2}$. Therefore, using Lemmas~\ref{lem:wavelet:properties} and~\ref{lem:noise:structure} and the independence between $W^H$ and the noise $Y$, we obtain
\begin{align*}
	\mathbb{E}_{\theta}^{n}\left[\widehat{Q}_{J,p,n}^{(w)}\right]
	=J\left(
	\sigma^{2} \kappa_{p}^{(w)}(H) (p/n)^{1+2H} + \tau^{2} \gamma_{p}^{(w)} n^{-1}
	\right).
\end{align*}
Then, using Wick's theorem and the independence between $W^H$ and $Y$ again, we obtain
\begin{align*}
	\Var_{\theta}^{n}\big[\widehat{Q}_{J,p,n}^{(w)}\big]
	&=
	\sum_{j_{1},j_{2}=0}^{J- 1} 
	\Cov_{\theta}^{n}\big[
	\big|\widetilde{d}_{j_{1},p,n}^{(w)}\big|^{2},
	\big|\widetilde{d}_{j_{2},p,n}^{(w)}\big|^{2}
	\big]\\
	&=
	2\sum_{j_{1},j_{2}=0}^{J- 1} 
	\Cov_{\theta}^{n}\big[\widetilde{d}_{j_{1},p,n}^{(w)},\widetilde{d}_{j_{2},p,n}^{(w)}\big]^{2}
	\\
	&=
	2\sum_{j_{1},j_{2}=0}^{J- 1}\big(
	\sigma^{2}\Cov_{\theta}^{n}\big[d_{j_{1},p,n}^{(w)},d_{j_{2},p,n}^{(w)}\big] 
	+\tau^{2}\Cov_{\theta}^{n}\big[e_{j_{1},p,n}^{(w)},e_{j_{2},p,n}^{(w)}\big]
	\big)^{2}
	\\
	&\leq
	4\sum_{j_{1},j_{2}=0}^{J- 1}\big(
	\sigma^{2}\Cov_{\theta}^{n}\big[d_{j_{1},p,n}^{(w)},d_{j_{2},p,n}^{(w)}\big]^{2} 
	+\tau^{2}\Cov_{\theta}^{n}\big[e_{j_{1},p,n}^{(w)},e_{j_{2},p,n}^{(w)}\big]^{2}
	\big).
\end{align*}
We first derive the upper bound. 
Using Lemma~\ref{lem:noise:structure}, we know $\Cov_{\theta}^{n}[e_{j_{1},p,n}^{(w)},e_{j_{2},p,n}^{(w)}]=0$ whenever $|j_{1}-j_{2}|\geq K+4$ and we obtain
\begin{align*}
	\Cov_{\theta}^{n}\left[e_{j_{1},p,n}^{(w)},e_{j_{2},p,n}^{(w)}\right]^{2}
	\leq\Var_{\theta}^{n}\left[e_{0,p,n}^{(w)}\right]^{2}
	\lesssim \frac{1}{n^{2}p^{4K}}
\end{align*}
uniformly on $\Theta$ whenever $|j_{1}-j_{2}|<K+4$.
Together with Lemma~\ref{lem:wavelet:properties}, we obtain
\begin{align*}
\Var_{\theta}^{n}\left[\widehat{Q}_{J,p,n}^{(w)}\right]
&\lesssim
J(p/n)^{2+4H}
\sum_{j=1}^{\infty}
j^{4H-8}
+
J
\frac{1}{n^{2} p^{4K}}
\end{align*}
uniformly on $\Theta$ and we conclude since $\sum_{j=1}^{\infty}
j^{4H-8} \leq \sum_{j=1}^{\infty}
j^{-4} < \infty$.
\\

	We now focus on the lower bound. Using Lemmas~\ref{lem:wavelet:properties} and~\ref{lem:noise:structure}, we obtain
\begin{align*}
\Var_{\theta}^{n}\left[\widehat{Q}_{J,p,n}^{(w)}\right]
&
\geq
\sum_{j=0}^{J-1}\left(\Var_{\theta}^{n}\left[d_{j,p,n}^{(w)}\right]+\Var_{\theta}^{n}\left[e_{j,p,n}^{(w)}\right]\right)^{2}
\\
&=
J\left(\Var_{\theta}^{n}\left[d_{0,p,n}^{(w)}\right]+\Var_{\theta}^{n}\left[e_{0,p,n}^{(w)}\right]\right)^{2}
\\
&\geq
c_{-}^{(w)} J
\left(
\left(\frac{p}{n}\right)^{1+2H}+\frac{1}{np^{{2K}}}
\right)^{2}
\\
&\geq 
c_{-}^{(w)}
J
\left(
\left(\frac{p}{n}\right)^{2+4H}+\frac{1}{n^{2}p^{4K}}\right).
\end{align*}
This completes the proof.

\section{Proof of the results of Section~\ref{Sec:QVestimators}}

\subsection{Proof of Lemma~\ref{lemma:uncorrectedQV:H}}
\label{Sec:lemma:uncorrectedQV:H:proof}

Recall that 
$p_{n}^{(0)} = \big \lfloor n^{2/(2K+3)} \big \rfloor \vee 2$ and that $\widehat{H}^{(0)}_{n}$ is defined in \eqref{eq:Def:guess_est} by
\begin{align*}
	\widehat{H}^{(0)}_{n} = \left( 
	\left( 
	\frac{1}{2}\log_{2}\left[
	\widehat{Q}_{J_{n}^{(0)},2p_{n}^{(0)},n}^{(w)} \, \Big / \, \widehat{Q}_{2J_{n}^{(0)},p_{n}^{(0)},n}^{(w)}
	\right] 
	\right)
	\vee H_{-}
	\right)
	\wedge H_{+}.
\end{align*}
Since $H \in [H_{-}, H_{+}]$ and $t \mapsto 2^{2t}$ is invertible on $(0,1)$ with inverse uniformly Lipschitz on the compact sets of $(0,1)$, it suffices to prove that
\begin{align*}
	\left\{v_{n}^{(0)}(H)
	\left(\frac{\widehat{Q}_{J_{n}^{(0)},2p_{n}^{(0)},n}^{(w)}}{\widehat{Q}_{2J_{n}^{(0)},p_{n}^{(0)},n}^{(w)}} - 2^{2H}\right)\right\}_{n\in\mathbb{N}}
\end{align*}
is bounded in $\mathbb{P}^{n}_{\theta}$-probability, uniformly over $\Theta$. We first prove a lower bound (in $\mathbb{P}^{n}_{\theta}$-probability) for $(p_{n}^{(0)}/n)^{2H} \widehat{Q}_{2J_{n}^{(0)},p_{n}^{(0)},n}^{(w)}$. 

\begin{lemma}
\label{lem:lowerbound:hatQ}
There exists $r>0$ such that 
\begin{align*}
\mathbb{P}^{n}_{\theta}\left[
(n/p_{n}^{(0)})^{2H} \widehat{Q}_{2J_{n}^{(0)},p_{n}^{(0)},n}^{(w)} \leq r
\right]
=o(1)\ \ \mbox{as $n\to\infty$ uniformly on compact subsets of $\Theta_{\ast}$.}
\end{align*}
\end{lemma}

\begin{proof}
We have
\begin{align*}
	\mathbb{P}^{n}_{\theta}\bigg[
	(n/p_{n}^{(0)})^{2H} &\widehat{Q}_{J_{n}^{(0)},2p_{n}^{(0)},n}^{(w)} \leq r
	\bigg] 
 \\ &=\mathbb{P}^{n}_{\theta}\left[
	\widehat{Q}_{J_{n}^{(0)},2p_{n}^{(0)},n}^{(w)} 
	-\mathbb{E}_{\theta}^{n}\left[\widehat{Q}_{J_{n}^{(0)},2p_{n}^{(0)},n}^{(w)}\right] 
	\leq\left(
	r - \mathbb{E}_{\theta}^{n}\left[\widehat{Q}_{J_{n}^{(0)},2p_{n}^{(0)},n}^{(w)}\right] (n/p_{n}^{(0)})^{2H} 
	\right)(p_{n}^{(0)}/n)^{2H}
	\right].
\end{align*}
By Proposition~\ref{prop:deviation:Q}, we also have
\begin{align*}
	r - \mathbb{E}_{\theta}^{n}\left[\widehat{Q}_{J_{n}^{(0)},2p_{n}^{(0)},n}^{(w)}\right] (n/p_{n}^{(0)})^{2H}
	&= r - (n/p_{n}^{(0)})^{2H} 2J_{p_{n}^{(0)},n} 
	\left(\sigma^{2} \kappa_{p_{n}^{(0)}}^{(w)}(H) (p_{n}^{(0)}/n)^{1+2H} +\tau^{2} \gamma_{p_{n}^{(0)}}^{(w)}n^{-1} \right)\\
	&\leq
	r - \tfrac{2}{3} \sigma^{2} \kappa_{p_{n}^{(0)}}^{(w)}(H) 
\end{align*}
since $J_{p_{n}^{(0)},n} \geq n/(3p_{n}^{(0)})$ for $n$ large enough. 
By Lemma~\ref{lem:wavelet:kappa_p}, we can pick $r>0$ small enough so that $r - \frac{2}{3} \kappa_{p_{n}^{(0)}}^{(w)}(H)  \sigma^{2} \leq - c^{\ast} $ for some $c^{\ast} > 0$, 
independent of $n$ and $\theta=(H,\sigma,\tau)$.
Using Proposition~\ref{prop:deviation:Q} again, we obtain
\begin{align*}
	\mathbb{P}^{n}_{\theta}\left[
	(n/p_{n}^{(0)})^{2H} \widehat{Q}_{J_{n}^{(0)},2p_{n}^{(0)},n}^{(w)} \leq r
	\right]
	&\leq   
	\mathbb{P}^{n}_{\theta}\left[
	\widehat{Q}_{J_{n}^{(0)},2p_{n}^{(0)},n}^{(w)} 
	-\mathbb{E}_{\theta}^{n}\left[\widehat{Q}_{J_{n}^{(0)},2p_{n}^{(0)},n}^{(w)}\right]
	\leq -c^{\ast} (p_{n}^{(0)}/n)^{2H}
	\right]\\
	&\lesssim  
	(n/p_{n}^{(0)})^{4H} \Var_{\theta}^{n}\left[\widehat{Q}_{J_{n}^{(0)},2p_{n}^{(0)},n}^{(w)}\right]\\
	&\lesssim  
	\left(\frac{n}{p_{n}^{(0)}}\right)^{4H}
	\left(
	\left(\frac{p_{n}^{(0)}}{n}\right)^{1+4H}
	+\frac{1}{n(p_{n}^{(0)})^{4K+1}}
	\right)
	\\
	&=\frac{p_{n}^{(0)}}{n} + \frac{n^{4H-1}}{(p_{n}^{(0)})^{4K+4H+1}}
\end{align*}
uniformly on $\Theta$. Therefore, the conclusion follows because $p_{n}^{(0)} = \lfloor n^{2/(2K+3)} \rfloor \vee 2$ and $4H-5 \leq -1$ so that 
\begin{equation}
\label{Q_hat_lb-suff-p0}
	\frac{p_{n}^{(0)}}{n} + \frac{n^{4H-1}}{(p_{n}^{(0)})^{4K+4H+1}}
 \lesssim
 n^{\frac{-(2K+1)}{2K+3}} + n^{4H-1 - \frac{2(4K+4H+1)}{2K+3}}
 =
 n^{\frac{-(2K+1)}{2K+3}} + n^{\frac{(2K+1)(4H-5)}{2K+3}}
 \lesssim 
 n^{\frac{-(2K+1)}{2K+3}}
	\to 0
\end{equation}
as $n \to \infty$, uniformly on $\Theta$.
\end{proof}

Now we return to the proof of Proposition~\ref{prop:deviation:Q}. 
Consider $r > 0$ as in Lemma~\ref{lem:lowerbound:hatQ}. Then we can show that for any $M > 0$, 
\begin{align*}
	&\mathbb{P}^{n}_{\theta}\left[
	v_{n}^{(0)}(H)\left|
	\frac{\widehat{Q}_{J_{n}^{(0)},2p_{n}^{(0)},n}^{(w)}}{\widehat{Q}_{2J_{n}^{(0)},p_{n}^{(0)},n}^{(w)}} - 2^{2H} 
	\right| \geq M
	\right]\leq
	\mathbb{P}^{n}_{\theta}\left[
	(n/p_{n}^{(0)})^{2H} v_{n}^{(0)}(H)
	\left| \widehat{Q}_{J_{n}^{(0)},2p_{n}^{(0)},n}^{(w)}- 2^{2H} \widehat{Q}_{2J_{n}^{(0)},p_{n}^{(0)},n}^{(w)}\right| \geq M r
	\right] + o(1)
\end{align*}
as $n\to\infty$ uniformly on compact subsets of $\Theta_{\ast}$. Moreover, a straightforward calculation gives
\begin{align*}
	(n/p_{n}^{(0)})^{2H} v_{n}^{(0)}(H)
	\left| \widehat{Q}_{J_{n}^{(0)},2p_{n}^{(0)},n}^{(w)}- 2^{2H} \widehat{Q}_{2J_{n}^{(0)},p_{n}^{(0)},n}^{(w)}\right|
	\leq V_{n}^{(1)} + 2^{2H}V_{n}^{(2)} + B_{n},
\end{align*}
where 
\begin{align*}
\begin{cases}
	V_{n}^{(i)} := 
	v_{n}^{(0)}(H) \big(\frac{n}{p_{n}^{(0)}}\big)^{2H} \big|
	\widehat{Q}_{(3-i)J_{n}^{(0)},ip_{n}^{(0)},n}^{(w)}
	-\mathbb{E}_{\theta}^{n}\big[\widehat{Q}_{(3-i)J_{n}^{(0)},ip_{n}^{(0)},n}^{(w)}\big]
	\big|,\ \ i\in\{1,2\},\\
	B_{n} := v_{n}^{(0)}(H) \big(\frac{n}{p_{n}^{(0)}}\big)^{2H} \big|
	\mathbb{E}_{\theta}^{n}\big[\widehat{Q}_{J_{n}^{(0)},2p_{n}^{(0)},n}^{(w)}\big]
	-2^{2H}\mathbb{E}_{\theta}^{n}\big[\widehat{Q}_{2J_{n}^{(0)},p_{n}^{(0)},n}^{(w)}\big]
	\big|. 
\end{cases}
\end{align*}
Therefore, the conclusion follows once we have proved that $\{V_{n}^{(1)}\}_{n\in\mathbb{N}}$, $\{V_{n}^{(2)}\}_{n\in\mathbb{N}}$ and $\{B_{n}\}_{n\in\mathbb{N}}$ are bounded in $\mathbb{P}^{n}_{\theta}$-probability, uniformly over $\Theta$.
\\

We firstly focus on the sequences of variance terms $\{V_{n}^{(i)}\}_{n\in\mathbb{N}}$ for $i\in\{1,2\}$. 
Note that it is enough to prove that for each $i\in\{1,2\}$, $\{\mathbb{E}_{\theta}^{n}[|V_{n}^{(i)}|^{2}]\}_{n\in\mathbb{N}}$ is bounded uniformly over $\Theta$. 
By Proposition~\ref{prop:deviation:Q} and using also the definition of $J_{n}^{(0)}$ and \eqref{Q_hat_lb-suff-p0}, we obtain
\begin{align*}
	\max_{i\in\{1,2\}}
	\mathbb{E}_{\theta}^{n}[|V_{n}^{(i)}|^{2}]
	&\lesssim 
	v_{n}^{(0)}(H)^{2}
	\bigg(\frac{n}{p_{n}^{(0)}}\bigg)^{4H}
	J_{n}^{(0)} \bigg(
	\bigg(\frac{p_{n}^{(0)}}{n}\bigg)^{2+4H}
	+\frac{1}{n^{2}(p_{n}^{(0)})^{4K}}\bigg)
 \\
    &\lesssim
	v_{n}^{(0)}(H)^{2}\bigg(\frac{p_{n}^{(0)}}{n} + \frac{n^{4H-1}}{(p_{n}^{(0)})^{4K+4H+1}}\bigg)
 \\
    &\lesssim
    v_{n}^{(0)}(H)^{2}n^{\frac{-(2K+1)}{2K+3}}
\end{align*}
which is uniformly bounded since $v_{n}^{(0)}(H) \leq n^{\frac{2K+1}{2(2K+3)}}$ by definition. This proves tightness of $\{V_{n}^{(i)}\}_{n\in\mathbb{N}}$ . 
\\

Finally, we focus on the sequence of bias terms $\{B_{n}\}_{n\in\mathbb{N}}$. 
Using \eqref{eq:expec:energy} and Lemmas~\ref{lem:wavelet:properties} and~\ref{lem:noise:structure}, we can show 
\begin{align*}
	&\bigg|
	\mathbb{E}_{\theta}^{n}\Big[\widehat{Q}_{J_{n}^{(0)},2p_{n}^{(0)},n}^{(w)}\Big]
	-2^{2H}\mathbb{E}_{\theta}^{n}\Big[\widehat{Q}_{2J_{n}^{(0)},p_{n}^{(0)},n}^{(w)}\Big]
	\bigg|\\
	&\leq 2^{1+2H}J_{n}^{(0)}\sigma^{2}(p_{n}^{(0)}/n)^{1+2H}
	\Big|\kappa_{2p_{n}^{(0)}}^{(w)}(H) - \kappa_{p_{n}^{(0)}}^{(w)}(H)\Big|
	+J_{n}^{(0)}\tau^{2}n^{-1}\Big|\gamma_{2p_{n}^{(0)}}^{(w)} - 2^{1+2H}\gamma_{p_{n}^{(0)}}^{(w)}\Big|\\
	&\leq 2^{1+2H}J_{n}^{(0)}\sigma^{2}(p_{n}^{(0)}/n)^{1+2H}
	\Big(|\overline{\kappa}_{2p_{n}^{(0)}}^{(w)}(H)| + |\overline{\kappa}_{p_{n}^{(0)}}^{(w)}(H)|\Big)
	+J_{n}^{(0)}\tau^{2}n^{-1}
	\Big(|\gamma_{2p_{n}^{(0)}}^{(w)}| + 2^{1+2H}|\gamma_{p_{n}^{(0)}}^{(w)}|\Big)\\
	&\lesssim 
	(p_{n}^{(0)}/n)^{2H}(p_{n}^{(0)})^{-\{(2H)\wedge 1\}} + (p_{n}^{(0)})^{-2K-1}
\end{align*}
uniformly on $\Theta$. By definition of $p_{n}^{(0)}$ , we obtain
\begin{align*}
	v_{n}^{(0)}(H)^{-1} B_{n}\lesssim 
    \frac{1}{(p_{n}^{(0)})^{(2H)\wedge 1}} + \frac{n^{2H}}{(p_{n}^{(0)})^{2K+2H+1}}
    \lesssim 
     n^{-2\frac{(2H)\wedge 1}{2K+3}} + n^{2H - 2\frac{2K+2H+1}{2K+3}}
     =
     n^{-2\frac{(2H)\wedge 1}{2K+3}} + n^{\frac{(4K+2)(H-1)}{2K+3}}
\end{align*}
so that by definition of $v_{n}^{(0)}(H)$, $B_{n}$ is uniformly bounded on $\Theta$. This completes the proof.

\subsection{Proof of Lemma~\ref{lemma:uncorrectedQV:tau}}
\label{Sec:lemma:uncorrectedQV:tau:proof}

Since $\tau \in [\tau_-, \tau_+]$ and $t \mapsto t^{2}$ is invertible on $(0,\infty)$ with inverse uniformly Lipschitz on the compact sets of $(0,\infty)$, it suffices to prove that
\begin{align*}
	\left\{
	u_{n}^{(0)}(H)
	\left|
	n\frac{
	\widehat{Q}_{n}	
	}{
	(n-1) \gamma_{1}^{(1)}
	}
	-\tau^{2}
	\right|
	\right\}_{n\in\mathbb{N}}
\end{align*}
is bounded in $\mathbb{P}^{n}_{\theta}$-probability, uniformly over $\Theta$. 
Moreover, since $\inf_{n\in\mathbb{N}} n^{-1} (n-1)\gamma_{1}^{(1)}> 0$, it suffices to prove that
\begin{equation}\label{stoch_bdd_tau1}
	\left\{
	u_{n}^{(0)}(H)
	\left|
	\widehat{Q}_{n} - \tau^{2} n^{-1}(n-1) \gamma_{1}^{(1)}
	\right|
	\right\}_{n\in\mathbb{N}}
\end{equation}
is bounded in $\mathbb{P}^{n}_{\theta}$-probability, uniformly over $\Theta$. 
Note that, for each $n\in\mathbb{N}$, the random variable in the sequence \eqref{stoch_bdd_tau1} is bounded above by $V_{n}+B_{n}$, where
\begin{align*}
\begin{cases}
	V_{n} = u_{n}^{(0)}(H)\big|
	\widehat{Q}_{n} - \mathbb{E}_{\theta}^{n}[\widehat{Q}_{n}]
	\big|, \\
	B_{n}= 
	u_{n}^{(0)}(H)\big|\mathbb{E}_{\theta}^{n}[\widehat{Q}_{n}] - \tau^{2} n^{-1}(n-1) \gamma_{1}^{(1)}\big|.  
\end{cases}
\end{align*}
Therefore, the conclusion follows once we have proved that both $\{V_{n}\}_{n\in\mathbb{N}}$ and $\{B_{n}\}_{n\in\mathbb{N}}$ are bounded in $\mathbb{P}^{n}_{\theta}$-probability, uniformly over $\Theta$. 

We start with the variance terms $\{V_{n}\}_{n\in\mathbb{N}}$. 
By Proposition~\ref{prop:deviation:Q}, we have
\begin{align*}
	\mathbb{E}_{\theta}^{n}[V_{n}^{2}]
	\lesssim
	u_{n}^{(0)}(H)^{2}n (n^{-2-4H} + n^{-2})
	\lesssim  u_{n}^{(0)}(H)^{2}n^{-1}
\end{align*}
uniformly on $\Theta$ and therefore $\{\mathbb{E}_{\theta}^{n}[V_{n}^{2}]\}_{n\in\mathbb{N}}$ is bounded uniformly which ensure tightness of $V_n$.

For the bias terms $\{u_{n}^{(0)}(H)B_{n}(\theta)\}_{n\in\mathbb{N}}$, we use Proposition~\ref{prop:deviation:Q} to get
\begin{align*}
	B_{n}
	\lesssim u_{n}^{(0)}(H)(n-1)n^{-(1+2H)}\leq u_{n}^{(0)}(H)n^{-2H}
\end{align*}
which is bounded in $\mathbb{P}^{n}_{\theta}$-probability, uniformly over $\Theta$ since $u_{n}^{(0)}(H)n^{-2H}\leq 1$ for each $n\in\mathbb{N}$.

\subsection{Proof of Proposition~\ref{prop:corrected:sigma} }
\label{sec:prop:corrected:sigma:proof}

Since $\sigma \in [\sigma_-, \sigma_+]$ and $t \mapsto t^{2}$ is invertible on $(0,\infty)$ with inverse uniformly Lipschitz on the compact sets of $(0,\infty)$, it suffices to prove that
\begin{align*}
	\left\{
	s_{n}(H)\left|
	\frac{
		\widehat{Q}_{\widehat{J}_{n}^{\star}(\widehat{H}_{n}),\widehat{p}_{n}(\widehat{H}_{n}), n}^{(w)}
		- 
		\widehat{J}_{n}^{\star}(\widehat{H}_{n}) \widehat{\tau}_{n}^{2} \gamma_{\widehat{p}_{n}(\widehat{H}_{n})}^{(w)}n^{-1}
	}{
		\widehat{J}_{n}^{\star}(\widehat{H}_{n}) \kappa_{\widehat{p}_{n}(\widehat{H}_{n})}^{(w)}(\widehat{H}_{n}) (\widehat{p}_{n}(\widehat{H}_{n})/n)^{1+2\widehat{H}_{n}}
	}
	- \sigma^{2}
	\right|
	\right\}_{n\in\mathbb{N}}
\end{align*}
is bounded in $\mathbb{P}^{n}_{\theta}$-probability, uniformly over $\Theta$. Moreover, Using Lemma~\ref{lem:behaviour_adaptive_level} and the assumption on $\{\widehat{H}_{n}\}_{n\in\mathbb{N}}$, 
it suffices to show that for each $b \in \{-1, 0, 1\}$, the sequence
\begin{align*}
	\left\{
	s_{n}(H)\left|
	\frac{
	\widehat{Q}_{J_{n}^{\star,b}(H),p_{n}^{b}(H), n}^{(w)} 
	-J_{n}^{\star,b}(H)\widehat{\tau}_{n}^{2} \gamma_{p_{n}^{b}(H)}^{(w)} n^{-1}
	}{
	J_{n}^{\star,b}(H)\kappa_{p_{n}^{b}(H)}^{(w)}(\widehat{H}_{n}) (p_{n}^{b}(H)/n)^{1+2\widehat{H}_{n}}
	}
	-\sigma^{2}
	\right|
	\right\}_{n\in\mathbb{N}}
\end{align*}
is bounded in $\mathbb{P}^{n}_{\theta}$-probability, uniformly over $\Theta$. 
Furthermore, since we know that $|\kappa_{p_{n}^{b}(H)}^{(w)}(\widehat{H}_{n})|^{-1}\lesssim 1$ uniformly on $\Theta$ Using Lemma~\ref{lem:wavelet:kappa_p} and 
$J_{n}^{\star,b}(H)(n/p_{n}^{b}(H))\to  1$ 
as $n\to\infty$ uniformly on $\Theta$ by the definition of $J_{n}^{\star,b}(H)$, it suffices to prove that 
\begin{align}\label{eq:large_expr_tightness_sigma}
	\scalebox{0.95}{$
	\left\{
	s_{n}(H) \left(
	\frac{n}{p_{n}^{b}(H)}\right)^{2\widehat{H}_{n}}
	\left|
	\widehat{Q}_{J_{n}^{\star,b}(H),p_{n}^{b}(H), n}^{(w)}
	-J_{n}^{\star,b}(H)\left(
	\widehat{\tau}_{n}^{2} \gamma_{p_{n}^{b}(H)}^{(w)}n^{-1}
	+\sigma^{2}\kappa_{p_{n}^{b}(H)}^{(w)}(\widehat{H}_{n}) \left(\frac{p_{n}^{b}(H)}{n}\right)^{1+2\widehat{H}_{n}}
	\right)
	\right|
	\right\}_{n\in\mathbb{N}}
	$}
\end{align}
is bounded in $\mathbb{P}^{n}_{\theta}$-probability, uniformly over $\Theta$. 
Recall then from Proposition~\ref{prop:deviation:Q} that 
\begin{align*}
\mathbb{E}_{\theta}^{n}\Big[\widehat{Q}^{(w)}_{J_{n}^{\star,b}(H),p_{n}^{b}(H),n}\Big]
=
J_{n}^{\star,b}(H)\bigg(
\sigma^{2}\kappa_{p_{n}^{b}(H)}^{(w)}(H)
\bigg(
\frac{p_{n}^{b}(H)}{n}\bigg)^{1+2H}
+\tau^{2}\gamma_{p_{n}^{b}(H)}^{(w)}n^{-1} 
\bigg)
\end{align*}
so that 
for each $n\in\mathbb{N}$, the random variable in the sequence \eqref{eq:large_expr_tightness_sigma} is bounded above by
\begin{align*}
	F_{n}(V_{n}+B_{n}^{(\tau)}+B_{n}^{(\kappa)}+B_{n}^{(H)}),
\end{align*}
where 
\begin{align*}
\begin{array}{l}
    F_{n} := (\frac{n}{p_{n}^{b}(H)})^{2(\widehat{H}_{n}-H)}
    \\
	V_{n} := s_{n}(H) \bigg(\frac{n}{p_{n}^{b}(H)}\bigg)^{2H}
	\bigg|
	\widehat{Q}^{(w)}_{J_{n}^{\star,b}(H),p_{n}^{b}(H),n}
	- 
	\mathbb{E}_{\theta}^{n}\bigg[\widehat{Q}^{(w)}_{J_{n}^{\star,b}(H),p_{n}^{b}(H),n}\bigg]
	\bigg|,
	\\
	B_{n}^{(\tau)} := s_{n}(H) 
	\left(\frac{n}{p_{n}^{b}(H)}\right)^{2H}
	J_{n}^{\star,b}(H)
	\gamma_{p_{n}^{b}(H)}^{(w)}n^{-1} 
	\left|\widehat{\tau}_{n}^{2}-\tau^{2}\right|,
	\\
	B_{n}^{(\kappa)} := s_{n}(H) 
	J_{n}^{\star,b}(H) \sigma^{2}
	\big|
	\kappa_{p_{n}^{b}(H)}^{(w)}(\widehat{H}_{n})
	-\kappa_{p_{n}^{b}(H)}^{(w)}(H)  
	\big|
	\frac{p_{n}^{b}(H) }{n},
	\\
	B_{n}^{(H)} := s_{n}(H) 
	J_{n}^{\star,b}(H) \sigma^{2}
	\kappa_{p_{n}^{b}(H)}^{(w)}(\widehat{H}_{n})
	\frac{p_{n}^{b}(H) }{n}
	\Big|
	\Big(\frac{p_{n}^{b}(H)}{n}\Big)^{2(\widehat{H}_{n}-H)}
	-1
	\Big|.
\end{array}\end{align*}
Therefore, the proof is down to proving that each sequence of these five terms is bounded in $\mathbb{P}^{n}_{\theta}$-probability, uniformly over $\Theta$.
\\

\noindent\textbf{Uniform boundedness of $\{F_{n}\}_{n\in\mathbb{N}}$. }
Note first that
\begin{align*}
F_{n}=\left( \frac{n}{p_{n}^{b}(H)}\right)^{2(\widehat{H}_{n} - H)}
	= \exp \left(2 v_{n}(H)(\widehat{H}_{n} - H) \cdot v_{n}(H)\log\left(\frac{n}{p_{n}^{b}(H)}\right) \right).
\end{align*}
Therefore, using that $\{v_{n}(H)(\widehat{H}_{n} - H)\}_{n\in\mathbb{N}}$ is bounded in $\mathbb{P}^{n}_{\theta}$-probability, uniformly over $\Theta$ and that we have $| v_{n}(H)\log ( \frac{n}{p_{n}^{b}(H)})| \lesssim  v_{n}(H)\log{n} \to  0$ as $n\to\infty$ uniformly on $\Theta$, we obtain that $\{F_{n}\}_{n\in\mathbb{N}}$ is uniformly bounded.
\\

\noindent\textbf{Uniform boundedness of $\{V_{n}\}_{n\in\mathbb{N}}$. }
Note that it suffices to prove that $\{\mathbb{E}_{\theta}^{n}[V_{n}^{2}]\}_{n\in\mathbb{N}}$ is bounded uniformly over $\Theta$. 
By Proposition~\ref{prop:deviation:Q}, we have
\begin{align*}
	\mathbb{E}_{\theta}^{n}[V_{n}^{2}]
	\lesssim 
	s_{n}(H)^{2}
	\left(\frac{n}{p_{n}^{b}(H)}\right)^{4H}
	J_{n}^{\star,b}(H) \left(
	\left(\frac{p_{n}^{b}(H)}{n}\right)^{2+4H}
	+\frac{1}{n^{2}p_{n}^{b}(H)^{4K}}\right)
\end{align*}
uniformly on $\Theta$. Then, by the definitions of $p_{n}^{b}(H)$ and $J_{n}^{\star,b}(H)$, we have
\begin{align*}
	\mathbb{E}_{\theta}^{n}[V_{n}^{2}]
	\lesssim 
	s_{n}(H)^{2}
	\left(
	\frac{p_{n}^{b}(H)}{n}
	+\frac{n^{4H-1}}{p_{n}^{b}(H)^{4H+4K+1}}
	\right)
	\lesssim 
	s_{n}(H)^{2} n^{-\frac{2K+1}{\Diamond(H)}}
\end{align*}
uniformly on $\Theta$ and the conclusion follows because $s_{n}(H)^{2} n^{-\frac{2K+1}{\Diamond(H)}}\leq 1$ by definition of $s_{n}(H)$. 
\\

\noindent\textbf{Uniform boundedness of $\{B_{n}^{(\tau)}\}_{n\in\mathbb{N}}$. } 
We use again the definition of $p_{n}^{b}(H)$ and $J_{n}^{\star,b}(H)$ together with Lemma~\ref{lem:noise:structure} ensuring that $\gamma_{p_{n}^{b}(H)}^{(w)}\lesssim {p_{n}^{b}(H)}^{-2K}$ uniformly on $\Theta$ to show that
\begin{align*}
	B_{n}^{(\tau)} \leq s_{n}(H)u_{n}(H)^{-1}  \times u_{n}(H)
	\left|\widehat{\tau}_{n}^{2}-\tau^{2}\right|.
\end{align*}
Then the conclusion follows because $\{u_{n}(H)|\widehat{\tau}_{n}^{2}-\tau^{2}|\}_{n\in\mathbb{N}}$ is bounded in $\mathbb{P}^{n}_{\theta}$-probability, uniformly over $\Theta$ from the assumption on $\{\widehat{\tau}_{n}\}_{n\in\mathbb{N}}$. 
\\

\noindent\textbf{Uniform boundedness of $\{B_{n}^{(\kappa)}\}_{n\in\mathbb{N}}$. }
By the definitions of $p_{n}^{b}(H)$, $J_{n}^{\star,b}(H)$ and $s_{n}(H)$, Taylor formula and Lemma~\ref{lem:wavelet:kappa_p}, we have
\begin{align*}
	B_{n}^{(\kappa)} 
	\lesssim  s_{n}(H) |\widehat{H}_{n}-H| \sup_{\widetilde{H}\in [H_{-}, H_{+}]}\left|\partial_{H}\kappa_{p_{n}^{b}(H)}^{(w)}(\widetilde{H})\right|
	\lesssim  v_{n}(H)|\widehat{H}_{n}-H|
\end{align*}
uniformly on $\Theta$. Then the conclusion follows from the assumption on $\{\widehat{H}_{n}\}_{n\in\mathbb{N}}$. 
\\

\noindent\textbf{Uniform boundedness of $\{B_{n}^{(H)}\}_{n\in\mathbb{N}}$. }
Note that $J_{n}^{\star,b}(H)(p_{n}^{b}(H)/n)$ is bounded uniformly on $H$ by the definition of $ $. Moreover, $\kappa_{p_{n}^{b}(H)}^{(w)}(\widehat{H}_{n})$ is bounded by Lemma~\ref{lem:wavelet:kappa_p}. Therefore, it suffices to prove that $s_{n}(H) \Big |
		  \left(\frac{p_{n}^{b}(H)}{n}\right)^{2(\widehat{H}_{n}-H)}
		  -
		  1
\Big |$ is tight uniformly over $\Theta$. But recall from Taylor's formula that for any real $x$, we have $|e^x - 1 | \leq e^{|x|}|x|$. Thus 
\begin{align*}
s_{n}(H)\bigg|
\Big(\frac{p_{n}^{b}(H)}{n}\Big)^{2(\widehat{H}_{n}-H)}-1
\bigg|
&\leq
s_{n}(H) 
\Big( \frac{n}{p_{n}^{b}(H)} \Big)^{2|\widehat{H}_{n}-H|}\Big|
2(\widehat{H}_{n}-H)\log\Big(\frac{p_{n}^{b}(H)}{n}\Big)
\Big|
\\
&\leq
2
s_{n}(H)v_{n}(H)^{-1}\log{n} \times 
\Big( \frac{n}{p_{n}^{b}(H)} \Big)^{2|\widehat{H}_{n}-H|} \times
v_{n}(H) \Big|\widehat{H}_{n}-H\Big| 
\end{align*}
and we conclude the proof using that $\{s_{n}(H)v_{n}(H)^{-1}\log{n}\}_{n\in\mathbb{N}}$ is bounded, that $\{(n/p_{n}^{b}(H))^{2|\widehat{H}_{n}-H|}\}_{n\in\mathbb{N}}$ is uniformly tight by applying the same ideas as for $\{F_{n}\}_{n\in\mathbb{N}}$ and that $\{v_{n}(H) |\widehat{H}_{n}-H|\}_{n\in\mathbb{N}}$ is uniformly tight.

\subsection{Proof of Proposition~\ref{prop:corrected:H} }
\label{sec:prop:corrected:H:proof}

Since $H \in [H_{-}, H_{+}]$ and $t \mapsto 2^{2t}$ is invertible on $(0,1)$ with inverse uniformly Lipschitz on the compact sets of $(0,1)$, it suffices to prove that 
\begin{align*}
	\left\{
	s_{n}(H)\left(
	\frac{
	\widehat{Q}^{(w),c}_{2\widehat{J}_{n}(\widehat{H}_{n}),\widehat{p}_{n}(\widehat{H}_{n}),n}(\widehat{\theta}_{n})
	}{
	\widehat{Q}^{(w),c}_{\widehat{J}_{n}(\widehat{H}_{n}),2\widehat{p}_{n}(\widehat{H}_{n}),n}(\widehat{\theta}_{n})
	}
	-2^{2H}
	\right)
	\right\}_{n\in\mathbb{N}}
\end{align*}
is bounded in $\mathbb{P}^{n}_{\theta}$-probability, uniformly over $\Theta$. 
Moreover, Using Lemma~\ref{lem:behaviour_adaptive_level} and the assumption on $\{\widehat{H}_{n}\}_{n\in\mathbb{N}}$, it suffices to prove that for each $b \in \{-1, 0, 1\}$, the sequence
\begin{align*}
	\left\{
	s_{n}(H)\left(
	\frac{\widehat{Q}^{(w),c}_{2J_{n}^{b}(H),p_{n}^{b}(H),n}(\widehat{\theta}_{n})}
    {\widehat{Q}^{(w),c}_{J_{n}^{b}(H),2p_{n}^{b}(H),n}(\widehat{\theta}_{n})}
	-2^{2H}
	\right)
	\right\}_{n\in\mathbb{N}}
\end{align*}
is bounded in $\mathbb{P}^{n}_{\theta}$-probability, uniformly over $\Theta$, where we write $J_{n}^{b}(H)=J_{n,p_{n}^{b}(H)}$. 
We start by proving a result similar to Lemma~\ref{lem:lowerbound:hatQ}.

\begin{lemma}
\label{lem:lowerbound:hatQc}
There exists $r > 0$ such that 
\begin{align*}
	\mathbb{P}^{n}_{\theta}\left[
	\left(\frac{n}{p_{n}^{b}(H)}\right)^{2H}
	\widehat{Q}^{(w),c}_{J_{n}^{b}(H),2p_{n}^{b}(H),n}(\widehat{\theta}_{n}) \leq r
	\right]=o(1)\ \ \mbox{as $n\to\infty$}
\end{align*}
uniformly on $\Theta$.
\end{lemma}

\begin{proof}
Proceeding as in the proof of Lemma~\ref{lem:lowerbound:hatQ}, but using
\begin{align*}
	\frac{p_{n}^{b}(H)}{n} + \frac{n^{4H-1}}{p_{n}^{b}(H)^{4K+4H+1}}
    \lesssim
    n^{\frac{-(2K+1)}{2H+2K+1}} + n^{-1 + \frac{
    2H
    }{2H+2K+1}}
	\to 0\ \ \mbox{as $n\to\infty$}
\end{align*}
instead of \eqref{Q_hat_lb-suff-p0}, 
we can show that for $r>0$ small enough,  
\begin{align*}
	\mathbb{P}^{n}_{\theta}\bigg[
	\bigg(\frac{n}{p_{n}^{b}(H)}\bigg)^{2H}\bigg(
	\widehat{Q}^{(w)}_{J_{n}^{b}(H),2p_{n}^{b}(H),n} - 2J_{n}^{b}(H)\tau^{2}\gamma_{p_{n}^{b}(H)}^{(w)}n^{-1} 
	\bigg) \leq r
	\bigg]=o(1)\ \ \mbox{as $n\to\infty$}
\end{align*}
uniforly on $\Theta$. Moreover, we know from \eqref{eq:Def:correctedenergy} that $\widehat{Q}^{(w),c}_{J_{n}^{b}(H),2p_{n}^{b}(H),n}(\widehat{\theta}_{n})$ is given by
\begin{align*}
\widehat{Q}^{(w),c}_{J_{n}^{b}(H),2p_{n}^{b}(H),n}(\widehat{\theta}_{n})
&=\widehat{Q}^{(w)}_{J_{n}^{b}(H),2p_{n}^{b}(H),n}
-2J_{n}^{b}(H)\tau^{2}\gamma_{p_{n}^{b}(H)}^{(w)}n^{-1}
\\
&\quad-
2J_{n}^{b}(H)\widehat{\sigma}_{n}^{2}
\overline{\kappa}_{p_{n}^{b}(H)}^{(w)}(\widehat{H}_{n})\left(\frac{p_{n}^{b}(H)}{n}\right)^{1+2\widehat{H}_{n}} 
-2J_{n}^{b}(H)(\widehat{\tau}_{n}^{2} - \tau^{2}) \gamma_{p_{n}^{b}(H)}^{(w)}n^{-1}.
\end{align*}
Proceeding as for 
$\{F_{n}\}_{n\in\mathbb{N}}$ in the proof of Proposition~\ref{prop:corrected:sigma}, we also see that $\{(n/p_{n}^{b}(H))^{2(\widehat{H}_{n} -H)}\}_{n\in\mathbb{N}}$ is bounded in $\mathbb{P}^{n}_{\theta}$-probability, uniformly over $\Theta$ and therefore the proof of Lemma~\ref{lem:lowerbound:hatQc} is down to proving that, uniformly over $\Theta$, we have
\begin{equation}\label{eq-add1}
	J_{n}^{b}(H)\widehat{\sigma}_{n}^{2}
	\overline{\kappa}_{p_{n}^{b}(H)}^{(w)}(\widehat{H}_{n}) \frac{p_{n}^{b}(H)}{n}
	 \to 0 \;\;\;
	\text{ and }\;\;\;
	\bigg(\frac{n}{p_{n}^{b}(H)} \bigg)^{2H} J_{n}^{b}(H)(\widehat{\tau}_{n}^{2} - \tau^{2}) \gamma_{p_{n}^{b}(H)}^{(w)}n^{-1}
	\to 0.
\end{equation}

We start with the first limit. Since $\widehat{H}_{n}\in[H_{-},H_{+}]$ and $\widehat{\sigma}_{n}\in[\sigma_-,\sigma_+]$ by assumption and using also that $J_{n}^{b}(H)\frac{p_{n}^{b}(H)}{n}\lesssim 1$ uniformly on $\Theta$ by construction and Lemma~\ref{lem:wavelet:kappa_p} to bound the difference $\overline{\kappa}_{p_{n}^{b}(H)}^{(w)}(\widehat{H}_{n}) = \kappa_{p_{n}^{b}(H)}^{(w)}(\widehat{H}_{n}) - \kappa_{\infty}^{(w)}(\widehat{H}_{n})$, we obtain
\begin{align*}
	J_{n}^{b}(H) \widehat{\sigma}_{n}^{2} \overline{\kappa}_{p_{n}^{b}(H)}^{(w)}(\widehat{H}_{n}) \frac{p_{n}^{b}(H)}{n}
	\lesssim 
	\frac{1}{p_{n}^{b}(H)^{(2\widehat{H}_{n})\wedge 1}} 
	\leq
	\frac{1}{p_{n}^{b}(H)^{(2H_{-})\wedge 1}}
	\to 0\ \ \mbox{as $n\to\infty$.}
\end{align*}
uniformly on $\Theta$, which proves $\big \{
	J_{n}^{b}(H)\widehat{\sigma}_{n}^{2}
	\overline{\kappa}_{p_{n}^{b}(H)}^{(w)}(\widehat{H}_{n}) \frac{p_{n}^{b}(H)}{n}
	\big \}_{n\in\mathbb{N}}$ is bounded in $\mathbb{P}_{\theta}^{n}$-probability uniformly over $\Theta$.

Finally, we focus on the second limit in \eqref{eq-add1}. 
Using Lemma~\ref{lem:noise:structure} to bound $\gamma_{p_{n}^{b}(H)}^{(w)}$ and the definitions of $p_{n}^{b}(H)$ and $J_{n}^{b}(H)$, we can show that
\begin{align*}
	\bigg\{
	\bigg(\frac{n}{p_{n}^{b}(H)} \bigg)^{2H} J_{n}^{b}(H)\gamma_{p_{n}^{b}(H)}^{(w)}n^{-1}
	\bigg\}_{n\in\mathbb{N}}
\end{align*}
is bounded in $\mathbb{P}_{\theta}^{n}$-probability uniformly over $\Theta$. 
Then the conclusion follows because $\{u_{n}(H)(\widehat{\tau}_{n}^{2} - \tau^{2})\}_{n\in\mathbb{N}}$ is bounded in $\mathbb{P}^{n}_{\theta}$-probability, uniformly over $\Theta$ from the assumption.
\end{proof}

We now return to the proof of 
Proposition~\ref{prop:corrected:H}. Consider $r>0$ as in Lemma~\ref{lem:lowerbound:hatQc}. Then we can show that for any $M > 0$, 
\begin{align*}
	&\mathbb{P}^{n}_{\theta}\left[
	s_{n}(H)\left|
	\frac{\widehat{Q}^{(w),c}_{2J_{n}^{b}(H),p_{n}^{b}(H),n}(\widehat{\theta}_{n})}{\widehat{Q}^{(w),c}_{J_{n}^{b}(H),2p_{n}^{b}(H),n}(\widehat{\theta}_{n})}
	-2^{2H} 
	\right|\geq M
	\right]
	\\
	&\leq 
	\mathbb{P}^{n}_{\theta}\left[
	s_{n}(H) \left(\frac{n}{p_{n}^{b}(H)}\right)^{2H} \left|
	\widehat{Q}^{(w),c}_{2J_{n}^{b}(H),p_{n}^{b}(H),n}(\widehat{\theta}_{n})
	-2^{2H}
	\widehat{Q}^{(w),c}_{J_{n}^{b}(H),2p_{n}^{b}(H),n}(\widehat{\theta}_{n})
	\right|
	\geq Mr
	\right]
	+ o(1)
\end{align*}
as $n\to\infty$ uniformly on $\Theta$ and it is therefore enough to prove that
\begin{align*}
	s_{n}(H) \left(\frac{n}{p_{n}^{b}(H)}\right)^{2H} \left|
	\widehat{Q}^{(w),c}_{2J_{n}^{b}(H),p_{n}^{b}(H),n}(\widehat{\theta}_{n})
	-
	2^{2H}
	\widehat{Q}^{(w),c}_{J_{n}^{b}(H),2p_{n}^{b}(H),n}(\widehat{\theta}_{n})
	\right|
\end{align*}
is uniformly bounded in probability. But by its definition, it is bounded by above by $V_{n}^{(1)} + 2^{2H}V_{n}^{(2)} + B_{n}^{(\tau)} + 2^{1+2H}B_{n,1}^{(\kappa)} + B_{n,2}^{(\kappa)}$, where, for $i=1, 2$, we have
\begin{align*}
\begin{cases}
	V_{n}^{(i)} := 
	s_{n}(H) \left(\frac{n}{p_{n}^{b}(H)}\right)^{2H}\left|
	\widehat{Q}^{(w)}_{(3-i)J_{n}^{b}(H),ip_{n}^{b}(H),n}
	-\mathbb{E}_{\theta}^{n}\left[\widehat{Q}^{(w)}_{(3-i)J_{n}^{b}(H),ip_{n}^{b}(H),n}\right]
	\right|,\\
	B_{n}^{(\tau)} :=
	s_{n}(H) \left(\frac{n}{p_{n}^{b}(H)}\right)^{2H}
	J_{n}^{b}(H)\left(\gamma_{2p_{n}^{b}(H)}^{(w)}+ 2^{1+2H}\gamma_{p_{n}^{b}(H)}^{(w)}\right) n^{-1}
	\left|\widehat{\tau}^{2}-\tau^{2}\right|,\\
	B_{n,i}^{(\kappa)} :=
	s_{n}(H) \left(\frac{n}{p_{n}^{b}(H)}\right)^{2H} J_{n}^{b}(H)
	\left|
	\widehat{\sigma}_{n}^{2} \overline{\kappa}_{ip_{n}^{b}(H)}^{(w)}(\widehat{H}_{n}) \left(\frac{ip_{n}^{b}(H)}{n}\right)^{1+2\widehat{H}_{n}}
	-\sigma^{2} \overline{\kappa}_{ip_{n}^{b}(H)}^{(w)}(H) \left(\frac{ip_{n}^{b}(H)}{n}\right)^{1+2H}
	\right|.
 \end{cases}
\end{align*}
Thus, the proof is down to proving that each sequence of these terms is bounded in $\mathbb{P}^{n}_{\theta}$-probability, uniformly over $\Theta$.
We omit the proofs of $\{V_{n}^{(i)}\}_{n\in\mathbb{N}}$ and $\{B_{n}^{(\tau)}\}_{n\in\mathbb{N}}$ as they are similar to the proofs of  $\{V_{n}\}_{n\in\mathbb{N}}$ and $\{B_{n}^{(\tau)}\}_{n\in\mathbb{N}}$ in Proposition~\ref{prop:corrected:sigma} respectively. 
Instead we focus on the sequence $\{B_{n,i}^{(\kappa)}\}_{n\in\mathbb{N}}$ with $i=1,2$ and <e consider only in the case of $i=1$ since the result can be proved similarly for $i=2$. 
By definition of $J_{n}^{b}(H)$ and using $0 \leq \widehat{\sigma}_{n} \leq \sigma_+$ and results similar to the uniform stochastic boundedness of $\{F_{n}\}_{n\in\mathbb{N}}$ and $\{B_{n}^{(H)}\}_{n\in\mathbb{N}}$ in the proof of Proposition~\ref{prop:corrected:sigma}, we can show 
\begin{align*}
	B_{n,1}^{(\kappa)}
	\lesssim &
	s_{n}(H)\left|
	\widehat{\sigma}_{n}^{2}-\sigma^{2}
	\right|
	\left|\overline{\kappa}_{p_{n}^{b}(H)}^{(w)}(H)\right|
	+s_{n}(H) \widehat{\sigma}_{n}^{2} \left|\overline{\kappa}_{p_{n}^{b}(H)}^{(w)}(H)\right|
	\left|
	\left(\frac{p_{n}^{b}(H)}{n}\right)^{2(\widehat{H}_{n}-H)} - 1
	\right|\\
	&+s_{n}(H) \widehat{\sigma}_{n}^{2} \left|
	\overline{\kappa}_{p_{n}^{b}(H)}^{(w)}(\widehat{H}_{n}) - \overline{\kappa}_{p_{n}^{b}(H)}^{(w)}(H)
	\right|
	\left(\frac{p_{n}^{b}(H)}{n}\right)^{2(\widehat{H}_{n}-H)}\\
	\\
	\lesssim &
	s_{n}(H) \left|\widehat{\sigma}_{n}^{2}-\sigma^{2}\right| \left|\overline{\kappa}_{p_{n}^{b}(H)}^{(w)}(H)\right|
	+s_{n}(H) \left|\widehat{H}_{n}-H\right| 
	\left(
	\left|\overline{\kappa}_{p_{n}^{b}(H)}^{(w)}(H)\right|\log{n}
	+D_{n}
	\right)
\end{align*}
uniformly on $\Theta$, where
\begin{align*}
	D_{n} := \sup_{H \wedge \widehat{H}_{n} \leq H^{\ast} \leq H \vee \widehat{H}_{n}}
	\left| \partial_{H}\overline{\kappa}_{p_{n}^{b}(H)}^{(w)}(H^{\ast}) \right|.
\end{align*}
Since Lemma~\ref{lem:wavelet:kappa_p} gives $|\partial_{H}^{i}\overline{\kappa}_{p_{n}^{b}(H)}^{(w)}(H)| \lesssim p_{n}^{b}(H)^{-\{(2H)\wedge 1\}}$ for $i=0,1$ uniformly on $\Theta$, we obtain
\begin{align*}
	B_{n,1}^{(\kappa)}
	\lesssim &
	s_{n}(H) \left(v_{n}(H)|\log{n}|^{-1}p_{n}^{b}(H)^{2H \wedge 1}\right)^{-1}\\
	&\times\left\{
	v_{n}(H)|\log{n}|^{-1}\left|\widehat{\sigma}_{n}^{2}-\sigma^{2}\right|
	+v_{n}(H)\left|\widehat{H}_{n}-H\right|
	\left(
	1+
	|\log{n}|^{-1}p_{n}^{b}(H)^{(2H \wedge 1)-\{2(H \wedge \widehat{H}_{n})\wedge 1\}}
	\right)
	\right\}.
\end{align*}
uniformly on $\Theta$. Then note that we have
\begin{align*}
	p_{n}^{b}(H)^{(2H \wedge 1)-\{2(H \wedge \widehat{H}_{n})\wedge 1\}}
	\leq
	\exp \left(v_{n}(H)|(2H \wedge 1)-\{2(H \wedge \widehat{H}_{n})\wedge 1\}| \times v_{n}(H)^{-1}\log{p_{n}^{b}(H)} \right)
\end{align*}
and the RHS of the above inequality converges to $1$ in $\mathbb{P}^{n}_{\theta}$-probability (uniformly over $\Theta$) as $n\to\infty$ since $\{v_{n}(H)|(2H \wedge 1)-\{2(H \wedge \widehat{H}_{n})\wedge 1\}|\}_{n\in\mathbb{N}}$ is bounded in $\mathbb{P}^{n}_{\theta}$-probability and $v_{n}(H)^{-1}\log{n} \to 0$ as $n\to\infty$ uniformly on $\Theta$. 
Then, the conclusion follows since $s_{n}(H) \left(v_{n}(H)|\log{n}|^{-1}p_{n}^{b}(H)^{2H \wedge 1}\right)^{-1}\leq 1$ and the sequences $\{v_{n}(H)|\widehat{H}_{n}-H|\}_{n\in\mathbb{N}}$ and $\{v_{n}(H)|\log{n}|^{-1}|\widehat{\sigma}_{n}-\sigma|\}_{n\in\mathbb{N}}$ are bounded in $\mathbb{P}^{n}_{\theta}$-probability, uniformly over $\Theta$. 

\subsection{Proof of Proposition~\ref{prop:corrected:tau} }
\label{sec:prop:corrected:tau:proof}

Proceeding as in the proof of Lemma~\ref{lemma:uncorrectedQV:tau}, it suffices to prove that
\begin{align*}
	\left\{
	s_{n}(H)
	\left|
	\widehat{Q}_{n} - (n-1) \widehat{\sigma}_{n}^{2} \kappa_{1}^{(1)}(\widehat{H}_{n}) n^{-1-2\widehat{H}_{n}}
	-\tau^{2} n^{-1}(n-1) \gamma_{1}^{(1)}
	\right|
	\right\}_{n\in\mathbb{N}}
\end{align*}
is bounded in $\mathbb{P}^{n}_{\theta}$-probability, uniformly over $\Theta$. 
For each $n\in\mathbb{N}$, the random variable in the above sequence is bounded above by $\overline{V}_{n} + \overline{B}_{n}$, where 
\begin{align*}
\begin{cases}    
	\overline{V}_{n} := s_{n}(H)\Big|
	\widehat{Q}_{n} - \mathbb{E}_{\theta}^{n}[\widehat{Q}_{n}]
	\Big|, \\
	\overline{B}_{n}:= 
	s_{n}(H)(n-1)\Big|
	\widehat{\sigma}_{n}^{2} \kappa_{1}^{(1)}(\widehat{H}_{n}) n^{-1-2\widehat{H}_{n}}
	-\sigma^{2} \kappa_{1}^{(1)}(H) n^{-1-2H} 
	\Big| 
 \end{cases}
\end{align*}
Since we see that $\{\overline{V}_{n}\}_{n\in\mathbb{N}}$ is bounded in $\mathbb{P}^{n}_{\theta}$-probability, uniformly over $\Theta$ proceeding as for $\{V_{n}\}_{n\in\mathbb{N}}$ in the proof of Lemma~\ref{lemma:uncorrectedQV:tau}, it suffices to prove that $\{\overline{B}_{n}\}_{n\in\mathbb{N}}$ is bounded in $\mathbb{P}^{n}_{\theta}$-probability, uniformly over $\Theta$.
Using Lemma~\ref{lem:wavelet:kappa_p}, we know that $x\mapsto \kappa_{1}^{(1)}(x)$ is bounded and differentiable with bounded derivative on $[H_{-}, H_{+}]$. Therefore, the absolute value of the derivative of the function $x\mapsto \kappa_{1}^{(1)}(x)n^{-2x}$ is bounded by $Cn^{-2x}\log{n}$ for some uniform constant $C>0$
. Since we also have 
$\widehat{\sigma}_{n} \leq \sigma_+$ and $\widehat{H}_{n}\in[H_{-},H_{+}]$, we can show
\begin{align*}
	\overline{B}_{n}
	&\leq 
	s_{n}(H) \left|\widehat{\sigma}_{n}^{2}-\sigma^{2}\right| \kappa_{1}^{(1)}(H) n^{-2H}
	+
	s_{n}(H) \widehat{\sigma}_{n}^{2}
	\left|
	\kappa_{1}^{(1)}(H) n^{-2H}
	-
	\kappa_{1}^{(1)}(\widehat{H}_{n}) n^{-2\widehat{H}_{n}}
	\right|\\
	&
	\lesssim 
	s_{n}(H) \left(v_{n}(H)|\log{n}|^{-1}n^{2H}\right)^{-1}
	\left(
	v_{n}(H)|\log{n}|^{-1}|\widehat{\sigma}_{n}^{2}-\sigma^{2}|
	+v_{n}(H)|\widehat{H}_{n}-H|
	\right)
\end{align*}
uniformly on $\Theta$. Then we finish the proof since $s_{n}(H)\left(v_{n}(H)|\log{n}|^{-1}n^{2H}\right)^{-1}\leq 1$ and the sequences $\{v_{n}(H)|\widehat{H}_{n}-H|\}_{n\in\mathbb{N}}$ and $\{v_{n}(H)|\log{n}|^{-1}|\widehat{\sigma}_{n}-\sigma|\}_{n\in\mathbb{N}}$ are bounded in $\mathbb{P}^{n}_{\theta}$-probability, uniformly over $\Theta$. 

\subsection{Proof of Theorem~\ref{thm:rate_optimal_estimator}}
\label{Sec:thm:rate_optimal_estimator:proof}
Recall that $r_{n}(H)=n^{\frac{2K+1}{2\Diamond(H)}}$. We define
 \begin{align*}
	v_{n}^{(0)}(H) := n^{\frac{2K+3}{2(2K+5)}} \wedge n^{\frac{2K+3}{2K+5}(1-H)} \wedge n^{\frac{2}{2K+5}\{(2H)\wedge 1\}} 
	\text{ and } u_{n}^{(0)}(H) := n^{\frac{1}{2}} \wedge n^{2H}
\end{align*}
and by induction
\begin{align*}
	v_{n}^{(m)}(H) &:= 
	r_{n}(H)^{\frac{1}{2}} \wedge u_{n}^{(m-1)}(H)  \wedge \left( v_{n}^{(m-1)}(H) |\log{n}|^{-1} n^{\frac{\{(2H)\wedge 1\} 2H }{\Diamond(H)}} \right),\\
	u_{n}^{(m)}(H) &:= 
	n^{\frac{1}{2}} \wedge \left(v_{n}^{(m)}(H)|\log{n}|^{-1}n^{2H}\right)
\end{align*}
for each $m\in\mathbb{N}$. 
Then Lemmas~\ref{lemma:uncorrectedQV:H} and~\ref{lemma:uncorrectedQV:tau} and Proposition~\ref{prop:corrected:sigma} ensure that 
\begin{align*}
	\mbox{$\{v_{n}^{(0)}(H)(\widehat{H}^{(0)}_{n}-H)\}_{n\in\mathbb{N}}$, $\{v_{n}^{(0)}(H)|\log{n}|^{-1}(\widehat{\sigma}^{(0)}_{n}-\sigma)\}_{n\in\mathbb{N}}$ and $\{{u_{n}^{(0)}}(H)(\widehat{\tau}^{(m)}_{n}-\tau)\}_{n\in\mathbb{N}}$}
\end{align*}
are bounded in $\mathbb{P}^{n}_{\theta}$-probability, uniformly over $\Theta$. 
Moreover, Propositions~\ref{prop:corrected:H} and~\ref{prop:corrected:tau} yield that for each $m\in\mathbb{N}$, 
\begin{align*}
	\mbox{$\{v_{n}^{(m)}(H)(\widehat{H}^{(m)}_{n}-H)\}_{n\in\mathbb{N}}$, $\{v_{n}^{(m)}(H)|\log{n}|^{-1}(\widehat{\sigma}^{(m)}_{n}-\sigma)\}_{n\in\mathbb{N}}$ and $\{{u_{n}^{(m)}}(H)(\widehat{\tau}^{(m)}_{n}-\tau)\}_{n\in\mathbb{N}}$}
\end{align*}
are also bounded in $\mathbb{P}^{n}_{\theta}$-probability, uniformly over $\Theta$. 

Since $\Diamond(H)>1$ implies $u_{n}^{(m)}(H) \geq v_{n}^{(m)}(H)|\log{n}|^{-1}n^{\frac{\{(2H)\wedge 1\} 2H }{\Diamond(H)}}$ for each $m\in\mathbb{N}$, we conclude by induction that
\begin{align*}
	v_{n}^{(m)}(H) = 
	n^{\frac{2K+1}{2\Diamond(H)}} \wedge \left( v_{n}^{(0)}(H)|\log{n}|^{-m} n^{m \frac{\{(2H)\wedge 1\} 2H }{\Diamond(H)}} \right).
\end{align*}
Then the conclusion follows because $m_{\mathrm{opt}} \frac{\{(2H)\wedge 1\} 2H }{\Diamond(H)} > \frac{2K+1}{2\Diamond(H)}$ holds for any $H \in [H_{-}, H_{+}]$ from the definition of $m_{\mathrm{opt}}$ so that we obtain $v_{n}^{(m_{\mathrm{opt}})}(H)=r_{n}(H)^{\frac{1}{2}}$ and $u_{n}^{(m_{\mathrm{opt}})}(H)=n^{\frac{1}{2}}$ due to  $n^{\frac{1}{2}}\leq v_{n}^{(m_{\mathrm{opt}})}(H)|\log{n}|n^{2H}$.

\section{Proof of Theorem~\ref{Theorem:OneStep} and Lemma~\ref{Lemma:OneStep}}
\label{Sec:Proof-OneStep}

\subsection{Proof of Theorem~\ref{Theorem:OneStep}} 
Using Lemma~\ref{Lemma_RateMat}, we can assume without loss of generality that $\Phi_{n}(\theta)$ is invertible on $\Theta$. 
For conciseness, we write $\widetilde{u}_{n}(\theta):=\Phi_{n}(\theta)^{-1}(\widetilde{\theta}_{n}-\theta)$, $\widetilde{\Psi}_{n}(\theta):=
\Phi_{n}(\widetilde{\theta}_{n})^{-1}\Phi_{n}(\theta)$ and 
\begin{equation*}
	\overline{v}_{n}(u,\theta)
	:=\int_{0}^{1}(1-z)\Phi_{n}(\theta)^{\top}\left[\partial_{\theta}^{2}\ell_{n}\left(\theta+z\varphi_{n}(\theta)u\right)
	-\partial_{\theta}^{2}\ell_{n}(\theta)\right]\Phi_{n}(\theta)\,\mathrm{d}z.
\end{equation*}
First note that we have
\begin{equation*}
	\Phi_{n}(\theta)^{-1}(\widehat{\theta}_{n}-\theta)
	=
	\widetilde{u}_{n}(\theta)
	+\left(\widetilde{\Psi}_{n}(\theta)^{\top}\mathcal{I}(\widetilde{\theta}_{n})\widetilde{\Psi}_{n}(\theta)\right)^{-1}\Phi_{n}(\theta)^{\top}\partial_{\theta}\ell_{n}(\widetilde{\theta}_{n}).
\end{equation*}
Using Taylor's theorem, we have
\begin{equation*}
	\Phi_{n}(\theta)^{\top}\partial_{\theta}\ell_{n}(\widetilde{\theta}_{n})
	=\zeta_{n}(\theta)-\mathcal{I}_{n}(\theta)\widetilde{u}_{n}(\theta)
	+\overline{v}_{n}(\widetilde{u}_{n}(\theta),\theta)\widetilde{u}_{n}(\theta).
\end{equation*}
Then, we obtain 
\begin{align}
\begin{split}
	\Phi_{n}(\theta)^{-1}(\widehat{\theta}_{n}-\theta)
	&=\left(\widetilde{\Psi}_{n}(\theta)^{\top}\mathcal{I}(\widetilde{\theta}_{n})\widetilde{\Psi}_{n}(\theta)\right)^{-1}\zeta_{n}(\theta) \\
	&\quad+\left[ I_{3} -\left(\widetilde{\Psi}_{n}(\theta)^{\top}\mathcal{I}(\widetilde{\theta}_{n})\widetilde{\Psi}_{n}(\theta)\right)^{-1}\mathcal{I}_{n}(\theta)\right]\widetilde{u}_{n}(\theta) \\
	&\quad+\left(\widetilde{\Psi}_{n}(\theta)^{\top}\mathcal{I}(\widetilde{\theta}_{n})\widetilde{\Psi}_{n}(\theta)\right)^{-1}
	\overline{v}_{n}(\widetilde{u}_{n}(\theta),\theta)\widetilde{u}_{n}(\theta). 
 \end{split}
 \label{Theorem:OneStep_Key1}
\end{align}
Using \eqref{Assumption_IniEstimator} and the continuity of $\mathcal{I}(\theta)$ on $(0,1)\times(0,\infty)^{2}$, we know that  $\mathcal{I}(\widetilde{\theta}_{n}) \to \mathcal{I}(\theta)$ in distribtion as $n\to\infty$ and using also Lemma~\ref{Lemma_RateMat}, $\widetilde{\Psi}_{n}(\theta) \to I_{3}$ in distribtion as $n\to\infty$. Moreover, proceeding as the proof of \eqref{LAN:remainder}, we can show
\begin{equation*}
	\sup_{\|u\|_{\mathbb{R}^{3}}\leq c}\left\|\overline{v}_{n}(u,\theta)\right\|_{\mathrm{F}}=o_{\mathbb{P}^{n}_{\theta}}(1)\ \ \mbox{as $n\to\infty$}
\end{equation*}
for each $c>0$, so that \eqref{Assumption_IniEstimator} gives
\begin{equation*}
	\left\|\overline{v}_{n}(\widetilde{u}_{n}(\theta),\theta)\right\|_{\mathrm{F}}\stackrel{\mathbb{P}^{n}_{\theta}}{\rightarrow}0\ \ \mbox{as $n\to\infty$.}
\end{equation*}
We conclude by plugging the last limit, $\mathcal{I}(\widetilde{\theta}_{n}) \to \mathcal{I}(\theta)$ and $\widetilde{\Psi}_{n}(\theta) \to I_{3}$, and \eqref{LAN:Score-CLT}, into \eqref{Theorem:OneStep_Key1} (using Slutsky's lemma).

\subsection{Proof of Lemma~\ref{Lemma:OneStep}}\label{Sec:Proof-Tightness_Initial-Estimator}
Write $\widetilde{\xi}_{n}:=(\widetilde{H}_{n},\widetilde{\sigma}_{n})$. 
Using the notation of Section~\ref{Appendix:Rate-Matrix}, we have
\begin{equation*}
	\sqrt{r_{n}^{1}(H)}\varphi_{n}(\theta)^{-1}(\widetilde{\xi}_{n}-\xi)
	=
	\widetilde{\varphi}_{n}(\theta)^{-1}
	\begin{pmatrix}
		\sqrt{r_{n}^{1}(H)}(\widetilde{H}_{n}-H)\\
		2\sqrt{r_{n}^{1}(H)}a_{n}(H)(\widetilde{H}_{n}-H)
		+2\sigma^{-1}\sqrt{r_{n}^{1}(H)}(\widetilde{\sigma}_{n}-\sigma)
	\end{pmatrix}.
\end{equation*}
Since we have $\|\widetilde{\varphi}_{n}(\theta)\|_{\mathrm{F}}\lesssim 1$ and $\|\widetilde{\varphi}_{n}(\theta)^{-1}\|_{\mathrm{F}}\lesssim 1$ uniformly on $\Theta$, the conclusion follows.

\end{document}